\documentclass[12pt,a4paper,titlepage]{report}
\usepackage[utf8]{inputenc}
\usepackage[style=numeric,backref]{biblatex}%
\renewbibmacro{in:}{\ifentrytype{incollection}{%
\printtext{\bibstring{in}\intitlepunct}}{%
\ifentrytype{inproceedings}{%
\printtext{\bibstring{in}\intitlepunct}}{}%
}}
\renewbibmacro*{volume+number+eid}{\printfield{volume}%
\setunit*{\addnbthinspace}\printfield{number}%
\setunit{\addcomma\space}\printfield{eid}%
}
\DeclareFieldFormat[article]{number}{\mkbibparens{#1}}
\DefineBibliographyStrings{english}{%
	bibliography = {References},
}
\usepackage{pythonhighlight}

\newenvironment{closealign}{\par\centering$\displaystyle\begin{aligned}}{\end{aligned}$\par}
\usepackage{amsmath,amsfonts,amssymb,amsthm,graphicx,dsfont}
\usepackage{bbm,thmtools,hyperref,xcolor,enumitem}
\usepackage[a4paper]{geometry}
\DeclareSymbolFont{bbold}{U}{bbold}{m}{n}
\DeclareSymbolFontAlphabet{\mathbbold}{bbold}
\DeclareRobustCommand{\stirling}{\genfrac\{\}{0pt}{}}
\newtheorem*{theorem*}{Theorem}
\newtheorem{theorem}{Theorem}[section]
\newtheorem{claim}[theorem]{Claim}
\newtheorem{lemma}[theorem]{Lemma}
\newtheorem{corollary}[theorem]{Corollary}
\theoremstyle{definition}
\newtheorem*{definition}{Definition}
\newtheorem*{convention}{Convention}
\newtheorem*{notation}{Notation}
\newtheorem{remark}{Remark}
\numberwithin{remark}{chapter}

\newtheorem*{remark*}{Remark}
\newtheorem*{lemma*}{Lemma}
\newtheorem*{corollary*}{Corollary}

\usepackage{parskip,mathrsfs,xtab}

\usepackage{setspace}
\usepackage[nodayofweek]{datetime}

\newdateformat{monthyear}{\monthname[\THEMONTH], \THEYEAR}
\hypersetup{
	pdftitle={Moments of Representation Numbers},
	pdfsubject={math.NT},
	pdfauthor={Naomi Bazlov},
	pdfkeywords={Analytic number theory, Representation numbers},
}
\title{Moments of Representation Numbers}
\author{Naomi Bazlov}
\begin{document}
\maketitle
\begin{abstract}
	A representation number is a function which expresses the number of ways an integer can be written as a sum of elements of chosen sets. One of the oldest number-theoretic results on representation numbers is Fermat's theorem which says that an odd prime can be written as a sum of two squares in exactly $0$ or $2$ ways (if order of summands is important).
	
	In this dissertation, we discuss a selection of methods from modern analytic number theory and apply them to study asymptotics of certain representation numbers. In particular, we work through an argument in the recent paper~\cite{granville2023} by Granville, Sabuncu and Sedunova to obtain upper bounds on higher moments of the number of ways of writing $n$ as the sum of a square and a square of a prime. We then use the method from the paper to obtain new results on moments of representation numbers where ``prime'' is replaced by ``sum of two squares''.
\end{abstract}
\tableofcontents
\section*{Acknowledgments}
I would like to thank Dr Simon Rydin Myerson for supervising me in this project, suggesting a problem to investigate and meeting with me regularly to discuss my ideas and provide help and advice. His guidance has been inspiring and crucial to my research endeavours. I am also thankful to Dr Ofir Gorodetsky and Dr Alisa Sedunova for their advice and suggestions.
I am grateful to my family for their steadfast support.

\chapter{Introduction}
\section{Overview}
This dissertation grew out of my attempt to understand the very recent results
by Andrew Granville, Cihan Sabuncu and Alisa Sedunova on moments of
representation numbers, given in their paper~\cite{granville2023}. 
Part of the dissertation is devoted the moments of $r_1$, the function which counts 
the number of ways to write $n$ as a sum of a square and a square of a prime.
My treatment of these moments involves an array of analytic number theory tools: the Prime Number Theorem, Mertens’ theorems, versions of the above for primes in arithmetic
progressions, sieve theory and a Tauberian theorem. In particular, I deduce asymptotics 
for the first moment of $r_1$ using a new method, which combines~\cite{granville2023} with Stephan Daniel's earlier work~\cite{daniel_2001}.
Further to that, I work through the 
Granville-Sabuncu-Sedunova approach to estimating higher moments of $r_1$,
and then modify their method to attack a new problem about representation numbers $r_S$. 
My work on the moments of $r_S$ contains new results.

For centuries, mathematicians have studied representation numbers -- functions which count the number of ways an integer can be represented in a certain form. A famous example is the Gauss circle problem: 
one wishes to estimate how many integer lattice points are contained within a circle of radius $R$. Carl Gauss was the first to suggest asymptotics for the answer in his 1837 paper~\cite{Gauss_circle}, and subsequently much gradual progress on estimating the error term was made in the 20th century by Sierpinski, van der Corput, Vinogradov, Hardy and others. Note that this problem is about the first moment of the function $r_0(n)$ which counts the number of ways $n$ can be expressed as a sum of two squares.

In this dissertation, my aim is to study the behaviour of functions that
count the number of ways $n$ can be represented as a sum of two squares subject to additional number-theoretic constraints. I am interested
in the asymptotics of moments of these functions, where the k-th moment of an arithmetic function $ f$ is $m_k(x)=\sum_{n\le x}f^k(n)$. An example is the functions $r_0$ , $r_1$ and $r_2$ where $r_i(n)$ is defined as the number of ways to write n as the sum of two
squares with i of them being squares of primes. I work towards and
prove a result from~\cite{granville2023}, which is Theorem~\ref{thm:r1cl upper bound} here:
\begin{theorem*}[Granville, Sabuncu, Sedunova, 2023]
	For any fixed integer $\ell\ge 1$ and integer $k\ll\log\log x =: L$ we have
	\[\sum_{{\substack{n\le x\\\omega^*(n)=k}}}\binom{r_1(n)}{\ell}\ll_\ell\frac{xL^{O_\ell(1)}}{(\log x)^{\ell+1}}\frac{(2^{\ell-1}L)^k}{k!}.\]
\end{theorem*}
Subsequently, I look to extend the methods involved to a  wider class of representation numbers. In particular, Theorems~\ref{moments of rR'} and~\ref{moments of rR} are new results about the moments of $r_S$, the number of ways to write $n$ as a sum of a square and
the square of two squares.
\section{Structure of this preprint}
In Chapter $2$, I start by introducing some notation used throughout the dissertation, and provide various properties of important arithmetic functions such the prime counting function $\pi$, the von Mangoldt function $\Lambda$ and the divisor function~$\tau$.

The analogue of $\pi$ restricted to an arithmetic progression is also defined in Chapter $2$, but its properties will be introduced later (see Mertens' Theorem for primes in arithmetic progressions in Chapter $6$.) Its behaviour is linked with equidistribution of primes among invertible residue classes modulo $q$, as attested by Dirichlet's theorem. More deeply, this behaviour of primes extends to number fields, reflected in results such as Chebotarev's density theorem; I do not pursue this direction in the dissertation, leaving it as a possible avenue for future work.

In Chapter $3$, I introduce representation numbers and provide a collection of results which can be understood before I introduce heavier machinery such as sieves. I work out the asymptotics of the Gauss circle problem, deducing the main term with the help of Gaussian primes and the non-principal Dirichlet character mod $4$, but do not prove a strong form of the error term, quoting Huxley's work~\cite{huxleyexpsums} instead. I then look at first moments of $r_1$ and $r_2$. Drawing inspiration from proofs in~\cite{daniel_2001} and~\cite{granville2023}, I provide a proof method which works for both functions. My method gives a better error term for the first moment of $r_2$ than the one currently known and stated in the 2024 paper~\cite{Sabuncu_2024}.

I also consider the ``zeroth moment'' of $r_i$ in Chapter $3$; the zeroth moment of a non-negative $f$ is defined as $M(x)=\#\{n\le x: f(n)>0\}$. Landau~\cite{landaugerman} studied the zeroth moment of $r_0$, and Stephan Daniel includes a discussion of zeroth moments in his paper~\cite{daniel_2001}. In recent years, moments of $r_2$ have been studied in depth by Sabuncu, and I conclude the chapter with some results of interest from his 2024 paper~\cite{Sabuncu_2024}.

The representation number that I have found most interesting to study is $r_1$, as this allowed me to explore and combine modern analytic number theory techniques. Chapters $4$ and $5$ build the necessary background in order to prove Theorem~\ref{thm:r1cl upper bound}, which gives asymptotic bounds on higher moments of $r_1$.

Chapter $4$ focuses on sieve theory, a powerful 20th century analytic tool inspired by Eratosthenes' idea from Ancient Greece. Given a set of integers and a collection of primes, one removes all multiples of these primes from the given set. What is a good estimate on the new size of the set? If there are reasons to assume that being divisible by a prime $p$ are (almost) independent events for distinct $p$, sieve theory can answer this and similar questions effectively. This is used to construct a key bounding argument in the proof of Theorem~\ref{thm:r1cl upper bound}. I referred to many sources, such as~\cite{Fordsieves},~\cite{heathbrown2002lecturessieves} and~\cite{richert1976lectures}, to get an understanding of sieve theory and some important results, and provide an introduction to the topic in this chapter.

Chapter $5$ discusses properties of functions similar to $r_1$, which will help us to approximate it. The results I discuss are stated in~\cite{granville2023}, although the proofs are sometimes omitted or incorrect -- I have provided my own proofs.
Chapter $6$ is focused on proving Theorem~\ref{thm:r1cl upper bound}: I follow the method of Granville, Sabuncu and Sedunova to obtain upper bounds on higher moments of $r_1$, filling in the details of the proof and adding corrections where appropriate.

In Chapter $7$, I study the representation numbers $r_{S}$, where $r_{S}(n)$ is the number of ways to write $n$ as the sum of a square and the square of an element of $S$. In particular, I look at an `intermediate' case between $r_{\mathbb{N}}=r_0$ and $r_1$, namely $r_{\mathcal{R}}$, where $\mathcal{R}$ is the set of numbers expressible as a sum of two squares; I also consider $r_{\mathcal{R}'}$, where $\mathcal{R}'$ is the set of numbers expressible as a sum of two coprime squares. I find the behaviour of their first moment in a similar way to that of the first moment of $r_1$, expressing them through the zeroth moments of $r_0$ and $r_0^*$ respectively. The zeroth moment of $r_0$ was introduced briefly in Chapter $3$, and in Chapter $7$, I provide a proof for asymptotics of the zeroth moment of $r_0^*$. The proof is inspired by Landau's methods in~\cite{landaugerman} in combination with a more modern Tauberian theorem relating Dirichlet series to the sums of their coefficients.

Subsequently, I use methods from Theorem~\ref{thm:r1cl upper bound} to obtain an upper bound on the higher moments of $r_{\mathcal{R}}$ and $r_{\mathcal{R}'}$. The number of sums of two squares up to $x$ is $\frac{x}{\sqrt{\log x}}$ compared to the number of primes up to $x$ which is $\frac{x}{\log x}$. This relationship is reflected in the heuristics of the bounds I obtain, where the $\log x$ terms in denominators of bounds on moments of $r_1$ become replaced with $\sqrt{\log x}$ once I switch to considering~$r_{\mathcal{R}}$. This is shown in Theorem~\ref{moments of rR}, which I state presently:
\begin{theorem*}[Theorem~\ref{moments of rR}]
	For any fixed integer $\ell\ge 1$ and integer $k\ll L$ we have
	\[\smash{\sum_{{\substack{n\le x\\\omega^*(n)=k}}}\binom{r_{\mathcal{R}}^*(n)}{\ell}\ll_\ell\frac{xL^{O_\ell(1)}}{(\sqrt{\log x})^{\ell+1}}\frac{(2^{\ell-1}L)^k}{k!}.}\]
\end{theorem*}
\chapter{Foundations}\vspace{-2ex}
\section{Definitions}
In this section we provide key definitions. The definitions of well-known functions, such as Chebyshev functions and the prime counting function, are as according to convention -- they can be found in many textbooks; in this document we follow~\cite{hardy75}.\vspace{-4ex}
\begin{convention}
	Throughout the dissertation, if the letter $p$ appears as a summation variable, it will always run over primes.
\end{convention}
	\begin{definition}[$O_a,\ll,\asymp,\sim$]
		Let $f,g:\mathbb{R}\to\mathbb{R}$. Then:
	\begin{itemize}
		\item Given a parameter $a\in\mathbb{R}$, we say $f(x,a)=O_a(g(x,a))$ if $\exists c(a)>0$ such that $\forall x>x_0\colon |f(x,a)|\le c(a)|g(x,a)|$.
		\item We say $f(x)\ll g(x)$ to mean $f(x)=O(g(x))$. Given a parameter $a\in\mathbb{R}$, we write $f(x)\ll_a g(x)$ to mean $f(x)=O_a(g(x))$. Dependence on $a$ is implied.
		\item We say $f(x)\asymp g(x)$ to mean $f(x)\ll g(x)\ll f(x	)$.
		\item We say $f(x)\sim g(x)$ to mean $\lim_{x\to\infty}\frac{f(x)}{g(x)}=1$.
	\end{itemize}
\end{definition}
	\begin{definition} (Representation numbers $r_i$) 
		
		The \textbf{representation numbers} $r_0, r_1, r_2:\mathbb{Z}_{\ge 1}\to\mathbb{Z}_{\ge 0}$ are defined as follows:
		\begin{itemize}
		\item  $r_0(n)=\#\{(a,b)\in\mathbb{Z}_{\ge 0}\times\mathbb{Z}_{\ge 1}: a^2+b^2=n\}$
		\item  $r_1(n)=\#\{(a,p)\in\mathbb{Z}_{\ge 0}\times\mathbb{Z}_{\ge 1}: a^2+p^2=n,\, p\text{ is prime}\}$
		\item  $r_2(n)=\#\{(p,q)\in\mathbb{Z}_{\ge 0}\times\mathbb{Z}_{\ge 1}: p^2+q^2=n,\, p,q\text{ are prime}\}$
		\end{itemize}
	In particular, $r_0(1)=1$.
	\end{definition}
	\begin{definition}(Arithmetic functions,~{convolution})
		\begin{itemize}
			\item An \textbf{arithmetic function} is a function $f:\mathbb{N}\to\mathbb{C}$; it is \textbf{multiplicative} if for all coprime $a,b\in\mathbb{N}$ one has $f(ab)=f(a)f(b)$ and \textbf{totally multiplicative} if $f(ab)=f(a)f(b)$ holds for \textit{all} $a,b\in\mathbb{N}$.
			\item The \textbf{convolution} $f*g$ of two arithmetic functions $f$ and $g$ is defined by ${(f*g)(n)=\sum_{d\mid n}f(d)g(\frac nd)}$, where $d$ runs over all divisors of $n$. Note that if $f$ and $g$ are multiplicative arithmetic functions, then $f*g$ is also multiplicative.
		\end{itemize}
	\end{definition}
	\begin{definition}(Mobius function, von Mangoldt function, divisor functions)
		\begin{itemize}
			\item The \textbf{Mobius function} $\mu$ is nonzero on squarefree integers only, with $\mu(n)=(-1)^k$ for 
			$n=p_1\cdot\ldots\cdot p_k$ with $p_i$ distinct primes.
			\item \textbf{von Mangoldt's function} $\Lambda$ is nonzero on prime powers only, with
				$\Lambda(n)=\log p$ if $n=p^k$ for some prime $p$.
		\item The \textbf{prime divisor function} is defined by $\omega(x):=\#\{p\text{ prime}:p\mid x\}$.
		\item The \textbf{divisor counting function} is defined by
		$\displaystyle\tau(n)=\smash{\sum_{d\mid n}1}$.	
	\end{itemize}
	\end{definition}
	\begin{definition} (Chebyshev functions, prime counting function)
		
		Chebyshev's theta and psi function are defined as follows:
		\begin{itemize}
			\item $\displaystyle\theta(x)=\sum_{p\le x}\log p$ (reminder: we sum over primes $p$.)
			\item $\displaystyle\psi(x)=\sum_{p^k\le x}\log p=\sum_{n\le x}\Lambda(n)$
		\end{itemize}
		The \textbf{prime counting function} is defined $\pi(x):=\#\{p\le x: p\text{ prime}\}$. Its analogue in arithmetic progressions is
		\[\pi(x; q, a)= \!\!\!\!\!\sum_{\substack{p\le x,\ p\equiv a\pmod q}}\!\!\!\!\!\!\!\!\!\!1 = \#\{p\le x: p\equiv a \mod q\}.\]
	\end{definition}
	\begin{notation}($L$)
		We write $L$ or $L(x)$ to mean $\log\log x$.
	\end{notation}
	\begin{notation} $P(n)$ denotes the largest prime factor of $n$.
	\end{notation}
\section{Basic Results}
In this section we present some properties of arithmetic functions which will be used in the proofs of specialised results later on. These proofs are standard and can be found in many textbooks, including Hardy's book~\cite{hardy75} which we reference here.

Note that the set of all arithmetic functions is a ring with respect to pointwise addition and Dirichlet's convolution $*$. The arithmetic function
$I(n)=\mathbbm{1}_{\{n=1\}}$, taking value $1$ at $1$ and $0$ elsewhere, is easily seen to be the multiplicative identity of this ring. It is well known that
	$\mu*\mathbbm{1} = I$, or in other words, that $\mu$ is the multiplicative inverse of the constant function $\mathbbm{1}$ in this ring.
From this we obtain Mobius inversion: $g = f*\mathbbm{1}\iff f = \mu*g$. It immediately follows from commutativity and associativity of the unital ring operation $*$ and using $\mu*\mathbbm{1} = I$.
\begin{lemma}\label{lambda}
	$\log n=\sum_{k\mid n}\Lambda(k)$.
\end{lemma}
\begin{proof}
	Let $n=\prod_{i=1}^{s}p_i^{\alpha_i}$ be the prime factorisation of $n$. The only factors of $n$ which give nonzero contributions to $\sum_{k\mid n}\Lambda(k)$ are pure powers of primes, so \[\smash{\sum_{k\mid n}\Lambda(k) = \sum_{i=1}^s\sum_{j=1}^{\alpha_i}\Lambda(p_i^{j})=\sum_{i=1}^s\sum_{j=1}^{\alpha_i}\log p_i=\sum_{i=1}^s\alpha_i\log p_i=\log\bigl(\prod_{i=1}^{s}p_i^{\alpha_i}\bigr)=\log n.}\qedhere\]
\end{proof}
\begin{lemma}\label{divisor-epsilon}
	For any $\epsilon>0$, we have $\tau(n)\ll_\epsilon n^\epsilon$.
\end{lemma}
\begin{proof} Refer to the second proof of Theorem 315 in~\cite{hardy75}.

Let $\epsilon>0$ be arbitary, and let $n=p_1^{\alpha_1}\dots p_k^{\alpha_k}$. We can see that $\tau(n)=\prod_{i=1}^{k}(\alpha_i+1)$. Hence
\begin{equation*}
	\smash[t]{\frac{\tau(n)}{n^\epsilon}=\prod_{i=1}^{k}\frac{\alpha_i+1}{p_i^{\epsilon\alpha_i}}.}
\end{equation*}
Now, for a fixed $p_i$, consider the term $(\alpha_i+1)p_i^{-\epsilon\alpha_i}$. We have two cases, depending on if $p<2^{\frac1\epsilon}$ or not. If $p\ge2^{\frac1\epsilon}$ then $\displaystyle\dfrac{\alpha_i+1}{p_i^{\epsilon\alpha_i}}\le\frac{\alpha_i+1}{2^{\alpha_i}}\le 1$.

In the remaining case, we will make use of the fact that $p_i^{\epsilon\alpha_i}=\exp(\epsilon\alpha_i\log p_i)\ge \epsilon\alpha_i\log p_i\ge \epsilon\alpha_i\log 2$. We know that $e^x\ge x$ by looking at the Taylor series, or derivative, or other methods; in fact, $e^x\ge 1+x$ for $x\ge 0$. Hence
\[\frac{\alpha_i+1}{p_i^{\epsilon\alpha_i}}\le\frac{\alpha_i+1}{\epsilon\alpha_i\log 2}\le 1+\frac1{\epsilon\log 2}\le\exp\Bigl(\frac1{\epsilon\log 2}\Bigr).\]

Thus \begin{equation}\label{tauin}
\smash[t]{\frac{\tau(n)}{n^\epsilon}=\prod_{i=1}^{k}\frac{\alpha_i+1}{p_i^{\epsilon\alpha_i}}\le\prod_{p\le 2^\frac1\epsilon}\exp\Bigl(\frac1{\epsilon\log 2}\Bigr)\le\exp\Bigl(\frac{2^\frac1\epsilon}{\epsilon\log 2}\Bigr)=O(1).}
\end{equation}
Then $\tau(n)=O(n^\epsilon)$, which is equivalent to our original statement.
\end{proof}
\begin{lemma}\label{lemma: tau o(1)}
	$\forall\epsilon>0:\exists n_0$ s.t. $\log(\tau(n))\le\dfrac{(1+\epsilon)\log 2\log n}{\log\log n}$ for any $n\ge n_0$.
\end{lemma}
\begin{proof}
	We improve the bound on $\tau(n)$ using Lemma~\ref{divisor-epsilon}, following the method in~\cite{hardy75}, and filling in the missing steps. We replace $\epsilon$ with the expression $\alpha=\dfrac{(1+\frac12\epsilon)\log 2}{\log\log n}$. In the above proof, we did not use that $\epsilon$ was a constant until the very last step of asserting $\exp\bigl(\frac{2^{1/\epsilon}}{\epsilon\log 2}\bigr)=O(1)$. Hence, as in~\eqref{tauin}, we have $\frac{\tau(n)}{n^\alpha}\le \exp\bigl(\frac{2^{1/\alpha}}{\alpha\log 2}\bigr)$. Note that $\frac1\alpha = \frac{\log\log n}{(1+\frac12\epsilon)\log 2}$, so $2^{\frac1\alpha}=\exp\bigl(\frac{\log\log n}{(1+\frac12\epsilon)}\bigr)=(\log n)^{1/(1+\frac12\epsilon)}$. Thus
	\[\log\Bigl(\frac{\tau(n)}{n^\alpha}\Bigr)\le \frac{2^\frac1\alpha}{\alpha\log 2}=\frac{(\log n)^{1/(1+\frac12\epsilon)}\log\log n}{(1+\frac12\epsilon)(\log 2)^2}.\]
	As $n\to\infty$, we have $\log\log n=o((\log n)^\delta)$ for any $\delta$, i.e. $\log\log n$ is small compared to any power of $\log n$. In fact, this holds for all powers of $\log\log n$, so for large enough $n$ ($n\ge n_0(\epsilon)$, some constant $n_0(\epsilon)$) we can write \[\log\Bigl(\frac{\tau(n)}{n^\alpha}\Bigr)\le\frac{(\log n)^{1/(1+\frac12\epsilon)}\log\log n}{(1+\frac12\epsilon)(\log 2)^2}\le\frac{\epsilon\log 2\log n}{2\log\log n}.\]
	But $\log\bigl(\frac{\tau(n)}{n^\alpha}\bigr)=\log(\tau(n))-\alpha\log n$. Hence \[\log(\tau(n))\le\frac{\frac12\epsilon\log 2\log n}{\log\log n}+\frac{(1+\frac12\epsilon)\log 2\log n}{\log\log n}=\frac{(1+\epsilon)\log 2\log n}{\log\log n}.\qedhere\]
\end{proof}
In particular, either of the above results shows that $\frac{\log(\tau(n))}{\log n}\to 0$ as $n\to\infty$, so $\frac{\log(\tau(n))}{\log n}=o(1)$. Since we can write $\tau(n)=n^{\log(\tau(n))/\log n}$, we have
\begin{equation}\label{eq: tau o(1)}
	\tau(n)=n^{o(1)}.
\end{equation}
We give the following result for our future reference:
\begin{theorem}\label{estimate on omega}
	$\omega(n)\ll\dfrac{\log n}{\log\log n}$.
\end{theorem}
\begin{proof}
	By exponentiating both sides in Lemma~\ref{lemma: tau o(1)}, we see that $\forall\epsilon>0$ we eventually have $\tau(n)\le 2^{(1+\epsilon)\log n/\log\log n}$. Moreover, by considering the prime factorisation $n=p_1^{\alpha_1}\dots p_k^{\alpha_k}$, where $\forall i\colon\alpha_i\ge 1$, we compare $k=\omega(n)$ and $\prod_{i=1}^k(\alpha_i+1)\ge 2^k$ to get $2^{\omega(n)}\ll\tau(n)$. Hence, picking a particular constant $\epsilon$ (e.g. $\epsilon=1$) we get $2^{\omega(n)}\ll2^{2\log n/\log\log n}\implies \omega(n)\ll\frac{2\log n}{\log\log n}\ll\frac{\log n}{\log\log n}$, as required.
\end{proof}
\section{PNT with error term}
The asymptotic behaviour of the prime counting function is well-known:
$\pi(x)\sim\frac{x}{\log x}$. Given a zero-free region for the Riemann zeta function, we can find an error term for the prime counting function $\pi$. Standard techniques such as Perron's Formula and complex contours are often used in the proof. Perron's Formula relates Chebyshev's function $\psi$ to an integral of $\frac{-\zeta'(s)}{\zeta(s)}\cdot\frac{x^s}{s}$, and we then change the contour of integration using our knowledge of the poles of $\frac{\zeta'}{\zeta}$ and complex analysis. The stronger the known zero-free region, the better the error term; thus, proving the Riemann Hypothesis would greatly improve the error term estimate for $\pi$, as the zero-free region would extend all the way to the line $\Re(s)=\frac12$. Unfortunately, even the strongest known zero-free regions are nowhere near to this.

The following result is equivalent to the Prime Number Theorem with weak error term $O(x\exp(-c\sqrt[10]{\log x}))$.
\begin{lemma*}
	$\psi(x) = x + O(x\exp(-c\sqrt[10]{\log x}))$, for some real $c>0$, and $x\ge 2$.
\end{lemma*}
The classical zero-free region for $\zeta$ is $\sigma\ge 1-\frac{c}{\log(|t|+2)}$, giving the following result:
\begin{lemma*}
	$\psi(x) = x + O(x\exp(-c\sqrt{\log x}))$, for some real $c>0$.
\end{lemma*}
\begin{corollary*}[Prime Number Theorem with classical error term]\label{lemma: pi = x/logx}
	\[\pi(x) = \dfrac{x}{\log x} + O(x\exp(-c\sqrt{\log x})), \text{ for some real } c>0.\]
\end{corollary*}
\begin{remark*}
	Although the logarithmic integral $\text{Li}(x)$ is the same as $\frac{x}{\log x}$ modulo the error term we can prove without assuming results such as the Riemann Hypothesis, in practice we can observe that $\pi(x)$ is closer to the behaviour of $\text{Li}(x)$ than $\frac{x}{\log x}$, which motivates its use in the following.
\end{remark*} 
\begin{lemma*}
	$\pi(x) = \mathrm{Li}(x) + O(x\exp(-c\sqrt{\log x}))$, for some real $c>0$, where the logarithmic integral $\mathrm{Li}(x)$ is defined as $\displaystyle \int_2^x\frac{\mathrm{d}t}{\log t}$.
\end{lemma*}
This can be done as $\text{Li}(x)-\frac{x}{\log x}$ is small, in fact of order $\frac{x}{\log^2 x}$ -- smaller than the error term in the result.

\chapter{Representation numbers}\vspace{-2ex}
In this chapter we provide a collection of results about moments of representation numbers. We will look at the moments of $r_0$ and $r_2$, including zeroth, first and higher moments. We will also see the first and zeroth moments of $r_1$. In later chapters, we will focus more on the behaviour of higher moments of $r_1$.
%
%
%
\section{Moments of $r_0$}
The representation number $r_0$ has been widely studied, for example in relation to the Gauss circle problem which counts the number of lattice points in a circle of given radius; this relates to the first moment of the function. Since $r_0(n)$ is the number of ways to represent $n$ as $x^2+y^2$, its moments can also be studied as part of the topic of quadratic forms. The Gauss circle problem has many generalisations in modern mathematics in areas including number theory and ergodic theory.
\begin{lemma}[First moment of $r_0$]\label{r0asymp}
	$\displaystyle\smash{\sum_{n\le x}r_0(n)\sim\frac{\pi}4x.}$
\end{lemma}
Note that \[\sum_{n\le x}r_0(n)=\sum_{\substack{m,k\le \sqrt{x}:\\ m^2+k^2\le x\\0\le m,1\le k}}1\asymp\frac14\#\{\text{Lattice points in a circle of radius }\sqrt{x}\}.\] (It is a quarter of the number of nonzero lattice points.)

Here, $r_0(n)=\frac14\#\{\text{Lattice points on a circle of radius }\sqrt{n}\}$ for all $n\ge 1$; we start by finding an expression for $r_0(n)$. We will think of all circles as circles drawn in the complex plane $\mathbb{C}$. The lattice points of $\mathbb{C}$ are $\mathbb{Z}[i]$, the Gaussian integers.
	
Consider prime factorisation in $\mathbb{Z}[i]$. Primes that are $1\pmod 4$ factor into exactly $2$ Gaussian primes. 
Primes that are $3\pmod 4$ are prime in $\mathbb{Z}[i]$ as well. ($2$ is special as it is a square up to associates.)

Let the prime factorisation of $n$ over the integers be $n=p_1^{a_1}p_2^{a_2}\dots p_k^{a_k}$. If $p_j\equiv 1\pmod 4$, then let $A_j$ and $B_j$ be the complex conjugate Gaussian primes whose product is $p_j$; if $p_j\equiv 3\pmod 4$, then $p_j$ is a Gaussian prime; if $p_j = 2$, then $p_j = (1+i)(1-i)$.

Suppose $\alpha + \beta i\in\mathbb{Z}_{\ge 0}[i]$ lies on the circle of radius $\sqrt{n}$. Then $(\alpha + \beta i)(\alpha - \beta i)=n$, so $\alpha + \beta i$ is a product of some Gaussian prime factors of $n$, and $\alpha - \beta i$ is the product of their complex conjugates. So to count the number of lattice points on the circle, we count the number of ways to split the prime factors of $n$ into conjugate pairs.

If $p_j=A_jB_j\equiv 1\pmod 4$, there are $a_j+1$ choices for the power of $A_j$ dividing $\alpha + \beta i$. If $p_j\equiv 3\pmod 4$, and $a_j$ is odd, then there are actually no lattice points; if $a_j$ is even, then $p_j^{a_j/2}$ divides $\alpha + \beta i$. Finally, if $p_j=2$, then $(1+i)^{a_j}$ divides $\alpha + \beta i$.

Now, consider the non-principal Dirichlet character modulo $4$, called $\chi$. It is periodic modulo $4$, with $\chi(2k)=0, \chi(4k\pm 1)=\pm 1$. Then $r_0(n)=\prod_{i=1}^k(\chi(1)+\chi(p_i)+\dots+\chi(p_i^{a_i}))$; any power of $2$ gives a contribution of $1$, powers of primes $3$ mod $4$ give $1$ in the even power case and $0$ else, and if $p_i\equiv 1\pmod 4$ then we get a factor of $a_i+1$ as needed. But $\chi$ is multiplicative, so
\begin{equation}\label{r0eq}
	r_0(n)=\sum_{d\mid n}\chi(d).
\end{equation}
\begin{proof}[Proof of Lemma~\ref{r0asymp}]
Substituting the above, $\sum_{n\le x}r_0(n)=\sum_{n\le x}\sum_{d\mid n}\chi(d)=\sum_{d\le x}\lfloor\frac xd\rfloor\chi(d)\sim x(1-\frac13+\frac15-\frac17+\dots)\sim\frac{\pi}4x$ (by the arctan identity.)
\end{proof}
Note that this method does not give any estimate on the error term, but we can provide a crude estimate. We centre a square of side $1$ around each lattice point in a circle of radius $\sqrt{x}$. These squares do not overlap, and the border of the shape covered by them lies fully inside the annulus $\sqrt{x}-1\le r\le \sqrt{x}+1$. Thus, the error on the number of lattice points in the circle is $O((\sqrt{x}+1)^2-(\sqrt{x}-1)^2)=O(\sqrt{x})$. Hence \vspace{-1ex}
	\begin{equation}\label{r0good}
	\sum_{n\le x}r_0(n)=\frac{\pi}4x+O(\sqrt{x}).\end{equation}
The error of the first moment of $r_0$ can be greatly refined, for example:
\begin{lemma}
$\displaystyle\smash{\sum_{n\le x}r_0(n)=\frac{\pi}4x+O(x^\alpha)}$, some $\alpha<\frac13$.
\end{lemma}
\begin{proof}
See~\cite{huxleyexpsums}. Huxley proves for various $\beta<\frac23$ that for a closed curve $C$ of area $A$, drawing it with a scale factor of $M$ gives $AM^2+O(M^\beta)$ integer points inside it. In our case, we take a unit circle which has area $\pi$, and enlarge it by a scale factor of $\sqrt{x}$ to get $\pi\cdot x+O(x^{\beta/2})$ integer lattice points inside a circle of radius $\sqrt{x}$.

But now, $\sum_{n\le x}r_0(n)$ is the number of ways to write any integer $n\le x$ as the sum of squares of a positive and non-negative integer, so we are only interested in lattice points in the first quadrant of the circle of radius $\sqrt{x}$, with the convention of counting lattice points of the form $(0,a)$ but not $(b,0)$. Hence we get the result.
\end{proof}
Estimates for $\alpha$ can be obtained by methods in analytic number theory, such as the exponent pairs method. Currently the best known bound of $\alpha\le\frac{131}{416}+\epsilon$ is due to Huxley's third paper on the topic in 2000. One can show that $\alpha<\frac14$ is impossible, and the conjecture that $\alpha=\frac14+\epsilon$ remains open, after an attempted proof by Cappell and Shaneson in 2007 was retracted. 

Higher moments of $r_0$ have been studied through the lens of binary quadratic forms, and good estimates are shown in Blomer and Granville's paper~\cite[Theorem 1]{blo-gran-repno}. Below we state a Lemma showing the asymptotic behaviour of the second moment.
\begin{lemma}[Second moment of $r_0$] There exists a real constant $H$ such that
	\[\smash{\sum_{n\le x}r_0^2(n)=\frac{x\log x}{4}+Hx+O(x^{\frac12+\epsilon}).}\]
\end{lemma}
\begin{proof}
	As $x^2+y^2$ is a quadratic form, $r_0$ is a special case of $r_f$ as named in the paper~\cite{blo-gran-repno}; we can apply Theorem $1$ with $\beta=2$.
\end{proof}
\section{First moments of representation numbers}
In this section we provide asymptotics for the first moments of $r_1$ and $r_2$. The Prime Number Theorem and its consequences form an important part of the argument, as the distribution of solutions to $a^2+p^2\le x$ or $p^2+q^2\le x$ is closely related to the distribution of primes.

We begin with asymptotics of the first moment of $r_1$. The error term appears to be the best currently known.
We will split the range of summation into two important parts that are estimated in different ways. The method is based on the Prime Number Theorem, however, the error term need not be strong.
\begin{lemma}[The first moment of $r_1$]\label{lemma: first moment of r1}
\[\sum_{n\le x}r_1(n)=\frac{\pi}2\frac{x}{\log x}\left(1+O\Bigl(\frac1{\log x}\Bigr)\right).\]

\end{lemma}
\begin{proof} We use ideas from the part of the proof of~\cite[Theorem~2]{granville2023} which deals with the first moment of~$r_1$, and Daniel's proof in~\cite[Section~4]{daniel_2001}.
	We may write
\begin{align}
\sum_{n\le x}r_1(n)&=\sum_{a^2+p^2\le x}1=
\sum_{a\le\sqrt{\frac x2}}\sum_{p\le\sqrt{x-a^2}}1+\sum_{p\le\sqrt{\frac x2}}\sum_{a\le\sqrt{x-p^2}}1-\sum_{a\le\sqrt{\frac x2}}\sum_{p\le\sqrt{\frac x2}}1 \notag\\
&=\sum_{a\le\sqrt{\frac x2}}\pi(\sqrt{x-a^2})+\sum_{p\le\sqrt{\frac x2}}\lfloor\sqrt{x-p^2}\rfloor-\sqrt{x/2}\cdot\pi(\sqrt{x/2}).\label{eq: rewrite r1}
\end{align}
Note that the weak error term in the PNT is still smaller than $O(\frac{x}{\log^2x})$, as $\exp(\sqrt[10]{\log x})$ grows faster than any power of $\log x$. Hence $\pi(t)=\frac t{\log t}(1+O(\frac1{\log t}))$.

We deal with the components in~\eqref{eq: rewrite r1}. For the first one, since we have specified $a\le\sqrt{x/2}$, we have
$\log x\ge\log(x-a^2) \ge\log(x/2)=\log x-\log 2$, so  \[\frac{2\sqrt{x-a^2}}{\log x}\le\frac{2\sqrt{x-a^2}}{\log (x-a^2)}\le\frac{2\sqrt{x-a^2}}{\log x-\log 2}.\] Since
$(\log x-\log 2)^{-1}=\log^{-1} (x)(1+O(\frac{1}{\log x}))$ we get \[\pi(\sqrt{x-a^2})=\frac{2\sqrt{x-a^2}}{\log x}\left(1+O\Bigl(\frac{1}{\log x}\Bigr)\right).\]
We approximate $\sum_{a\le\sqrt{x/2}}\sqrt{x-a^2}=\int_0^{\sqrt{x/2}}\sqrt{x-a^2}\mathrm{d}a + O(\sqrt{x})$. Now, we calculate
\begin{align*}
	&\int_0^{\sqrt{x/2}}\!\!\sqrt{x-a^2}\ \mathrm{d}a + O(\sqrt{x})
	=\sqrt{x}\int_0^{\sqrt{1/2}}\!\!\sqrt{x}\sqrt{1-u^2}\ \mathrm{d}u + O(\sqrt{x})\\
	&=\frac{x}2\left[\arcsin(u)+u\sqrt{1-u^2}\right]_0^{\sqrt{1/2}}+ O(\sqrt{x})
	=\frac{x}2\Bigl(\frac{\pi}4+\frac12\Bigr)+O(\sqrt{x}),
\end{align*}
and thus \begin{align*}
	\sum_{a\le\sqrt{\frac x2}}\pi(\sqrt{x-a^2}) &=\frac2{\log x}\sum_{a\le\sqrt{\frac x2}}\sqrt{x-a^2}\left(1+O\Bigl(\frac{1}{\log x}\Bigr)\right)
	=\frac2{\log x}\cdot\frac{\pi+2}8x\left(1+O\Bigl(\frac{1}{\log x}\Bigr)\right) \\
	&=\smash[b]{\frac{\pi+2}4\frac{x}{\log x}}\bigl(1+O((\log x)^{-1})\bigr).
\end{align*}

We now deal with the second component of~\eqref{eq: rewrite r1}.

By partial summation, letting $a_n=\mathbbm{1}_{n\text{ is prime}}$ and $f(y)=\sqrt{x-y^2}$ we have
\[\smash[b]{\sum_{p\le\sqrt{\frac x2}}}(\sqrt{x-p^2}+O(1))=O(\sqrt{x})+\pi(\sqrt{x/2})\sqrt{x/2}+\smash[t]{\int_1^{\sqrt{x/2}}\pi(t)t(x-t^2)^{-1/2}\mathrm{d}t}.\] The $\pi(\sqrt{x/2})\sqrt{x/2}$ term cancels with the third component of~\eqref{eq: rewrite r1}, and $O(\sqrt{x})$ can be absorbed into the error term of the first component as $O(\sqrt{x})=O(\frac{x}{\log^2x})$. Hence (noting that $\pi(t)=0$ for $t< 2$) we have
\begin{equation}\label{eq: r1 with int}
\sum_{n\le x}r_1(n)=\frac{\pi+2}4\frac{x}{\log x}\left(1+O\Bigl(\frac1{\log x}\Bigr)\right)+\int_2^{\sqrt{x/2}}\pi(t)t(x-t^2)^{-1/2}\mathrm{d}t.\end{equation} It remains to estimate the integral.
We substitute $t=u\sqrt{x}$ to get
\[\int_2^{\sqrt{x/2}}\frac{t}{\log t}\frac{t}{\sqrt{x-t^2}}\mathrm{d}t=\sqrt{x}\int_{2/\sqrt{x}}^{\sqrt{1/2}}\frac{xu^2}{\sqrt{x}\sqrt{1-u^2}}\frac{\mathrm{d}u}{\log(u\sqrt{x})},\]
and use integration by parts, noting that $\int\frac{u^2}{\sqrt{1-u^2}}\mathrm{d}u=\frac12(\arcsin(u)-u\sqrt{1-u^2})+c$ to get
\begin{align}
&\frac{x}2\left[\frac{\arcsin(u)-u\sqrt{1-u^2}}{\log(u\sqrt{x})}\right]_{2/\sqrt{x}}^{\sqrt{1/2}}+\frac{x}2\int_{2/\sqrt{x}}^{\sqrt{1/2}}\frac{(\arcsin(u)-u\sqrt{1-u^2})\mathrm{d}u}{u\log^2(u\sqrt{x})} \notag\\
=&\frac{x(\frac{\pi}4-\frac12)}{2\log(\sqrt{x/2})}+O(\sqrt{x})+\frac{x}{2}\int_{2/\sqrt{x}}^{\sqrt{1/2}}\Bigl(\frac{\arcsin(u)}u-\sqrt{1-u^2}\Bigr)\frac{\mathrm{d}u}{\log^2(u\sqrt{x})}\label{eq: first estimate r1 int}
.\end{align} Now the last integral is written as
\begin{align*}
\int_{2/\sqrt{x}}^{1/\sqrt[4]{x}}\Bigl(\frac{\arcsin(u)}u-\sqrt{1-u^2}\Bigr)\frac{\mathrm{d}u}{\log^2(u\sqrt{x})}+\int_{1/\sqrt[4]{x}}^{\sqrt{1/2}}\Bigl(\frac{\arcsin(u)}u-\sqrt{1-u^2}\Bigr)\frac{\mathrm{d}u}{\log^2(u\sqrt{x})}.
\end{align*} In the first component, we use the Taylor series of $\arcsin$ and $\sqrt{1-x}$ to get $\frac{\arcsin(u)}u-\sqrt{1-u^2}=\frac23u^2+O(u^4)$, and bound $\log^2(u\sqrt{x})$ from below by $\log^22$. In the second, we bound $\frac{\arcsin(u)}u-\sqrt{1-u^2}$ by a constant, and bound $\log^2(u\sqrt{x})$ from below by $\log^2(\sqrt[4]{x})$. Hence
\begin{align}
\int_{2/\sqrt{x}}^{\sqrt{1/2}}\Bigl(\frac{\arcsin(u)}u-\sqrt{1-u^2}\Bigr)\frac{\mathrm{d}u}{\log^2(u\sqrt{x})}
&=O\left(\int_{2/\sqrt{x}}^{1/\sqrt[4]{x}}\frac1{\sqrt{x}}{\mathrm{d}u}\right)+O\left(\int_{1/\sqrt[4]{x}}^{\sqrt{1/2}}\frac{\mathrm{d}u}{\log^2(\sqrt[4]{x})}\right) \notag\\
&=O(x^{-1/4})+O(\log^{-2}(x)).\label{eq: r1 estimate small int}
\end{align}
We recall that $\pi(t)=O(\frac{t}{\log t})$. Hence substituting~\eqref{eq: r1 estimate small int} into~\eqref{eq: first estimate r1 int} gives
\[\smash{\int_2^{\sqrt{x/2}}\pi(t)t(x-t^2)^{-1/2}\mathrm{d}t=\frac{x}{\log x}\Bigl(\frac{\pi}4-\frac12\Bigr)+O\Bigl(\frac{x}{\log^2x}\Bigr).}\]
Substituting this into~\eqref{eq: r1 with int} gives the required main term of order $\frac x{\log x}$ with coefficient $\frac{\pi+2}4+\frac{\pi-2}4=\frac{\pi}2$, and correct error term $O(\frac x{\log^2x})$.
\end{proof}
We will now use a similar method to obtain asymptotics of the first moment of $r_2$. Note that the error term obtained in this method is better than the error term cited in Sabuncu's recent paper~\cite{Sabuncu_2024}, by a factor of $\log\log x$. A proof of the result with weaker error term can be found in~\cite{plaksin}, for example.
\begin{lemma}[The first moment of $r_2$]\label{r2sum}
	\[\sum_{n\le x}r_2(n)=\pi\frac{x}{\log^2 x}\left(1+O\Bigl(\frac1{\log x}\Bigr)\right).\]
\end{lemma}
\begin{proof} We proceed as in Lemma~\ref{lemma: first moment of r1}.
	We may write
	\begin{align}
		\sum_{n\le x}r_2(n)&=\sum_{p^2+q^2\le x}1=
		\sum_{p\le\sqrt{\frac x2}}\sum_{q\le\sqrt{x-p^2}}1+\sum_{q\le\sqrt{\frac x2}}\sum_{p\le\sqrt{x-q^2}}1-\sum_{p\le\sqrt{\frac x2}}\sum_{q\le\sqrt{\frac x2}}1 \notag\\
		&=2\sum_{p\le\sqrt{\frac x2}}\pi(\sqrt{x-p^2})-(\pi(\sqrt{x/2}))^2.\label{eq: rewrite r20}
	\end{align}
As previously, $\pi(t)=\frac t{\log t}(1+O(\frac1{\log t}))$. Since $t\le\sqrt{x/2}$ we may write $\pi(\sqrt{x-t^2})=\frac {2\sqrt{x-t^2}}{\log (x-t^2)}(1+O(\frac1{\log x}))$, so $\sum_{p\le\sqrt{\frac x2}}\pi(\sqrt{x-p^2})=\sum_{p\le\sqrt{\frac x2}}\frac {2\sqrt{x-t^2}}{\log (x-t^2)}(1+O(\frac1{\log x}))$.
	
By partial summation, letting $a_n=\mathbbm{1}_{n\text{ is prime}}$ and $f(y)=\frac {2\sqrt{x-y^2}}{\log (x-y^2)}$ we have
\begin{equation}\label{summation for r2}
\smash[b]{\sum_{p\le\sqrt{\frac x2}}}\pi(\sqrt{x-p^2})=\left((\pi(\sqrt{x/2}))^2-\smash[t]{\int_1^{\sqrt{x/2}}\pi(t)f'(t)\mathrm{d}t}\right)\Bigl(1+O\Bigl(\frac1{\log x}\Bigr)\Bigr).
\end{equation}We calculate $f'(t)$:
\[\frac{\mathrm{d}}{\mathrm{d}t}\Bigl(\frac {2\sqrt{x-t^2}}{\log (x-t^2)}\Bigr)=-2t\frac{\log(x-t^2)-2}{\sqrt{x-t^2}\log^2(x-t^2)}.\] The $(\pi(\sqrt{x/2}))^2$ term thus has overall coefficient $1$ in~\eqref{eq: rewrite r20}.

We use $\pi(t)=0$ for $t\le 2$.
Substitute $y=\sqrt{x-t^2}$, $u\sqrt{x}=y$ to get
\begin{align}
&\int_2^{\sqrt{x/2}}\frac{t}{\log t}\frac{t}{\sqrt{x-t^2}\log(x-t^2)}\mathrm{d}t=\int_{\sqrt{x/2}}^{\sqrt{x-4}}\frac{\sqrt{x-y^2}}{\log(\sqrt{x-y^2})\log(y^2)}\mathrm{d}y \notag\\
=&\int_{\sqrt{x/2}}^{\sqrt{x-4}}\frac{\sqrt{x-y^2}}{\log(x-y^2)\log y}\mathrm{d}y=\sqrt{x}\int_{\sqrt{1/2}}^{\sqrt{1-4/x}}\frac{\sqrt{x}\sqrt{1-u^2}}{\log(x-xu^2)\log(u\sqrt{x})}\mathrm{d}u. \label{substitutions}
\end{align}
Now, note that $\log(x-xu^2)\log(u\sqrt{x})=(\log x+\log(1-u^2))(\log u+\frac12\log x)=\frac12\log^2x+\log x(\log(u\sqrt{1-u^2}))+\log u\log(1-u^2)$. Since $u\sqrt{1-u^2}\le\frac12$ and $\log u\log(1-u^2)=O(1)$ for $u\in[\sqrt{1/2},\sqrt{1-4/x}]$ we have
\begin{align}
\log(x-xu^2)\log(u\sqrt{x})&=\frac12\log^2x+O(\log x)=\frac12\log^2x\Bigl(1+O\Bigl(\frac1{\log x}\Bigr)\Bigr) \notag \\
\implies \frac{\sqrt{x}\sqrt{1-u^2}}{\log(x-xu^2)\log(u\sqrt{x})} &=\frac{\sqrt{x}\sqrt{1-u^2}}{\frac12\log^2x}\Bigl(1+O\Bigl(\frac1{\log x}\Bigr)\Bigr). \label{estimate integrand}
\end{align}

Since $\int\sqrt{1-u^2}\mathrm{d}u=\frac12(\arcsin(u)+u\sqrt{1-u^2})+c$ we get
	\begin{align}
&\sqrt{x}\int_{\sqrt{1/2}}^{\sqrt{1-4/x}}\frac{\sqrt{x}\sqrt{1-u^2}}{\frac12\log^2x}\mathrm{d}u=\frac{2x}{\log^2x}\Bigl[\frac12(\arcsin(u)+u\sqrt{1-u^2})\Bigr]_{\sqrt{1/2}}^{\sqrt{1-4/x}} \notag\\
=&\frac{x}{\log^2x}\left(\arcsin(\sqrt{1-4/x})+O\Bigl(\frac1{\sqrt{x}}\Bigr)-\frac{\pi}4-\frac12\right)=\frac{x}{\log^2x}\left(\frac{\pi}2+O\Bigl(\frac1{\sqrt{x}}\Bigr)-\frac{\pi}4-\frac12\right) \notag\\
=&\frac{x}{\log^2x}\frac{\pi-2}4\Bigl(1+O\Bigl(\frac1{\sqrt{x}}\Bigr)\Bigr)
=\frac{x}{\log^2x}\frac{\pi-2}4\Bigl(1+O\Bigl(\frac1{\log x}\Bigr)\Bigr).\label{correct answer}\end{align}
We now wish to substitute~\eqref{correct answer} into~\eqref{summation for r2}. Indeed, 
\begin{align*}
&-\smash[t]{\int_1^{\sqrt{x/2}}\pi(t)f'(t)\mathrm{d}t} \\
=&-\int_1^{\sqrt{x/2}}\frac{t}{\log t}\Bigl(1+O\Bigl(\frac1{\log t}\Bigr)\Bigr)(-2t)\frac{\log(x-t^2)-2}{\sqrt{x-t^2}\log^2(x-t^2)}\Bigl(1+O\Bigl(\frac1{\log x}\Bigr)\Bigr)\mathrm{d}t\end{align*} which by~\eqref{correct answer} is equal to
\begin{align}
&2\Bigl(1+O\Bigl(\frac1{\log x}\Bigr)\Bigr)\frac{x}{\log^2x}\frac{\pi-2}4\Bigl(1+O\Bigl(\frac1{\log x}\Bigr)\Bigr) \notag\\
&+\int_1^{\sqrt{x/2}}O\Bigl(\frac{t}{\log^2 t}\Bigr)(-2t)\frac{\log(x-t^2)}{\sqrt{x-t^2}\log^2(x-t^2)}\mathrm{d}t \notag\\
&+\int_1^{\sqrt{x/2}}\frac{t}{\log t}\Bigl(1+O\Bigl(\frac1{\log t}\Bigr)\Bigr)(-2t)\frac{-2}{\sqrt{x-t^2}\log^2(x-t^2)}\mathrm{d}t.
\label{key}
\end{align}
The second component of the sum above can be estimated using the same substitutions as in~\eqref{substitutions}. Proceeding as in~\eqref{estimate integrand} we get that the integrand is \[O\left(\frac{\sqrt{x}\sqrt{1-u^2}}{\frac12\log^3x}\Bigl(1+O\Bigl(\frac1{\log x}\Bigr)\Bigr)\right),\] where $u\sqrt{x}=\sqrt{x-t^2}$, and hence this component is $O(\frac{x}{\log^3 x})$.

Similarly, replacing $\frac{t}{\log t}\bigl(1+O(\frac1{\log t})\bigr)$ with $O(\frac{t}{\log t})$ and using the substitution $u\sqrt{x}=\sqrt{x-t^2}$ gives that the integrand in the third component is
\[O\left(\frac{\sqrt{x}\sqrt{1-u^2}}{\frac14\log^3x}\Bigl(1+O\Bigl(\frac1{\log x}\Bigr)\Bigr)\right),\] so this component is also $O(\frac{x}{\log^3 x})$.

We now substitute~\eqref{key} into~\eqref{summation for r2} to get
\[\smash[b]{\sum_{p\le\sqrt{\frac x2}}}\pi(\sqrt{x-p^2})=(\pi(\sqrt{x/2}))^2+O\Bigl(\frac1{\log x}\Bigr)+2\frac{x}{\log^2x}\frac{\pi-2}4\Bigl(1+O\Bigl(\frac1{\log x}\Bigr)\Bigr)^3,\] and substituting this into~\eqref{eq: rewrite r20} gives
\begin{align*}
	\sum_{n\le x}r_2(n)
	&=(\pi(\sqrt{x/2}))^2+4\frac{x}{\log^2x}\frac{\pi-2}4\Bigl(1+O\Bigl(\frac1{\log x}\Bigr)\Bigr)^2 \\
	&=\frac{x/2}{(\frac12\log x)^2}\Bigl(1+O\Bigl(\frac1{\log x}\Bigr)\Bigr)+\frac{x}{\log^2x}({\pi-2})\Bigl(1+O\Bigl(\frac1{\log x}\Bigr)\Bigr) \\
	&=\pi\frac{x}{\log^2x}\Bigl(1+O\Bigl(\frac1{\log x}\Bigr)\Bigr).\qedhere
\end{align*}
\end{proof}

\section{Zeroth moments}
In this section we discuss the zeroth moments of $r_0$, $r_1$ and $r_2$:
\begin{definition}[Zeroth moments of $r_i$]
	We define $M_i(x)=\#\{n\le x:\ r_i(n)\ge 1\}$.
\end{definition}
We begin by quoting a result of Landau about the zeroth moment of $r_0$.
\begin{lemma}\label{lemma: zeroth moment of r0}
	$\displaystyle \smash{M_0(x)=\Bigl(2\prod_{p\equiv 3\mod 4}\bigl(1-\tfrac1{p^2}\bigr)\Bigr)^{-\frac12}\frac x{\sqrt{\log x}}\bigl(1+O((\log x)^{-1})\bigr)}.$
\end{lemma}
\begin{proof}
	In~\cite[\S 176]{landaugerman}, the main term is shown subject to some constant $b$. The value of $b$ is later shown in~\cite[\S 183]{landaugerman} to be the left hand bracketed expression.
\end{proof}

The following Lemma relates the zeroth, first and second moments of each of the $r_i$. It is a version of the Cauchy-Schwarz inequality.
\begin{lemma}\label{Mibound}
	$\displaystyle\smash[b]{\Bigl(\sum_{n\le x}r_i(n)\Bigr)^2\Bigl(\sum_{n\le x}r_i^2(n)\Bigr)^{-1}\le M_i(x)\le\sum_{n\le x}r_i(n)}$.
\end{lemma}
\begin{proof}
	For the left-hand inequality, we need an application of Cauchy-Schwarz: $\left(\sum_{n\le x}a_nb_n\right)^2\le(\sum_{n\le x}a_n^2)(\sum_{n\le x}b_n^2)$. Set $a_n=r_i(n)$ and $b_n=~\mathbbm{1}_{\{r_i(n)>0\}}$; then
\begin{align*}
	\smash[t]{\Bigl(\sum_{n\le x}r_i(n)\Bigr)^2=\Bigl(\sum_{n\le x}a_nb_n\Bigr)^2}
	&\le\Bigl(\sum_{n\le x}a_n^2\Bigr)\Bigl(\sum_{n\le x}b_n^2\Bigr)\\
	&=\smash[b]{\Bigl(\sum_{n\le x}r_i^2(n)\Bigr)\Bigl(\sum_{n\le x}\mathbbm{1}_{\{r_i(n)>0\}}\Bigr) 
	=\Bigl(\sum_{n\le x}r_i^2(n)\Bigr)M_i(x).}
\end{align*}
Multiplying both sides by $\left(\sum_{n\le x}r_i^2(n)\right)^{-1}$ gives the inequality.
The right-hand inequality is satisfied as $M_i(x)=\#\{n\le x:\ r_i(n)\ge 1\}=\sum_{n\le x}\mathbbm{1}_{\{r_i(n)>0\}}\le\sum_{n\le x}r_i(n)\mathbbm{1}_{\{r_i(n)>0\}}=\sum_{n\le x}r_i(n)$.
\end{proof}
\begin{remark}
	Unless $n$ is a perfect square, we must have $r_2(n)=0$ or $r_2(n)\ge 2$. Hence $M_2(x)\le\frac12\sum_{n\le x}r_2(n)+O(x^{\frac12+\varepsilon})$.
\end{remark}
To estimate the zeroth moment of $r_2$, we will make use of estimates on the first and second moment of $r_2$; Lemma~\ref{r2sqsum} which describes the behaviour of the second moment of $r_2$ can be found in the subsequent section.
\begin{corollary}[The zeroth moment of $r_2$]
	$M_2(x)=\dfrac{\pi}2\dfrac{x}{\log^2 x}\left(1+O\Bigl(\frac{L(x)}{\log x}\Bigr)\right)$.
\end{corollary}
\begin{proof}
We have $M_2(x)\le\frac12\sum_{n\le x}r_2(n)+O(x^{\frac12+\varepsilon})$. So by Lemma~\ref{r2sum}, we get that $M_2(x)$ is at most the RHS of the above equation.

Using Lemma~\ref{r2sum} and Lemma~\ref{r2sqsum} and noting that $O(\frac{x}{\log^3x})=O\bigl(\frac{xL(x)}{\log^3x}\bigr)$ we have
\begin{align*}
	\smash{\sum_{n\le x}r_2(n)^2=2\sum_{n\le x}r_2(n)+O\Bigl(\frac{x}{\log^3x}\Bigr)}
	&=2\frac{\pi x}{\log^2 x}\Bigl(1+O\Bigl(\frac{L(x)}{\log x}\Bigr)\Bigr).
\end{align*}
From Lemma~\ref{Mibound}, $M_2(x)\ge (\sum_{n\le x}r_2(n))^2(\sum_{n\le x}r_2(n)^2)^{-1}$. Hence
\begin{align*}
	M_2(x)&\ge \left(\frac{\pi x}{\log^2 x}\left(1+O\Bigl(\frac{L(x)}{\log x}\Bigr)\right)\right)^2\left(2\frac{\pi x}{\log^2 x}\left(1+O\Bigl(\frac{L(x)}{\log x}\Bigr)\right)\right)^{-1} \\
	&=2\frac{\pi x}{\log^2 x}\left(1+O\Bigl(\frac{L(x)}{\log x}\Bigr)\right).
\end{align*}
We have equality due to the inequalities holding in both directions. 
\end{proof}
Finally, we can bound the zeroth moment of $r_1$ through Lemma~\ref{Mibound}. Indeed, we substitute our known estimate of the first moment of $r_1$ from Lemma~\ref{lemma: first moment of r1}; an asymptotic formula for the second moment was proved by Daniel in~\cite{daniel_2001} and we state the result in Lemma~\ref{lemma: daniel second moment r1}. We obtain \[\frac{\pi^2}4\frac{x^2}{\log^2 x}\frac{4}{2\pi+9}\frac{\log x}x\left(1+O\Bigl(\frac{(\log\log x)^2}{\log x}\Bigr)\right)\le M_1(x)\le\frac{\pi}2\frac x{\log x}\left(1+O\Bigl(\frac1{\log x}\Bigr)\right), \] where the left hand leading term simplifies to $\frac{\pi^2}{2\pi+9}\frac{x}{\log x}$. The lower and upper bounds have the same order of leading term, and moreover Stephan Daniel~\cite{daniel_2001} and Granville, Sabuncu, Sedunova~\cite{granville2023} show that the upper bound is the most correct estimate. The more recent paper provides a more refined error term in~\cite[Theorem 1]{granville2023}:
\begin{theorem}
	$\displaystyle \smash[b]{M_1(x)=\frac{\pi}2\frac x{\log x}-\frac{x(\log\log x)^{O(1)}}{(\log x)^{1+\delta}}}$, where $\delta=1-\frac{1+\log\log 2}{\log 2}$.
\end{theorem}
\section{Second and higher moments of $r_2$}
We finish this chapter with a discussion about moments of $r_2$. For a given $n$, the quantity $r_2^k(n)$ counts the number of $k$-tuples of pairs of primes that sum to $n$, i.e. $r_2^k(n)=\#\{(\textbf{p, q})\in \mathcal{P}^k\times\mathcal{P}^k : \forall i\le k, p_i^2+q_i^2=n\}$. Here, and only in this section:
\begin{notation} $\mathcal{P}$ is the set of all primes.\end{notation}

This set of tuples can be partitioned according to how many of the pairs $p_i,q_i$ are in fact the same pair. We write the following definition:
\begin{definition}[Non-diagonal solutions of level $k$]\
	
	$\mathcal{D}_k(x) := \#\{(\textbf{p, q})\in (\mathcal{P}\cap [0,\sqrt{x}])^{2k} : p_i^2+q_i^2 = p_j^2 + q_j^2 , \{p_i,q_i\} \ne\{p_j,q_j\}, \forall i\ne j\}.$
\end{definition}
The $k$-th moment of $r_2$ can be written as a linear combination of non-diagonal solutions of levels from $1$ to $k$. In particular, $\mathcal{D}_1(x)$, the first moment of $r_2$, appears with coefficient $2^{k-1}$ in the $k$-th moment as for a fixed pair of primes $p,q$ with $p^2+q^2=n$, there are $2^{k-1}$ possible vectors $(\textbf{p, q})\in \mathcal{P}^k\times\mathcal{P}^k$ with $p_1=p$, $q_1=q$ and $\{p_i,q_i\}=\{p,q\}$ for all $i>1$. (We can ignore the case $n=2p^2$, which is rare.)

We already know the asymptotics of the first moment of $r_2$, and Lemma~\ref{r2sqsum} below tells us the behaviour of the second moment of $r_2$. The error term comes from the work of Rieger: his result~\cite[Satz 3]{Rieger1968} is in fact showing that $\mathcal{D}_2(x)\ll\frac{x}{\log^3 x}$. As explained above, $\mathcal{D}_2(x)$ is number of distinct pairs of primes $(p_1,q_1)$, $(p_2,q_2)$ such that $p_1^2+q_1^2=p_2^2+q_2^2\le x$, and $\sum_{n\le x}r_2(n)^2=2\sum_{n\le x}r_2(n)+\mathcal{D}_2(x)$.
\begin{lemma}\label{r2sqsum}
	$\displaystyle\sum_{n\le x}r_2(n)^2=2\sum_{n\le x}r_2(n)+O\left(\frac{x}{\log^3x}\right)$.
\end{lemma}
We state Rieger's theorem from~\cite{Rieger1968} and explain why this shows ${\mathcal{D}_2(x)\ll\frac{x}{\log^3 x}}$.
\begin{theorem}\label{theorem: rieger}
	Let $T(y)$ be the number of tuples $(d,t,n,m)\in\mathbb{N}^4$ such that
	\begin{closealign}dn\le y,\ tn\le y,\ m<n,\ d\ne t;\ tn\pm dm,\ tm\pm dn\in\mathcal{P}.\end{closealign}
	
Then $T(y)=O(y^2(\log y)^{-3})$.
\end{theorem}
\begin{proof}[Proof that ${\mathcal{D}_2(x)\ll\frac{x}{\log^3 x}}$ assuming Theorem~\ref{theorem: rieger}]
	Let $x>0$ and $y=\sqrt{x}$. Suppose we have two distinct pairs of primes $(p_1,q_1)$ and $(p_2,q_2)$ such that $p_1^2+q_1^2=p_2^2+q_2^2=n$, some $n\le x$. Note that it is impossible for just one of the $p_i$, $q_i$ to be $2$, so we can assume that all $p_i$ and $q_i$ are odd. Thus, $c_1=\frac12(p_2-p_1), c_2=\frac12(q_2-q_1)\in\mathbb{Z}$. Let $d=\gcd(c_1,c_2)$ and let the integers $m,n$ be defined by $m=\frac{c_1}d$ and $n=\frac{c_2}d$.
	
	Note that by definition of the $p_i$ and $q_i$, and factorising difference of two squares, we can write $c_1(p_2+p_1)+c_2(q_2+q_1)=0$. But $p_1+c_1=\frac12(p_2+p_1)$ and $q_1+c_2=\frac12(q_2+q_1)$, by definition of the $c_i$. Hence $c_1(p_1+c_1)+c_2(q_1+c_2)=0$. Dividing by $d$ gives $m(p_1+c_1)+n(q_1+c_2)=0$. In particular, $m(p_1+c_1)$ is a multiple of $n$. But $m,n$ are coprime by definition, so $n\mid(p_1+c_1)$.
	
	Let $t=\frac{p_1+c_1}n\in\mathbb{Z}$. Then $tn=p_1+c_1$, so $p_1=tn-dm$. Also, $p_2=p_1+2c_1=tn+c_1=tn+dm$.
	
	Now $tmn+n(q_1+c_2)=0$, so $q_1+c_2=tm$ and hence $q_1=tm-c_2=tm-dn$. Finally $q_2=q_1+2c_2=tm+dn$.
	
	By swapping some signs and ordering of the prime pairs we can guarantee $d,t,n,m$ are as claimed in Theorem~\ref{theorem: rieger}; since $q_1,q_2\le\sqrt{x}$ we have $dn=|c_2|\le\sqrt{x}=y$. Similarly $tn=\frac12(p_2+p_1)\le\sqrt{x}=y$. 
	
	Hence $\mathcal{D}_2(x)\le T(y)=O(y^2(\log y)^{-3})=O(x(\frac12\log x)^{-3})\ll\frac{x}{\log^3 x}$.
\end{proof}
We omit the proof of Theorem~\ref{theorem: rieger}; it relies mainly on sieve methods. We will exhibit other proofs based on the sieve approach in Chapters $6$ and $7$.

Now we turn to third and higher moments of $r_2$. Blomer and Br{\"u}dern showed in~\cite{blombr} that $\mathcal{D}_3(x)\ll x(\log\log x)^6(\log x)^{-3}$; this bound was more recently improved by Sabuncu in his paper~\cite{Sabuncu_2024}. Sabuncu also gave a general formula for upper bounds on moments of $r_2$, as well as lower bounds conditional on a conjecture about the size of a particular set. The conjecture would rely on a generalisation of the Green-Tao theorem; however, any weaker lower bound for this set that can be obtained would also give lower bounds on higher moments of $r_2$.

We define some additional useful notation below:

\begin{definition}[$R_2$, Stirling numbers, $\mathcal{S}_\ell$]\ 
\begin{itemize}
	\item $R_2(n):=\#\{(p,q)\in\mathcal{P}\times\mathcal{P}:p<q,\ p^2+q^2=n\}$. Since order matters when counting the pairs of primes for $r_2$, we have $r_2(n)=2R_2(n)+\mathbbm{1}_{n\in2\mathcal{P}^2}$ ($R_2$ skips the possible case of $n$ being double a prime square.)
	\item $\stirling{k}{\ell}$ is a Stirling number of the second kind, and is the number of ways to partition $k$ distinct objects into $\ell$ nonempty sets.
	\item $\displaystyle\mathcal{S}_\ell:=\ell!\sum_{n\le x}\binom{R_2(n)}{\ell}$.
\end{itemize}
\end{definition}

We can express moments of $r_2$ in terms of $R_2$ as follows:
\begin{align}
\sum_{n\le x}r_2^k(n)&=\sum_{n\le x}(2R_2(n)+\mathbbm{1}_{n\in2\mathcal{P}^2})^k=2^k\sum_{n\le x}\left(R_2^k(n)+O(R_2^{k-1}(n)\mathbbm{1}_{n\in2\mathcal{P}^2})\right)\notag\\
&=2^k\sum_{n\le x}R_2^k(n)+O\Bigl(\!\!\!\sum_{p\le\sqrt{x}/2}R_2^{k-1}(2p^2)\Bigr)
=2^k\smash[b]{\sum_{n\le x}}R_2^k(n)+O\Bigl(\frac{\sqrt{x}}{\log x}\Bigr). \label{R2k}
\end{align}
The last equality is achieved as $R_2(2p^2)\le r_0(2p^2)\le 4$, so the sum in the error term is bounded by $4\#\{p\le\frac{\sqrt{x}}2\}\ll\frac{\sqrt{x}}{\log x}$, using the Prime Number Theorem.
\begin{lemma}\label{fallfact}
	$\displaystyle P(x,k):=\smash[b]{\sum_{\ell=0}^k\stirling{k}{\ell}x(x-1)\dots(x-\ell+1)}=x^k$.
\end{lemma}
\begin{proof}
	We begin by stating an equivalent definition for $\stirling{k}{\ell}$. Note that $\stirling{k}{\ell}\cdot\ell!$ is the number of ways to partition $k$ distinct objects into $\ell$ labelled sets, or equivalently, the number of surjections $\{1,\dots,k\}\to\{1,\dots,\ell\}$.
	
	Suppose $x$ is a given integer. Then $\stirling{k}{\ell}x(x-1)\dots(x-\ell+1)=\ell!\stirling{k}{\ell}\binom{x}{\ell}$, which is the number of functions $\{1,\dots,k\}\to\{1,\dots,x\}$ whose image contains exactly $\ell$ elements. Hence $P(x,k)$ is the total number of functions $\{1,\dots,k\}\to\{1,\dots,x\}$, as the image of such a function has size between $0$ and $k$. We have $x$ choices for each element of  $\{1,\dots,k\}$, so $P(x,k)=x^k$ for all integer $x$. By polynomial interpolation of degree $k$ polynomials we in fact get equality for all $x\in\mathbb{R}$.
\end{proof}
By Lemma~\ref{fallfact}, the main term of the right hand side of~\eqref{R2k} can be rewritten as
\begin{equation}\label{eq: R2k to Sl}
\!\!\!\smash{\sum_{n\le x}R_2^k(n)=\sum_{n\le x}\sum_{\ell=0}^k\stirling{k}{\ell}R_2(n)(R_2(n)-1)\dots(R_2(n)-\ell+1)=\sum_{\ell=1}^k\stirling{k}{\ell}\mathcal{S}_\ell.}
\end{equation}
(We may start the sum at $1$ as $\stirling{k}{0}$, the number of ways to partition $k$ objects into $0$ nonempty subsets, is $0$.)

Thus, to find an estimate for the $k$-th moment of $r_2$, we seek bounds for the $S_\ell$.
\begin{theorem}\cite[Theorem 1.1]{Sabuncu_2024}\label{Sk}
	\begin{align*}
		\mathcal{S}_2,\mathcal{S}_3&\smash[t]{\ll x\frac{(\log\log x)^2}{\log^3x}, \
		\mathcal{S}_4\ll x\frac{\log\log\log x}{\log x}},\\
		\mathcal{S}_k&\ll_k \smash[b]{x(\log x)^{2^{k-1}-2k-1}, k\ge 5.}
	\end{align*}
\end{theorem}
Lower bounds for the $S_k$ would correspondingly give lower bounds for $k$-th moments of $r_2$. Sabuncu conjectures such bounds in his paper~\cite{Sabuncu_2024}. However, the results stated in Theorem~\ref{Sk} above do not rely on any conjecture, and lead to the following:
\begin{corollary}[Higher moments of $r_2$]
	\begin{align*}
	\smash[t]{\sum_{n\le x}r_2^3(n)}&=\smash[t]{\frac{4\pi x}{\log^2x}+O\Bigl(x\frac{(\log\log x)^2}{\log^3x}\Bigr)} \\
		\sum_{n\le x}r_2^4(n)&\ll x(\log\log\log x)(\log x)^{-1} \\
		\smash[b]{\sum_{n\le x}r_2^k(n)}&\ll_k x(\log x)^{2^{k-1}-2k-1}, k\ge 5.
	\end{align*}
\end{corollary}
\begin{proof}
	We state upper bounds for various $\mathcal{S}_k$ in Theorem~\ref{Sk}. Note that for $k\ge 4$, the corresponding upper bound for $\mathcal{S}_k$ is the main contribution in the upper bound obtained for $\sum_{\ell=1}^k\stirling{k}{\ell}\mathcal{S}_\ell$. Hence from~\eqref{eq: R2k to Sl} we immediately get the upper bounds for $\sum_{n\le x}r_2^k(n),k\ge 4$.
	
	For the third moment, we use~\eqref{R2k} to get
	\begin{align*}
\smash[t]{\sum_{n\le x}r_2^3(n)}&=\smash[t]{2^3\sum_{n\le x}R_2^3(n)+O\Bigl(\frac{\sqrt{x}}{\log x}\Bigr)=8\mathcal{S}_1+O\Bigl(x\frac{(\log\log x)^2}{\log^3x}\Bigr)+O\Bigl(\frac{\sqrt{x}}{\log x}\Bigr)}\\
	&=8\sum_{n\le x}R_2(n)+O\Bigl(x\frac{(\log\log x)^2}{\log^3x}\Bigr)
	=4\sum_{n\le x}r_2(n)+O\Bigl(x\frac{(\log\log x)^2}{\log^3x}\Bigr)\\
	&=\frac{4\pi x}{\log^2x}+O\Bigl(x\frac{(\log\log x)^2}{\log^3x}\Bigr).\qedhere
	\end{align*}\end{proof}
The proof of Theorem~\ref{Sk} relies on sieve methods, which we will learn about more in the next chapter. Indeed, the proof structure is similar to that of Theorem~\ref{thm:r1cl upper bound}, and Sabuncu remarks in~\cite{Sabuncu_2024} that one obtains analogously to his methods that $\sum_{n\le x}r_1^k(n)\ll_k x(\log x)^{2^{k-1}-k-1}$ for $k\ge 5.$ We will see later that this is consistent with our results about moments of $r_1$.

We will now look at Sabuncu's conjecture. We provide a definition from~\cite{Sabuncu_2024}:
\begin{definition}
	Let $k\in\mathbb{Z}_{\ge 2}$ and $\{u_1,v_1,\dots,u_k,v_k\}\subseteq\mathbb{Z}$ with $u_i^2+v_i^2=u_j^2+v_j^2$ for all $i,j\le k$ such that $(u_i,v_i)\ne(u_j,v_j)$ for $i\ne j$. Let $\textbf{u}=(u_1,\dots,u_k)$ and $\textbf{v}=(v_1,\dots,v_k)$. We call $(\textbf{u, v})$ an \textbf{admissible representation} if
\[\forall p\in\mathcal{P},\exists a_p\in\mathbb{F}_p: (a_p^2+1)\prod_{i=1}^k(u_ia_p-v_i)(v_ia_p+u_i)\equiv 0\pmod p.\]
\end{definition}
For a set $\{u_1,v_1,\dots,u_k,v_k\}\subseteq\mathbb{Z}$ as above, define $m:=u_1^2+v_1^2$. Then $m=u_i^2+v_i^2$ by definition of the numbers $u_i$, $v_i$. Let $\mathcal{L}_m$ be the set of linear forms (linear functions $\mathbb{R}^2\to\mathbb{R}$) as follows: $\mathcal{L}_m=\{u_ix+v_iy,v_ix-u_iy\colon 1\le i\le k\}$. We may now state the conjecture.
\begin{claim}[Conjecture,~\cite{Sabuncu_2024}]
Let $k\in\mathbb{Z}_{\ge 2}$ and $\{u_1,v_1,\dots,u_k,v_k\}\subseteq\mathbb{Z}$ as above. If $(\textbf{u, v})$ is an admissible representation, then
\begin{align*}
\#\{(r,s)\in\mathbb{N}^2:r^2+s^2\le x\text{ is prime}, 0<u_ir-v_is&<v_ir+u_is,\forall i\le k\forall\phi\in\mathcal{L}_m,\phi(r,s)\in\mathcal{P}\} \\
&\gg_k \mathfrak{S}(u_i,\dots,v_k)\frac{x}{(\log x)^{2k+1}}
\end{align*}
where
\[\mathfrak{S}(u_i,\dots,v_k)=\prod_p\Bigl(1-\frac{\nu_p(u_1,\dots,v_k)}{p^2}\Bigr)\Bigl(1-\frac1p\Bigr)^{2k+1}\]
and \[\nu_p(u_1,\dots,v_k)=\#\{(r,s)\in\mathbb{F}_p^2:(r^2+s^2)\prod_{\phi\in\mathcal{L}_m}\phi(r,s)\equiv 0\pmod p\}.\]
\end{claim}
As we will see later, this conjecture is the equivalent of a good bound on a sieving set. Assuming that it holds, Sabuncu in particular proves in~\cite{Sabuncu_2024} that $\smash[b]{\sum_{n\le x}r_2^k(n)}\asymp_k x(\log x)^{2^{k-1}-2k-1}$ for $k\ge 5$, and under a similar conjecture one could show that
\begin{theorem} $\sum_{n\le x}r_1^k(n)\asymp_k x(\log x)^{2^{k-1}-k-1}$ for $k\ge 5.$\label{sabuncur1}\end{theorem}
\chapter{Sieves}\label{chapter: sieves}\vspace{-2ex}
In this chapter we will explore the basics of sieve theory and look at some particular examples of sieves that are used in analytic methods. Powerful results such as the fundamental lemma of sieve theory will be useful to us later on to provide upper bounds on certain sums when estimating asymptotics of representation numbers. For references, see~\cite{Fordsieves},~\cite{richert1976lectures} and~\cite{heathbrown2002lecturessieves}.

A general sieve problem typically requires the following ingredients:
\begin{itemize}
	\item A finite probability space $(\mathcal{A},\mathcal{F},\mathbb{P})$, whose size we can estimate.
	\item A collection of primes, $\mathfrak{p}$, with respect to which we construct the sieve.
	\item For each prime $p\in\mathfrak{p}$, a collection of residue classes $\mathfrak{r}(p)\subseteq\{0,1,\dots,p-1\}$ modulo $p$.
	\item A real number $z$.
\end{itemize}
A sieve problem usually deals with a multiset or sequence of natural numbers; our task is to sift out (remove) those that fall into one of our chosen residue classes modulo an element of $\mathfrak{p}_{\le z}$, and estimate how many remain. Depending on the number of residue classes chosen modulo each prime $p\in\mathfrak{p}$, we may consider the sieve to be \textbf{small} (if \textbf{few} residue classes are chosen) or \textbf{large} (if \textbf{many} residue classes are chosen).

\section{Definitions}
We will need to introduce some additional definitions and notation to use in this chapter. Throughout the chapter, we will consider $\mathbb{P}$ to be the uniform probability distribution -- each element of $\mathcal{A}$ has probability $\dfrac1{|\mathcal{A}|}$ of occurring. The probability space will be discrete, i.e. $\mathcal{F}$ is the power set of $\mathcal{A}$.
\begin{definition}(Sieve notation)
	For $(\mathcal{A},\mathcal{F},\mathbb{P})$ a finite probability space, $\mathfrak{p}$ a (possibly infinite) set of prime numbers and $z\in\mathbb{R}$:
	\begin{itemize}
	\item We denote by $X$ an estimate of $|\mathcal{A}|$, and let $R:=|\mathcal{A}|-X$ (the remainder).
	\item We define events $\mathcal{A}_p,\ \forall p\in\mathfrak{p}$ (often based on a congruence condition mod $p$).
	\item For each squarefree $d\in\mathbb{N}$ that is a product of primes in $\mathfrak{p}$, let $\mathcal{A}_d$ be the event  $\displaystyle\mathcal{A}_d=\bigcap_{p\mid d}\mathcal{A}_p$. In particular, we adopt the convention $\mathcal{A}_1=\mathcal{A}$.
	\item $\mathcal{P}(z) = \displaystyle\prod_{\substack{p\le z\\ p\in\mathfrak{p}}}p$ is the product of all primes in $\mathfrak{p}$ that are at most $z$.
	\item $S(\mathcal{A},\mathfrak{p},z)=|\mathcal{A}|\cdot\mathbb{P}(\bigcap_{p\mid\mathcal{P}(z)}\overline{\mathcal{A}_p})=\#\{a\in\mathcal{A}:\forall p\mid\mathcal{P}(z),\ a\not\in\mathcal{A}_p\}$.
	\end{itemize}
\end{definition}
\section{The small, or Selberg, sieve}
In this section, we work with a special case of a multivariate polynomial small sieve; let $F\colon\mathbb{Z}^k\to\mathbb{Z}$,  $\mathcal{A}\subseteq\mathbb{Z}^k$ a finite subset, and define $\mathcal{A}_p=\{\textbf{a}=(a_1,\dots,a_k)\in\mathcal{A}:F(\textbf{a})\equiv 0\pmod p\}$. Indeed, setting $S=\{F(\textbf{a}):a\in\mathcal{A}\}$, our sieve problem is sift out all multiples of $p$ for any $p\in\mathfrak{p}_{\le z}$ from $S$.

Recall the function $I(n)=\begin{cases}
	1, n=1\\
	0, n>1
\end{cases}$. Now, given $\textbf{a}\in\mathcal{A}$, $I(\gcd(F(\textbf{a}),\mathcal{P}(z)))$ is an indicator function which determines whether $\textbf{a}$ is counted towards the sum $S(\mathcal{A},\mathfrak{p},z)$. i.e. $\displaystyle S(\mathcal{A},\mathfrak{p},z)=\sum_{\textbf{a}\in\mathcal{A}}I(\gcd(F(\textbf{a}),\mathcal{P}(z)))$. Now, since $\mu*\mathbbm{1}=I$ we have $\displaystyle S(\mathcal{A},\mathfrak{p},z)=\sum_{\textbf{a}\in\mathcal{A}}(\mu*\mathbbm{1})(\gcd(F(\textbf{a}),\mathcal{P}(z)))=\sum_{\textbf{a}\in\mathcal{A}}\sum_{d\mid (F(\textbf{a}),\mathcal{P}(z))}\mu(d)$. 
We can reverse the order of summation to get
\begin{equation}\label{S(A,p,z) equality}
	S(\mathcal{A},\mathfrak{p},z)=\sum_{d\mid \mathcal{P}(z)}\sum_{\substack{\textbf{a}\in\mathcal{A}\\d\mid F(\textbf{a})}}\mu(d)=\sum_{d\mid \mathcal{P}(z)}\mu(d)\sum_{\substack{\textbf{a}\in\mathcal{A}\\d\mid F(\textbf{a})}}1=\sum_{d\mid \mathcal{P}(z)}\mu(d)|\mathcal{A}_d|.
\end{equation}
Thus in order to estimate $S(\mathcal{A},\mathfrak{p},z)$, we estimate $|\mathcal{A}_d|$ for squarefree $d\in\mathbb{N}$.
%
\begin{definition}
	We define a multiplicative function $g$ as follows:
\begin{itemize}
	\item $\forall n>1: g(p^n)=0$, for any prime $p$.
	\item $g(p)=0$ for any prime $p\not\in\mathfrak{p}$.
	\item For a prime $p\in\mathfrak{p}$, $0\le g(p)< 1$ is chosen so that $g(p)X$ is close to $|\mathcal{A}_p|$.
\end{itemize}
\end{definition}

Due to the first condition, $g$ is a function that is $0$ outside of squarefrees, and by multiplicativity, for squarefree $d$ we have $g(d)=\prod_{p\mid d}g(p)$.

We additionally define $R_d:=|\mathcal{A}_d|-g(d)X$; from the definition above we can see that $R_p$ is small for primes $p\in\mathfrak{p}$. Now, we can rewrite~\eqref{S(A,p,z) equality} and obtain
\begin{equation}\label{S(A,p,z) inequality}
	\smash{S(\mathcal{A},\mathfrak{p},z)=X\sum_{d\mid \mathcal{P}(z)}\mu(d)g(d)+\sum_{d\mid \mathcal{P}(z)}\mu(d)R_d\le X\sum_{d\mid \mathcal{P}(z)}\mu(d)g(d)+\sum_{d\mid \mathcal{P}(z)}|R_d|.}
\end{equation}
Note that it is very difficult to control the error terms $R_d$, especially as we are summing for $d$ up to $\mathcal{P}(z)$, which is large. One method to deal with this is to try and get rid of $R_d$ for $d$ past a certain point. We start with the following result.
\begin{lemma}\label{lemma: S(a,p,z) mu+}
	Let $\mu^+$ be an arithmetic function so that for all $\textbf{a}\in\mathcal{A}$,
\[\sum_{d\mid\gcd(F(\textbf{a}),\mathcal{P}(z))}\mu^+(d)\ge
\begin{cases}
1,&\gcd(F(\textbf{a}),\mathcal{P}(z))=1 \\
0,&\gcd(F(\textbf{a}),\mathcal{P}(z))>1
\end{cases}.\]
Then $\displaystyle	S(\mathcal{A},\mathfrak{p},z)\le X\sum_{d\mid \mathcal{P}(z)}\mu^+(d)g(d)+\sum_{d\mid \mathcal{P}(z)}|\mu^+(d)R_d|$.
\end{lemma}
We now find a suitable function $\mu^+$. Let $\lambda_d\in\mathbb{R}$ be an arbitrary sequence of real numbers, such that $\lambda_1=1$, and set $\displaystyle\mu^+(d)=\sum_{d=[d_1,d_2]}\lambda_{d_1}\lambda_{d_2}$, where $[d_1,d_2]$ denotes the lowest common multiple of $d_1$ and $d_2$. We check $\mu^+$ satisfies the condition in Lemma~\ref{lemma: S(a,p,z) mu+}: If for some $\textbf{a}\in\mathcal{A}$ we have $\gcd(F(\textbf{a}),\mathcal{P}(z))=1$, then $\sum_{d\mid\gcd(F(\textbf{a}),\mathcal{P}(z))}\mu^+(d)=\lambda_1^2=1\ge 1$, and if $\gcd(F(\textbf{a}),\mathcal{P}(z))>1$, then $\sum_{d\mid\gcd(F(\textbf{a}),\mathcal{P}(z))}\mu^+(d)=\Bigl(\sum_{d\mid\gcd(F(\textbf{a}),\mathcal{P}(z))}\lambda_d\Bigr)^2\ge 0$. Hence by Lemma~\ref{lemma: S(a,p,z) mu+} we get
\begin{equation}\label{S(A,p,z) other inequality}
	S(\mathcal{A},\mathfrak{p},z)\le X\sum_{d\mid \mathcal{P}(z)}\sum_{d=[d_1,d_2]}\lambda_{d_1}\lambda_{d_2}g(d)+\sum_{d\mid \mathcal{P}(z)}\sum_{d=[d_1,d_2]}|\lambda_{d_1}\lambda_{d_2}R_d|.
\end{equation}
We can now restrict the maximum $d$ for which $R_d$ is considered by taking $\lambda_d=0$ for $d\ge\xi$, some $\xi$; then $\mu^+(d)=0$ for $d\ge\xi^2$, as $d,d'<\xi\implies [d,d']\le d_1d_2<\xi^2$ so for $d\ge\xi^2$ we have $\displaystyle \sum_{d=[d_1,d_2]}\lambda_{d_1}\lambda_{d_2}=\sum_{d=[d_1,d_2]}0=0$. To improve the bound on $S(\mathcal{A},\mathfrak{p},z)$ as much as possible, we now optimise the values of $\lambda_d$ for $d\le\xi$.

Let $\displaystyle G:=\sum_{d\mid \mathcal{P}(z)}\sum_{d=[d_1,d_2]}\lambda_{d_1}\lambda_{d_2}g(d)=\sum\!\!\!\!\sum_{[d_1,d_2]\mid \mathcal{P}(z)}\mu^+([d_1,d_2])g([d_1,d_2])$. 

The sieve problem now becomes an optimisation problem. We do not provide the reasoning for finding the optimal choice here, but it can be found in Heath-Brown's lecture notes~\cite{heathbrown2002lecturessieves}, for example. Note that the convention followed in this text is to use a `weight' function denoted $\omega$; our function $g(p)$ corresponds to $\frac{\omega(p)}p$ in~\cite{heathbrown2002lecturessieves}.
\begin{definition}\
\begin{itemize}
	\item Let $h$ be defined by $\frac1{h(n)}=\sum_{d\mid n}\mu(\frac nd)g(d)^{-1}$. It is multiplicative as $\mu$, $g^{-1}$ are, and $\frac1h$ is a Dirichlet convolution. In particular, $\frac1{h(p)}=-1+\frac1{g(p)}=\frac{-g(p)+1}{g(p)}$, so $h(p)=\frac{g(p)}{1-g(p)}=g(p)(1-g(p))^{-1}$. Thus $h(d)=g(d)\prod_{p\mid d}(1-g(p))^{-1}$ by multiplicativity of $h$ and $g$.
	\item $G_d(\xi,z):=\sum_{dl\mid\mathcal{P}(z),l<\xi}h(l)$. We adopt the convention $G_1(\xi,z)=G(\xi,z)$, so
	$G(\xi,z):=\sum_{l\mid\mathcal{P}(z),l<\xi}h(l)$.
\end{itemize}
\end{definition}
Heath-Brown shows in~\cite{heathbrown2002lecturessieves} that the bound is optimised when
\[\smash{\lambda_d=\mu(d)\prod_{p\mid d}(1-g(p))^{-1}\frac{G_d(\frac\xi d,z)}{G(\xi,z)}.}\]
Substituting into~\eqref{S(A,p,z) other inequality} we get \[\smash{\sum_{d\mid \mathcal{P}(z)}\sum_{d=[d_1,d_2]}\lambda_{d_1}\lambda_{d_2}g(d) = \dfrac1{G(\xi,z)}.}\] Hence we obtain
\begin{lemma}[Fundamental lemma of sieve theory]\label{lemma: orig sieve lemma}
\[\smash{S(\mathcal{A},\mathfrak{p},z)\le\frac{X}{G(\xi,z)}+\sum_{d\le \xi^2,d\mid\mathcal{P}(z)}3^{\omega(d)}|R_d|.}\]
\end{lemma} Here, recall $\omega$ is the function counting the number of distinct prime factors of $n$.

Halberstam and Richert showed in~\cite[Lemma 5.4]{halberstam_richert} that $G(z,z)=\frac1{e^{\gamma\kappa}\Gamma(\kappa+1)}\prod_{p<z}(1-g(p))^{-1}(1+O(\tfrac1{\log z}))$. So setting $\xi=z$, we get

\begin{theorem}\label{lemma: fundamental lemma of sieves}
$\displaystyle S(\mathcal{A},\mathfrak{p},z)\le e^{\gamma\kappa}\Gamma(\kappa+1)\prod_{p<z}(1-g(p))(1+O(\tfrac1{\log z}))+\sum_{d\le z^2,d\mid\mathcal{P}(z)}\!\!\!\!\!\!3^{\omega(d)}|R_d|.$
\end{theorem}

Note that Lemma~\ref{lemma: orig sieve lemma} and Lemma~\ref{lemma: fundamental lemma of sieves} are more general than the polynomial sieve case we discussed earlier; as long as a suitable function $g(d)$ is defined in a sieving problem, we can use these results. In particular, in Chapter $7$ we use a slightly modified polynomial sieve where the events $\mathcal{A}_p$ are dependent on divisibility by $p^2$, not just $p$.
\chapter{Common divisors of sums of squares}\vspace{-2ex}
In this chapter we will explore some modified representation numbers, as defined below; applying properties of these functions in combination with methods such as sieves will help us prove results about higher moments of $r_1$.
\begin{definition}
	\ 
	
	\begin{itemize}
		\item $r_0^*(n):=\#\{a\ge 0,b>0:n=a^2+b^2\text{ and } \gcd(a,b)=1\}$
		\item $r_1^*(n):=\#\{a\ge 0,p>0:n=a^2+p^2\text{ and } \gcd(a,p)=1, p\text{ prime}\}$
		\item $r_2^*(n):=\#\{p,q>0:n=p^2+q^2,\ p,q\text{ prime, }p\ne q\}$
	\end{itemize}
\end{definition}
\section{The function $r_0^*$}
Granville, Sabuncu and Sedunova's paper~\cite{granville2023} uses auxiliary results about $r_0^*$ in their proof for an upper bound on expressions of form $\sum\binom{r_1(n)}{\ell}$. We provide full proofs of some of these results below, which were not given in the paper.
\begin{lemma}[Properties of $r_0^*$]\label{r0*}
	$r_0^*(2)=1$, $r_0^*(2^k)=0$ for $k\ge 2$, $r_0^*(p^k)=1+(\frac{-1}p)$ for odd primes $p$, where $(\frac{\cdot}p)$ is the Legendre symbol modulo $p$.
\end{lemma}
\begin{proof}
	\begin{itemize}
		\item $r_0^*(2)=1$ as the only possible representation is $2=1^2+1^2$.
		\item $r_0^*(2^k)=0$ for $k\ge 2$ as any multiple of $4$ written as a sum of two squares must be written as a sum of two even squares, which are then not coprime.
		\item If $p\equiv 3\pmod 4$ then only even powers of $p$ have a representation as a sum of two squares, and in any such representation, both squares have to be multiples of $p$. Also, $-1$ is not a square modulo $p$. Thus we indeed get $r_0^*(p^k)=0=1-1=1+\bigl(\frac{-1}p\bigr)$.
		\item If $p\equiv 1\pmod 4$ then, as shown in the proof of Lemma~\ref{r0asymp}, we have $r_0(p^k)=k+1$. So in particular, $r_0(p)=2$. Hence $r_0^*(p)=2$, because any possible common factor of squares whose sum is $p$ would be a factor of $p$, but $p<p^2$. Similarly, as $p^2 = 0 + p^2$ is the only representation of $p^2$ where both squares are multiples of $p^2$ (and the second one is positive), we get $r_0^*(p^2)=3-1=2$. In general, for powers $k\ge 2$, $-1$ is a square modulo $p$ so we have \\ \hspace*{\fill}%
$r_0^*(p^k)=r_0(p^k)-r_0(p^{k-2})=2=1+1=1+\bigl(\tfrac{-1}p\bigr).$\hspace*{\fill}\qedhere
	\end{itemize}
\end{proof}
\begin{lemma}\label{r0*mult}$r_0^*$ is a multiplicative function.
\end{lemma}
\begin{proof}
	Firstly, $r_0$ is multiplicative: Recall from equation~\eqref{r0eq} that $r_0(n)=\sum_{d\mid n}\chi(d)$. Since the non-principal Dirichlet character $\chi$ in the proof of Lemma~\ref{r0asymp} is multiplicative, and so is the function $\mathbbold{1}$ which takes the value $1$ everywhere, we have that their convolution $r_0$ is multiplicative too.
	
	



We start by checking that if $\gcd(a,b)=\gcd(c,d)=1$ and $a^2+b^2 = n_1$, $c^2+d^2=n_2$ with $\gcd(n_1,n_2)=1$ then $(a+bi)(c+di)$ has coprime real and imaginary part.

We have that $\Re((a+bi)(c+di)) = ac-bd$ and $\Im((a+bi)(c+di)) = ad+bc$. Suppose $p$ is a common prime factor of $ac-bd$ and $ad+bc$. Then in particular, $p$ divides $c(ac-bd)=ac^2-bcd$ and $d(ad+bc)=ad^2+bcd$ so
\begin{equation}\label{pdiv1}
	p\mid a(c^2+d^2)=an_2.
\end{equation} Similarly
\begin{equation}\label{pdiv2}
	p\mid abc-b^2d, p\mid a^2d + abc\implies p\mid d(a^2+b^2)=dn_1.
\end{equation}
Suppose $p\mid a$. Then, since $p\mid ac-bd$ we get $p\mid bd$. But $a,b$ are coprime, hence $p\mid d$. Since $c,d$ are coprime, $p$ does not divide $c$. But then $ad+bc$ cannot be a multiple of $p$, which is a contradiction.

Similarly, $p$ is not a factor of $d$. Hence~\eqref{pdiv1} $\implies p\mid n_2$, and~\eqref{pdiv2} $\implies p\mid n_1$. But $\gcd(n_1,n_2)=1$, so this is a contradiction.

Thus $ac-bd$ and $ad+bc$ have no common prime factors, so they are coprime.

We now conclude that $r_0^*$ is multiplicative using the same ideas as in the proof of Lemma~\ref{r0asymp}: $r_0^*(n)$ is the number of points  $\alpha + \beta i\in\mathbb{Z}_{\ge 0}[i]$ on the circle of radius $\sqrt{n}$ such that $\gcd(\alpha,\beta)=1$, and the way to count these points arises from considering the prime factorisation of $n$. All representations of $n$ as a sum of two coprime squares arise from products of Gaussian prime factors of $n$, so using the above we get $r_0^*(n)=\prod_{i=1}^{k}r_0^*(p^{a_i})$ where $n=\prod_{i=1}^{k}p^{a_i}$. Thus $r_0^*$ is indeed multiplicative.
\end{proof}
	

Lemma~\ref{r0*mult} and Lemma~\ref{r0*} together determine the value of the function $r_0^*$ at any integer. In particular, we see that $r_0^*(n)=0$ if $4\mid n$ or $p\mid n$ for some $p\equiv 3\pmod 4$; otherwise $r_0^*(n)=2^j$, where $j$ is the number of distinct odd (congruent to $1$ mod $4$) prime factors.

The following result is also from~\cite{granville2023}.
 I have provided a detailed proof below.
\begin{lemma}\label{r0*bound}
	Let $z=x^{1/L}$, and recall $P(n)$ is the largest prime factor of $n$. Then for any integer $m\ge 1$,
\[\sum_{{\substack{n\le x\\ P(n)\le z\text{ or }(P(n)^2)\mid n}}}r_0^*(n)^m\ll_m\frac{x}{(\log x)^{1000}}.\]
\end{lemma}
\begin{proof}
We split into two cases, according to the conditions on $P(n)$.

\textbf{Case 1}: First, suppose $P(n)\le z$.
Let  $\sigma=1-\frac{\log(L\log L)}{\log z}$, where we recall $L=\log\log x$. Note that by definition of $z$ the denominator is equal to $\frac1L\log x$. Since $L\log L\ll\log x$ we have $\log(L\log L)\ll L$, so $L\log(L\log L)\ll L^2$. Hence $L\log(L\log L)<\frac12\log x$ for large enough $x$. Thus we can assume $\sigma\in(\frac12,1)$.
 
Since $n\le x$, we can bound the LHS by another sum:
\begin{equation}\label{xnsig}
	\sum_{{\substack{n\le x\\ P(n)\le z}}}r_0^*(n)^m\le \sum_{{\substack{n\le x\\ P(n)\le z}}}r_0^*(n)^m\Bigl(\frac{x}{n}\Bigr)^\sigma
	\le \sum_{P(n)\le z}r_0^*(n)^m\Bigl(\frac{x}{n}\Bigr)^\sigma .
\end{equation}

Recall that $r_0^*$ is multiplicative. The 
right hand side of~\eqref{xnsig} is the Dirichlet series $D(\sigma)=x^{\sigma}\sum_{n\ge 1}\frac{r_0^*(n)^m\mathbbm{1}\{P(n)\le z\}}{n^\sigma}$. Since all prime factors of $n$ are bounded by $z$, we can write an Euler product
\[\smash{\sum_{P(n)\le z}r_0^*(n)^m\Bigl(\frac{x}{n}\Bigr)^\sigma=x^\sigma\prod_{p\le z}\left(1+\frac{r_0^*(p)^m}{p^\sigma}+\frac{r_0^*(p^2)^m}{p^{2\sigma}}+\dots\right).}\]
We know from Lemma~\ref{r0*} that $r_0^*(p)\le 2$ for any prime $p$; using the Taylor expansion of the exponential function, the above is thus bounded by $x^\sigma\exp(\sum_{p\le z}\frac{2^m}{p^\sigma})$.
We use Abel summation to obtain a bound for $x^\sigma\exp(\sum_{p\le z}\frac{1}{p^\sigma})$; the original expression is $x^\sigma\exp(\sum_{p\le z}\frac{1}{p^\sigma})^{O_m(1)}$.

Let $a_n\in\mathbb{C}$ be $1$ for prime $n$ and $0$ else, and define $f(x)=\frac1{x^a}$, for $a>0$. Then $f'(x)=-ax^{-1-a}$. Thus
\begin{equation}\label{partsum}
\smash{\sum_{p\le u}\frac1{p^a}=\sum_{n\le u}
a_nf(n)=\pi(u)\cdot\frac1{u^a}+a\int_1^u\pi(t)t^{-1-a}\mathrm{d}t.}
\end{equation}
Substitute $u=z$ into the above. By the Prime Number Theorem, $\pi(z)=\frac{z}{\log z}+O\bigl(\frac{z}{(\log z)^2}\bigr)$, so $\pi(z)\cdot\frac1{z^a}=\frac{z^{1-a}}{\log z}+O\bigl(\frac{z^{1-a}}{(\log z)^2}\bigr)$. We can also apply this estimate to the integrand above to get
\[a\int_1^z\pi(t)t^{-1-a}\mathrm{d}t=a\int_1^z\frac{t^{-a}}{\log t}+O\Bigl(\frac{t^{-a}}{(\log t)^2}\Bigr)\mathrm{d}t=a\frac{z^{1-a}}{(1-a)\log z}+O\Bigl(\frac{z^{1-a}}{(\log z)^2}\Bigr).\]
Hence
\begin{equation}\label{abelp^}
\!\!\!\smash{\sum_{p\le z}\frac1{p^a}= \frac{z^{1-a}}{\log z}+a\frac{z^{1-a}}{(1-a)\log z}+O\Bigl(\frac{z^{1-a}}{(\log z)^2}\Bigr) = \frac{z^{1-a}}{(1-a)\log z}+O\Bigl(\frac{z^{1-a}}{(\log z)^2}\Bigr).}
\end{equation}
Substitute $\sigma$ to get $\sum_{p\le z}\frac1{p^\sigma}\ll\frac{z^{1-\sigma}}{(1-\sigma)\log z}$.
Now, $\sigma = 1-\frac{\log(L\log L)}{\log z}$ so $1-\sigma = \frac{\log(L\log L)}{\log z}$. Hence $\frac{z^{1-\sigma}}{(1-\sigma)\log z}=\frac{z^{1-\sigma}}{\log(L\log L)}$. But $z^{1-\sigma}=z^{\frac{\log(L\log L)}{\log z}}$ and $z^{\frac1{\log z}}=e$ so $z^{1-\sigma}=\exp(\log(L\log L))=L\log L$. Thus $\frac{z^{1-\sigma}}{(1-\sigma)\log z}=L\frac{\log L}{\log(L\log L)}\ll L$.

Finally, note that $x^\sigma=\dfrac{x}{x^\frac{\log(L\log L)}{\log z}}$ and $\log z = \frac1L\log x$ so $x^\frac1{\log z}=\exp(L)$. Hence 
\[x^\sigma = \frac x{\exp(L\log(L\log L))} = \frac x{(L\log L)^L}.\]
We conclude:
\[\smash{x^\sigma\prod_{p\le z}\left(1+\frac{r_0^*(p)^m}{p^\sigma}+\dots\right)\ll \frac x{(L\log L)^L}\exp\Bigl(\sum_{p\le z}\frac{2^m}{p^\sigma}\Bigr)\ll \frac {x\exp(L)^{O_m(1)}}{(L\log L)^L}\ll_m \frac x{L^L}.}\]

This is indeed $O(x(\log x)^{-1000})$.

\textbf{Case 2:}
Now, suppose $P(n)>z$ and $P(n)^2\mid n$. In particular, $P(n)\le\sqrt{x}$. Considering possible values of $P(n)=:l$, we can write
\begin{equation}\label{case2}
\smash[t]{\sum_{{\substack{n\le x\\ P(n)> z,P(n)^2\mid n}}}r_0^*(n)^m =
	\sum_{{\substack{l\le \sqrt{x}\\ l> z\\l\text{ prime}}}}
	\sum_{{\substack{n\le x\\P(n)=l\\l^2\mid n}}}r_0^*(n)^m =
	 \sum_{{\substack{l\le \sqrt{x}\\ l> z\\l\text{ prime}}}}
	\sum_{{\substack{P(k)<l\\l^2k\le x}}}r_0^*(l^2k)^m.}
\end{equation}
The last equality holds by definition of $P(n)$. Also, $l$ is prime and $k$ is a product of primes strictly smaller than $l$, so $(k,l)=1$. As $r_0^*$ is multiplicative,~\eqref{case2} gives
\begin{equation*}
	\sum_{{\substack{n\le x\\ P(n)> z,P(n)^2\mid n}}}r_0^*(n)^m =
	\sum_{{\substack{l\le \sqrt{x}\\ l> z\\l\text{ prime}}}}r_0^*(l^2)^m
	\sum_{{\substack{P(k)<l\\k\le x/l^2}}}r_0^*(k)^m\ll
	\sum_{{\substack{z<l\le \sqrt{x}\\l\text{ prime}}}}
	\sum_{{\substack{P(k)<l\\k\le x/l^2}}}r_0^*(k)^m\Bigl(\frac{x}{l^2k}\Bigr)^{\sigma_l}.
\end{equation*}
The last relation is because $r_0^*(l^2)\le 2$, so we can ignore the coefficients $r_0^*(l^2)^m$ as they are bounded by the constant $2^m$; $\sigma_l$ are some positive numbers, and $n=l^2k\le x$.

Similarly to case $1$, for any given $l$, the corresponding summand in the final expression above can be written as 
\[\smash[t]{\Bigl(\frac{x}{l^2}\Bigr)^{\sigma_l}\prod_{p< l}\Bigl(1+\frac{r_0^*(p)^m}{p^{\sigma_l}}+\dots\Bigr)\le\Bigl(\frac{x}{l^2}\Bigr)^{\sigma_l}\exp\Bigl(\sum_{p< l}\frac{2^m}{p^{\sigma_l}}\Bigr).}\]
We use~\eqref{abelp^}, replacing $z$ with $l$, to obtain $\displaystyle\sum_{p< l}\frac{2^m}{p^{\sigma_l}}\ll 2^m \dfrac{l^{1-\sigma_l}}{(1-\sigma_l)\log l}\ll_m \dfrac{l^{1-\sigma_l}}{(1-\sigma_l)\log l}$.

Let $\tau$ be some positive constant, and set $\sigma_l = 1- \frac{\log\tau}{\log l}$, so $1-\sigma_l = \frac{\log\tau}{\log l}$; thus we can rewrite the above as \[\smash{\sum_{p< l}\frac{2^m}{p^{\sigma_l}}\ll_m \dfrac{l^{1-\sigma_l}}{\log\tau}=\dfrac{l^{\log\tau/\log l}}{\log\tau}=\dfrac{\tau}{\log\tau},}\text{ which is constant.}\] So $\exp\Bigl(\sum_{p\le l}\frac{2^m}{p^{\sigma_l}}\Bigr)\ll_m 1$.

For the total sum we get 
\begin{align*}
	\sum_{{\substack{n\le x\\ P(n)> z,P(n)^2\mid n}}}\!\!\!\!\!\!\!\! r_0^*(n)^m
	&\ll	\smash[t]{\sum_{{\substack{z<l\le \sqrt{x}\\l\text{ prime}}}}\Bigl(\frac{x}{l^2}\Bigr)^{\sigma_l}
	=	\sum_{{\substack{z<l\le \sqrt{x}\\l\text{ prime}}}}\frac{x}{l^2}\left(\frac{l^2}{x}\right)^{\frac{\log\tau}{\log l}}=	\sum_{{\substack{z<l\le \sqrt{x}\\l\text{ prime}}}}\frac{\tau^2}{l^2}\cdot x^{1-\frac{\log\tau}{\log l}}}\\
	&\ll\smash[b]{\sum_{{\substack{z<l\le \sqrt{x}\\l\text{ prime}}}}\frac{x^{1-\frac{\log\tau}{\log l}}}{l^2}.}
\end{align*}
Thus, to prove the stated bound we need $\displaystyle \sum_{{\substack{z<l\le \sqrt{x}\\l\text{ prime}}}}\frac{x^{-\frac{\log\tau}{\log l}}}{l^2}\ll\frac1{(\log x)^{1000}}$. We again use the method of partial summation to do this: as previously, let $a_l\in\mathbb{C}$ be $1$ for prime $l$ and $0$ else. 
We now set $f(l)=\dfrac{x^{1-\frac{\log\tau}{\log l}}}{l^2}$. Then, since $l\ge z=x^{1/L}\Rightarrow \frac{\log x}{(\log l)^2}=O(1)$, we have \[\smash[t]{f'(l)=\dfrac{x^{-\frac{\log\tau}{\log l}}}{l^3}\Bigl(-2+\frac{(\log\tau)(\log x)}{(\log l)^2}\Bigr)\ll\dfrac{x^{1-\frac{\log\tau}{\log l}}}{l^3}.}\] Thus
\[\smash{\sum_{{\substack{z<l\le \sqrt{x}\\l\text{ prime}}}}\frac{x^{-\frac{\log\tau}{\log l}}}{l^2}=\pi(\sqrt{x})\cdot\frac{x^{-\frac{\log\tau}{\log \sqrt{x}}}}{(\sqrt{x})^2}-\pi(z)\cdot\frac{x^{-\frac{\log\tau}{\log z}}}{z^2}-\int_{z}^{\sqrt{x}}\pi(t)f'(t)dt.}\]

We calculate that 
\begin{align*}
\int_{z}^{\sqrt{x}}\pi(t)f'(t)dt&\ll\int_{z}^{\sqrt{x}}\pi(\sqrt{x})x^{-\frac{\log\tau}{\log t}}\cdot\frac{dt}{t^3}\ll\pi(\sqrt{x})x^{-\frac{\log\tau}{\log x}}\int_{z}^{\sqrt{x}}\frac{dt}{t^3}\\
&	\ll \frac{x^{\frac12-\frac{\log\tau}{\log x}}}{\log x}\cdot\frac1x\ll\frac1{(\log x)^{1000}}.
\end{align*}

Now, note that $\pi(\sqrt{x})\cdot\frac{x^{-\frac{\log\tau}{\log \sqrt{x}}}}{(\sqrt{x})^2}\ll\frac{\sqrt{x}}{\log x}\cdot x^{-1-\frac{2\log\tau}{\log x}}\ll\frac1{\sqrt{x}\log x}\ll\frac1{(\log x)^{1000}}$. Also, $\pi(z)\cdot\frac{x^{-\frac{\log\tau}{\log z}}}{z^2}\ll\frac1{z(\log z)^2}x^{-\frac{L\log\tau}{\log x}}\ll\frac{L^2}{x^{1/L}(\log x)^2}\ll\frac1{(\log x)^{1000}}$. So we are done.
%
\end{proof}
\section{The function $r_1^*$}
We also briefly touch on $r_1^*$, providing another result from Granville, Sabuncu and Sedunova's paper~\cite{granville2023} that will be useful later on. The lemma below shows that the majority of solutions to $n=a^2+p^2$, where $p$ is prime, satisfy $p\nmid a$. The proof follows the method in~\cite{granville2023}.
\begin{lemma}\label{lemma1}
	The number of pairs $(a,p)$ such that $a^2+p^2\le x$, $p\mid a$ is
	\[\smash{\sum_{n\le x}(r_1(n)-r_1^*(n))=O(\sqrt{x}\log\log x).}\]
\end{lemma}
\begin{proof}
	Since $p\mid a$, we can write $a=pm$, for some integer $m$; moreover, $a^2+p^2=p^2(m^2+1)\le x$ so $m^2+1\le\frac{x}{p^2}\Rightarrow m\le\frac{\sqrt{x}}{p}$.
Also, for any possible value of $m$, the possible choices of $p$ are restricted to primes up to $\sqrt{x}$ as $p^2\le\frac{x}{m^2+1}\le x\Rightarrow p\le\sqrt{x}$. Thus the total number of pairs $(a,p)$ satisfying the conditions is $\displaystyle\sum_{p\le\sqrt{x}}\#\{m:p^2(m^2+1)\le x\}\le\sum_{p\le\sqrt{x}}\frac{\sqrt{x}}{p}=\sqrt{x}\sum_{p\le\sqrt{x}}\frac1{p}$.

We use partial summation to estimate the sum, and apply~\eqref{partsum} with $a=1$ to get
\begin{align*}\smash{\sum_{p\le \sqrt{x}}\frac1p} &= \pi(\sqrt{x})\cdot\frac1{\sqrt{x}}+\smash[t]{\int_1^{\sqrt{x}}\pi(t)t^{-2}\mathrm{d}t=\frac2{\log x}+\int_1^{\sqrt{x}}\frac1{t\log t}\mathrm{d}t}\\
&=O(1)+\bigl[\log\log t\bigr]_1^{\sqrt{x}}
=O(\log\log x).\end{align*}
Hence 	$\sum_{n\le x}(r_1(n)-r_1^*(n))=\sqrt{x}\cdot O(\log\log x)=O(\sqrt{x}\log\log x)$, as required.
\end{proof}
\chapter{Moments of $r_1$}\label{chapter: moments of r1}\vspace{-2ex}
In this chapter, we look at important results from~\cite{granville2023} and~\cite{daniel_2001} which focus on asymptotics of the second and higher moments of $r_1$.
\section{Sums of reciprocals}
In order to facilitate our proof of various results in Chapter~\ref{chapter: moments of r1}, we first provide the following well-known results, which will be used for bounding estimates in the proof of some subsequent Lemmata and Theorem~\ref{thm:r1cl upper bound}.
\begin{lemma}[Mertens' Second Theorem]\label{lemma: Mertens}
	There exists $M\in\mathbb{R}$ such that $\forall x\ge 2$: \[\smash{\sum_{p\le x}\frac1p=\log\log x+M+O\Bigl(\frac1{\log x}\Bigr).}\]
\end{lemma}
\begin{lemma}\label{lemma: sum of 1n}
	$\displaystyle\sum_{n\le x}\frac1n = \log x + \gamma + O(\tfrac1x)$, where $\gamma$ is the Euler-Mascheroni constant.
\end{lemma}
We will also require a theorem about sums of reciprocals of primes in arithmetic progressions, a stronger result than Lemma~\ref{lemma: Mertens}, proved by Mertens in his paper ``Ein Beitrag zur analytischen Zahlentheorie''. I provided a partial proof of the theorem in the special case that it is applied in, although the general case is true.
\begin{theorem}[Mertens' Theorem for Primes in Arithmetic Progressions]\label{theorem: mertens APs}
	Let $q\in\mathbb{N}$, and $a$ an integer coprime to $q$.
	There exists $M(q,a)\in\mathbb{R}$ a constant dependent only on $q$ and $a$ such that $\forall x\ge 2$: \[\sum_{{\substack{p\le x\\ p\equiv a\pmod q}}}\frac1p=\frac1{\phi(q)}\cdot \log\log x+M(q,a)+O\Bigl(\frac1{\log x}\Bigr).\]
\end{theorem}
\begin{proof}
	We only give a proof for $q=4$, although the general case is true.
	
	\textbf{Step 1:} We show that the series $\sum\frac{\chi(p)}p$ converges, where $\chi$ is the non-principal Dirichlet character modulo $4$ -- there is only one, with $\chi(2j)=0, \chi(4j\pm 1)=\pm 1$. Note that $\displaystyle L(1,\chi)=\sum_{n=1}^{\infty}\frac{\chi(n)}n=\sum_{j=0}^{\infty}\frac{(-1)^j}{2j+1}=1-\frac13+\frac15-\frac17+\dots$ converges conditionally, by the alternating series convergence test. The sum is nonzero and positive, by considering pairs of consecutive nonzero terms.
	
		We now consider the truncated sums of the series $\displaystyle L'(1,\chi)=\sum_{n=1}^{\infty}\frac{\chi(n)\log n}n=0-\frac{\log 3}3+\frac{\log 5}5-\frac{\log 7}7+\dots$, which is also conditionally convergent by the alternating series convergence test. Recall by Lemma~\ref{lambda} that $\log n=\sum_{k\mid n}\Lambda(k)$, so \[\smash{\frac{\chi(n)\log n}n=\frac1n\sum_{k\mid n}\chi(n)\Lambda(k)=\sum_{k\mid n}\frac{\chi(n/k)\chi(k)\Lambda(k)}n.}\] Here we used the total multiplicativity of $\chi$. Hence, we get
	\begin{align}
	\sum_{n=1}^N\frac{\chi(n)\log n}n &=\sum_{n=1}^N\sum_{{\substack{k\mid n\\km=n}}}\frac{\chi(m)\chi(k)\Lambda(k)}n
	=\sum_{{\substack{k,m\\km\le N}}}\frac{\chi(m)\chi(k)\Lambda(k)}{km}\notag \\
	&=\sum_{k\le N}\frac{\chi(k)\Lambda(k)}{k}\sum_{m\le N/k}\frac{\chi(m)}{m}\notag \\
	&=\sum_{k\le N}\frac{\chi(k)\Lambda(k)}{k}\Bigl(L(1,\chi)+O\bigl(\tfrac kN\bigr)\Bigr).\label{eq:partial sum chi(n)log(n) over n}
	\end{align}
	The last equality comes from evaluating the tail of $ L(1,\chi)=\sum_{n\ge 1}\frac{\chi(n)}n$. Note that $\displaystyle \sum_{k\le N}\frac{\chi(k)\Lambda(k)}{k}O\bigl(\tfrac kN\bigr)=O\Bigl(\tfrac1N\sum_{k\le N}|\chi(k)\Lambda(k)|\Bigr)=O\Bigl(\tfrac1N\sum_{k\le N}\Lambda(k)\Bigr)=O\bigl(\tfrac{\psi(N)}N\bigr)=O(1)$, by definition of the Chebyshev function $\psi$ and its asymptotic behaviour, $\psi(x)\asymp x$ (this is known as Chebyshev's estimate). We let $C>0$ be a constant so that $\displaystyle -C\le\sum_{k\le N}\frac{\chi(k)\Lambda(k)}{k}O\bigl(\tfrac kN\bigr)\le C$ for all $N\ge N_0$, some $N_0$. Substituting into~\eqref{eq:partial sum chi(n)log(n) over n} gives \[\forall N\ge N_0\colon\smash[t]{\sum_{k\le N}\frac{\chi(k)\Lambda(k)}{k}\cdot L(1,\chi)-C\le\sum_{n=1}^N\frac{\chi(n)\log n}n\le\sum_{k\le N}\frac{\chi(k)\Lambda(k)}{k}\cdot L(1,\chi)+C.}\]
	Now, $L'(1,\chi)$ is a finite constant, so all the partial sums $\sum_{n=1}^N\frac{\chi(n)\log n}n$ are uniformly bounded. We thus get $\displaystyle \forall N\ge N_1\colon-C'\le \sum_{k\le N}\frac{\chi(k)\Lambda(k)}{k}\cdot L(1,\chi)\le C'$, for $N_1$ and $C'$ appropriately chosen integer and positive real constants. $L(1,\chi)\ne 0$ is a constant so for $N\ge N_1$ and $C''=\frac{C'}{L(1,\chi)}$ we have
	\begin{align*}
	-C''\le\sum_{k\le N}\frac{\chi(k)\Lambda(k)}{k}&=\sum_{p\le N}\frac{\chi(p)\Lambda(p)}{p}+\sum_{{\substack{q\le N\\q=p^e,e\ge 2}}}\frac{\chi(q)\Lambda(q)}{q}\\
	&=\sum_{p\le N}\frac{\chi(p)\log p}{p}+\sum_{{\substack{p^e\le N\\e\ge 2}}}\frac{\chi(p^e)\log p}{p^e}\le C''.
	\end{align*}
	Since $\displaystyle\sum_{{\substack{p^e\le N\\e\ge 2}}}\left|\frac{\chi(p^e)\log p}{p^e}\right|\le \sum_{{\substack{p^e\\e\ge 2}}}\frac{\log p}{p^e}=\sum_{p\text{ prime}}\frac{\log p}{p(p-1)}:=D$, where $D$ is a finite positive constant, we get from the above that $-D'\le \sum_{p\le N}\frac{\chi(p)\log p}{p}\le D'$ for all $N\ge N_1$, some finite $D'>0$.
	We now apply the Dirichlet convergence test: Let $a_n$ be the sequence $\frac1{\log n}$ of real numbers converging to $0$ monotonically, and \vspace{-1ex} 
	\[b_n:=
	\begin{cases}
		\frac{\chi(p)\log p}{p}, n=p \text{ is prime} \\
		0 \text{ else}
	\end{cases},\]\vspace{-1ex} so that $|\sum_{n=1}^Nb_n|\le D'$ for all $N\ge N_1$. Then the series $\displaystyle\sum_{n\ge 1} a_nb_n=\sum_{p\text{ prime}}\!\frac{\chi(p)}p$ converges.
	
	\textbf{Step 2:} We now deduce the result. Since $\sum_{p\text{ prime}}\frac{\chi(p)}p$ converges, in particular we have $\sum_{p\le x}\frac{\chi(p)}p=O(1)$ for all $x$. Thus
	\begin{equation}\label{eq: p 1 mod 4 v 3 mod 4}
		\sum_{{\substack{p\le x\\ p\equiv 1\!\!\!\pmod 4}}}\!\!\!\!\!\!\frac1p-\sum_{{\substack{p\le x\\ p\equiv 3\!\!\!\pmod 4}}}\!\!\!\!\!\!\frac1p=O(1)\implies
		 \sum_{{\substack{p\le x\\ p\equiv 1\!\!\!\pmod 4}}}\!\!\!\!\!\!\frac1p=\sum_{{\substack{p\le x\\ p\equiv 3\!\!\!\pmod 4}}}\!\!\!\!\!\!\frac1p+O(1).\end{equation}
	Mertens' Second Theorem, Lemma~\ref{lemma: Mertens}, gives $\sum_{p\le x}\frac1p=\log\log x+M+O(\frac1{\log x})$. Substituting~\eqref{eq: p 1 mod 4 v 3 mod 4} into this, we get $\displaystyle\!\!\!\sum_{{\substack{p\le x\\ p\equiv 1\pmod 4}}}\frac1p=\frac12\log\log x+ M/2+O\Bigl(\frac1{\log x}\Bigr)+O(1)$, and the same for primes that are $3$ mod $4$. Let $M(q,a)=M/2+O(1)$ to get the required result.
	\end{proof}
\begin{remark}
	A similar proof with extra steps works for other $q$. It is true in general that for non-principal Dirichlet characters $\chi$, the Dirichlet L-function $L(s,\chi):=\sum_{n=1}^{\infty}\frac{\chi(n)}{n^s}$ and $L'(s,\chi)=\frac{\mathrm{d}}{\mathrm{d}s}L(s,\chi)$ converge for $\Re(s)>0$, in particular $s=1$, and $L(1,\chi)\ne 0$. We consider linear combinations of the non-principal Dirichlet characters to get the required result.
\end{remark}
%
\section{Integers with a fixed number of prime factors}
We include this section here as it will be needed in the proof of Theorem~\ref{thm:r1cl upper bound}.

We would like to obtain an estimate on $\rho_k(x):=\#\{n\le x\colon \omega^*(n)=k\}$. Here, $\omega^*(n)$ is the number of odd prime factors of $n$. We provide some necessary notation below. \vspace{-2ex}
\begin{notation}\
	\begin{itemize}
		\item $\mathcal{N}:=\{n\in\mathbb{N}\colon 4\nmid n,\text{ and } p\nmid n\text{ for all primes }p\equiv 3\pmod 4\}$.
		\item $\mathcal{N}_k(x):=\{n\le x\colon n\in\mathcal{N}, \omega^*(n)=k\}$
		\item $\rho_{k,\mathcal{N}}(x):=\#\{n\le x\colon n\in\mathcal{N}, \omega^*(n)=k\}$
	\end{itemize}
\end{notation}
We first prove a preliminary result about bounding a sum which will be very useful to us in Lemma~\ref{lemma: numbers with k factors}.
\begin{claim}\label{claim: bound for inductive step}
	$\displaystyle\!\!\!\!\!\!\!\!\!\!\sum_{{\substack{q=p^e,\text{ a prime power}\\q\le\sqrt{x}\\ p\equiv 1\pmod 4}}}\frac{x}{q\log(x/q)}\le\frac{x}{\log x}(\tfrac12\log\log(x)+\gamma_2)$, some constant $\gamma_2$.\end{claim}
\begin{proof}
	We prove that $\displaystyle\sum_{{\substack{q=p^e,\text{ a prime power}\\q\le \sqrt{x}\\ p\equiv 1\pmod 4}}}\frac{\log x}{q\log(x/q)}\le\tfrac12\log\log(x)+\gamma_2$. (Multiply by $\frac{x}{\log x}$ to get the required bound for the claim.) 
	
	
	Note that the contribution of prime powers $q=p^e$ where $e\ge 2$ is small. Indeed, $\displaystyle\sum_{n=2}^{\infty}\sum_{p<x}\frac1{p^n}=\sum_{p<x}\sum_{n=2}^\infty\frac1{p^n}=\sum_{p<x}\frac1{p^2}\frac1{1-\frac1p}=\sum_{p<x}\frac1{p(p-1)}<1$. Since $q\le\sqrt{x}$, we have $\log x\le 2\log(x/q)$, so
	\[\smash[t]{\sum_{{\substack{q=p^e,\text{ a prime power}\\q\le \sqrt{x},\ e\ge 2\\ p\equiv 1\pmod 4}}}\frac{\log x}{q\log(x/q)}\le \sum_{{\substack{q=p^e,\text{ a prime power}\\q\le \sqrt{x}\\e\ge 2}}}\frac{\log x}{q\log(x/q)}\le 2\sum_{{\substack{q=p^e,\text{ a prime power}\\q\le \sqrt{x}\\e\ge 2}}}\frac1q<2.}\]
Thus, \[\smash[t]{\sum_{{\substack{q=p^e,\ q\le \sqrt{x}\\ p\equiv 1\pmod 4}}}\frac{\log x}{q\log(x/q)}\le 2+\sum_{{\substack{\\p\le \sqrt{x}\\ p\equiv 1\pmod 4}}}\frac{\log x}{p\log(x/p)}.}\]
	
\vspace{-1ex} Now, for $p\le x^{1/n}$ we have $\log(x/p)\ge\log(x^{1-1/n})=\frac{n-1}n\log x$, so $\dfrac{\log x}{\log(x/p)}\le\dfrac{n}{n-1}$. We thus obtain the bound \vspace{-1ex}
	\begin{align}
		&\sum_{{\substack{p\le x^{1/n}\\ x^{1/(n+1)}\le p\\ p\equiv 1\pmod 4}}}\!\!\!\!\frac{\log x}{p\log(x/p)} \le
		\frac{n}{n-1}\!\!\!\!\sum_{{\substack{p\le x^{1/n}\\x^{1/(n+1)}\le p\\ p\equiv 1\pmod 4}}}\!\!\!\!\frac{1}{p}\\
		&=
		\frac{n}{n-1}\cdot\frac12\bigl(\log\log(x^{\frac1n})-\log\log(x^{\frac1{n+1}})+O(\tfrac1{\log x})\bigr) \notag\\
		&=\frac{n}{n-1}\cdot\frac12\Bigl(\log\bigl(\tfrac1n\log x\bigr)-\log\bigl(\tfrac1{n+1}\log x\bigr)+O(\tfrac1{\log x})\Bigr) \notag\\
		&=\frac{n}{n-1}\cdot\frac12\log\Bigl(\frac{\tfrac1n\log x}{\tfrac1{n+1}\log x}\Bigr)+O(\tfrac1{\log x}) = \frac{n}{n-1}\cdot\frac12\log\Bigl(\frac{n+1}n\Bigr)+O(\tfrac1{\log x}).\label{p in n to n+1}
	\end{align} Here we used Theorem~\ref{theorem: mertens APs} to sum $\frac1p$ for primes $p$ congruent to $1$ modulo $4$ in the relevant interval. In this case we have $q=4$, so $\phi(q)=2$, and $a=1$.
	Now \vspace{-1ex} 
	\begin{equation}\label{total sum over p}
		\sum_{{\substack{p\le \sqrt{x}\\ p\equiv 1\pmod 4}}}\frac{\log x}{p\log(x/p)} = O(\tfrac1{\log x})+\sum_{n=2}^{\lfloor \log x\rfloor} \sum_{{\substack{x^{1/(n+1)}< p\\ p\le x^{1/n}\\ p\equiv 1\pmod 4}}}\frac{\log x}{p\log(x/p)},
	\end{equation} as we split the range between $e=x^{1/\log x}$ and $\sqrt{x}=x^{1/2}$ into intervals $[x^{1/(n+1)}, x^{1/n}]$ and $\frac{\log x}{2\log(x/2)}=O(\tfrac1{\log x})$. (2 is the only prime which falls outside of our intervals.)
	
	We substitute~\eqref{p in n to n+1} into~\eqref{total sum over p} to get 
	\begin{align*}\label{total sum}
		\smash{\sum_{{\substack{p\le \sqrt{x}\\ p\equiv 1\pmod 4}}}\frac{\log x}{p\log(x/p)}} &\le O(\tfrac1{\log x})+\sum_{n=2}^{\lfloor \log x\rfloor} \frac{n}{2(n-1)}\log\Bigl(\frac{n+1}n\Bigr)+O(1) \\
		&=\sum_{n=2}^{\lfloor \log x\rfloor} \frac{n}{2(n-1)}\log\Bigl(1+\frac1n\Bigr)+O(1) \\
		&\le\sum_{n=2}^{\lfloor \log x\rfloor} \frac{n}{2(n-1)}\cdot\frac1n+O(1)
		=\frac12\sum_{j=1}^{\lfloor \log x\rfloor-1} \frac{1}{j}+O(1) \\
		& \le \tfrac12\log\log x + \tfrac12\gamma + O(1).
	\end{align*}
In the last line we applied Lemma~\ref{lemma: sum of 1n}. The $O(1)$ term represents a constant, so $\displaystyle\sum_{{\substack{q=p^e,\text{ a prime power}\\q\le \sqrt{x},\ p\equiv 1\pmod 4}}}\frac{\log x}{q\log(x/q)}\le 2+\tfrac12\log\log x+\tfrac12\gamma+c$, for some constant $c$.
Setting $\gamma_2\ge 2+c+\frac{\gamma}2$ gives $\displaystyle\sum_{{\substack{q=p^e,\text{ a prime power}\\q\le\sqrt{x},\ p\equiv 1\pmod 4}}}\frac{x}{q\log(x/q)}\le\frac{x}{\log x}(\tfrac12\log\log(x)+\gamma_2)$.
\end{proof}

Now we obtain an upper bound on the function $\rho_{k,\mathcal{N}}(x)$.
\begin{lemma}\label{lemma: numbers with k factors}
$\displaystyle\exists\gamma_1,\gamma_2\colon\rho_{k,\mathcal{N}}(x)\le\gamma_1\frac{x}{\log x}\frac{(\frac12L+\gamma_2)^{k-1}}{(k-1)!},\ \forall x\ge 1.$
\end{lemma}
\begin{proof} (See~\cite[\S 13.4]{granvillebook})
We will prove the following statement by induction on $k$: ``There exist $\gamma_1,\gamma_2>0$ such that for all $x\ge 1$ and  $k\in\mathbb{N}\colon\rho_{k,\mathcal{N}}(x)\le\gamma_1\frac{x}{\log x}\frac{(\frac12L+\gamma_2)^{k-1}}{(k-1)!}$.''
	
\textbf{Base case:} $k=1$. Now, $\rho_{1,\mathcal{N}}(x)$ is the number of prime powers that are at most $x$. Note that for a fixed integer $j$, the number of primes $p$ such that $p^j\le x$ is $\pi(x^{1/j})$, where $\pi$ is the prime counting function. Hence
\[\rho_{1,\mathcal{N}}(x)=\sum_{j\ge 1}\pi(x^{\frac1j})=\sum_{j\ge 1}^{\frac{\log x}{\log 2}}\pi(x^{\frac1j})\le\sum_{j= 1}^{\frac{\log x}{\log 2}}\frac{x^{\frac1j}}{\frac1j\log x}\le\frac{x}{\log x}+\sum_{j=2}^{\frac{\log x}{\log 2}}\frac{x^{\frac12}}{\frac12\log x}.\] The last inequality holds as $\pi$ is an increasing function.
Now, the number of $2\le j\le\frac{\log x}{\log 2}$ is bounded by $\log x$, so $\rho_{1,\mathcal{N}}(x)\le\frac x{\log x}+\log x\cdot\frac{x^{\frac12}}{\frac12\log x}=\frac x{\log x}+2\sqrt{x}\le 2\frac x{\log x}$, for $x>c$ (where $c$ is a constant independent of $x$). For $1\le x\le c$, the function $f(x)=1+2\frac{\sqrt{x}\log x}x$ is continuous and hence bounded, so $\exists\gamma\colon \frac x{\log x}+2\sqrt{x}\le\gamma \frac x{\log x}$, for all $x\ge 1$.

Thus the base case is proved; any $\gamma_1\ge\gamma$ works, so some $\gamma_1$ exists. We do not obtain any restriction on $\gamma_2$, so we pick a $\gamma_2$ satisfying the conditions of Claim~\ref{claim: bound for inductive step}.

\textbf{Inductive step:} Suppose that the statement holds for $k-1$; we now prove for $k$. Note that $\rho_{k,\mathcal{N}}(x)$ counts the number of integers $n$ with prime factorisation $n=2^{e_0}p_1^{e_1}\dots p_k^{e_k}$, where $p_i$ are primes that are $1\pmod 4$, and $e_0\in\{0,1\}$ by definition of $\mathcal{N}$. Without loss of generality, we may suppose that $p_1^{e_1}<\dots<p_k^{e_k}$. For each odd prime $p_j$, we can write $n=p_j^{e_j}m_j$, where $m_j\in\mathcal{N}$ and $\omega^*(m_j)=k-1$, and $p_j^{e_j}\mathrel\Vert n$ (i.e. $p_j^{e_j}\mid n$ and $p_j^{e_j+1}\nmid n$). Hence, given a prime power $p_j^{e_j}$,
\begin{equation}\label{eq: count mj}
\#\{n\le x\colon n\in\mathcal{N}_k(x),p_j^{e_j}\mathrel\Vert n\}=\rho_{k-1,\mathcal{N}}\Bigl(\dfrac{x}{p_j^{e_j}}\Bigr).
\end{equation}
For $j<k$, we further note that $p_j^{e_j}<p_k^{e_k}$ and $p_j^{e_j}p_k^{e_k}\le n$, so $(p_j^{e_j})^2<n\implies p_j^{e_j}<\sqrt{n}$. Since $n\le x$, for all $j=1,\dots,k-1$ we have $p_j^{e_j}\le \sqrt{x}$. We will use this to get an upper bound for $\rho_{k,\mathcal{N}}(x)$ using~\eqref{eq: count mj}. Note that $(k-1)\rho_{k,\mathcal{N}}(x)$ can be interpreted as multiple counting every $n\in\mathcal{N}_k(x)$ from the point of view of each of its odd prime power factors except the largest one (corresponding to $p_k^{e_k}$ above); each of those odd prime powers is at most $\sqrt{x}$, as shown. Thus
\[(k-1)\rho_{k,\mathcal{N}}(x)\le \sum_{{\substack{q=p^e,\text{ a prime power}\\q\le\sqrt{x}\\ p\equiv 1\pmod 4}}}\rho_{k-1,\mathcal{N}}\Bigl(\dfrac{x}{q}\Bigr).\]
We now apply the inductive hypothesis to each term in the summation to get
\begin{align*}
(k-1)\rho_{k,\mathcal{N}}(x) &\le
\sum_{{\substack{q=p^e,\text{ a prime power}\\q\le \sqrt{x}\\ p\equiv 1\pmod 4}}}\gamma_1\frac{x/q}{\log(x/q)}\frac{(\frac12\log\log(x/q)+\gamma_2)^{k-2}}{(k-2)!} \\
&\le\sum_{{\substack{q=p^e,\text{ a prime power}\\q\le \sqrt{x}\\ p\equiv 1\pmod 4}}}\gamma_1\frac{x}{q\log(x/q)}\ \cdot\ \frac{(\frac12\log\log(x)+\gamma_2)^{k-2}}{(k-2)!}\\
&\le \gamma_1\frac{(\frac12\log\log(x)+\gamma_2)^{k-1}}{(k-2)!},\text{ by Claim~\ref{claim: bound for inductive step}}\\
\implies \rho_{k,\mathcal{N}}(x) &\le\frac1{k-1}\cdot\gamma_1\frac{(\frac12\log\log(x)+\gamma_2)^{k-1}}{(k-2)!}=\gamma_1\frac{(\frac12L+\gamma_2)^{k-1}}{(k-1)!}.
\end{align*}
This completes the inductive step, and hence we have finished the proof.
\end{proof}

\section{Upper bounds on moments of $r_1$}
In this section we prove the following theorem:
\begin{theorem}[Granville, Sabuncu, Sedunova, 2024]\label{thm:r1cl upper bound}
	For any fixed integer $\ell\ge 1$ and integer $k\ll\log\log x =: L$ we have
	\[\sum_{{\substack{n\le x\\\omega^*(n)=k}}}\binom{r_1(n)}{\ell}\ll_\ell\frac{xL^{O_\ell(1)}}{(\log x)^{\ell+1}}\frac{(2^{\ell-1}L)^k}{k!}.\]
\end{theorem}
We provide a more detailed proof than in~\cite{granville2023}, breaking the process into several steps: in particular, we provide a sieve theory argument which starts at Claim~\ref{claim: sieve theory bound on summand} below.

We start by approximating the LHS by a function of $r_1^*$.
\begin{claim}\label{claim: approx r1 by r1*}
	\[\smash{\sum_{{\substack{n\le x\\\omega^*(n)=k}}}\left(\binom{r_1(n)}{\ell}-\binom{r_1^*(n)}{\ell}\right)\ll_\ell x^{\frac12+o(1)}.}\]
\end{claim}
\begin{proof}
$\displaystyle\sum_{{\substack{n\le x\\\omega^*(n)=k}}}\binom{r_1(n)}{\ell}-\sum_{{\substack{n\le x\\\omega^*(n)=k}}}\binom{r_1^*(n)}{\ell}=\sum_{{\substack{n\le x\\\omega^*(n)=k}}}\left(\binom{r_1(n)}{\ell}-\binom{r_1^*(n)}{\ell}\right)$. Now, since $\ell$ is a constant, so is $\ell!$; hence $\binom{r_1(n)}{\ell}\ll_\ell r_1(n)(r_1(n)-1)\dots (r_1(n)-\ell+1)\ll_\ell r_1(n)^\ell$. Thus we can write \[\sum_{{\substack{n\le x\\\omega^*(n)=k}}}\left(\binom{r_1(n)}{\ell}-\binom{r_1^*(n)}{\ell}\right)\ll_\ell\sum_{{\substack{n\le x\\\omega^*(n)=k}}}(r_1(n)-r_1^*(n))r_1(n)^{\ell-1}.\]
We will bound $r_1(n)$ rather crudely to achieve an acceptable estimate. Firstly, note that $r_1(n)\le r_0(n)$. Now, $r_0$ is a multiplicative function, and $r_0(p^k)\le k+1=\tau(p^k)$ for any prime $p$. Since the divisor counting function $\tau$ is also multiplicative, we get from~\eqref{eq: tau o(1)} that $r_1(n)\le\tau(n)\ll n^{o(1)}$. So for all $n\le x$ we can write $r_1(n)^{\ell-1}\ll_\ell x^{o(1)}$.

We apply Lemma~\ref{lemma1} to bound the sum of $r_1(n)-r_1^*(n)$; since $\log\log x\ll x^{o(1)}$, we get
\[\sum_{{\substack{n\le x\\\omega^*(n)=k}}}\left(\binom{r_1(n)}{\ell}-\binom{r_1^*(n)}{\ell}\right)\ll_\ell x^{\frac12+o(1)}.\qedhere\]
\end{proof}
Hence, to prove the theorem, we may now replace $r_1$ with $r_1^*$ in the statement, as we have shown the error accumulated through doing so is small. We next restrict the sum by introducing conditions on $P(n)$, the largest prime factor of $n$.
\begin{claim}\label{claim: bound for some P(n)}
	$\displaystyle\sum_{{\substack{n\le x\\\omega^*(n)=k \\ P(n)\le z\text{ or }P(n)^2\mid n}}}\binom{r_1(n)}{\ell}\ll_\ell \frac{xL^{O_\ell(1)}}{(\log x)^{\ell+1}}\frac{(2^{\ell-1}L)^k}{k!}.$
\end{claim}
\begin{proof}
Suppose $P(n)\le x^{1/L}$ or $P(n)^2\mid n$.
Lemma~\ref{r0*bound} bounds $m$-th moments of $r_0^*$ if $P(n)\le x^{1/L}=:z$ or $P(n)^2\mid n$, where  $L=\log\log x$. Note that the exponent of $1000$ was chosen arbitrarily in the proof of this Lemma, and indeed could be replaced with any constant, so we shall replace it with $\ell+2$ for our usage here. Since the bound in Lemma~\ref{r0*bound} holds for all powers $m$, we can also derive an asymptotic bound for any expression of form $\sum_{n\le x}f(r_0^*(n))$, for any polynomial $f$. In particular, a binomial coefficient is a polynomial. As $\forall n\colon r_1^*(n)\le r_0^*(n)$, we have $\binom{r_1^*(n)}{\ell}\le\binom{r_0^*(n)}{\ell}$, and so \vspace{-1ex}
\begin{align*}
&\sum_{{\substack{n\le x\\ P(n)\le z\text{ or }P(n)^2\mid n}}}\binom{r_1^*(n)}{\ell}\ll_\ell\frac{x}{(\log x)^{\ell+2}} 
\\
\implies \!\!\!&\smash{\sum_{{\substack{n\le x\\\omega^*(n)=k \\ P(n)\le z\text{ or }P(n)^2\mid n}}}\binom{r_1(n)}{\ell}\ll_\ell\frac{x}{(\log x)^{\ell+2}}+ x^{\frac12+o(1)}\ll_\ell\frac{xL^{O_\ell(1)}}{(\log x)^{\ell+1}}\frac{(2^{\ell-1}L)^k}{k!}.}\qedhere
\end{align*}
\end{proof}

From now on we may assume $P(n)>x^{1/L}$ and $P(n)^2\nmid n$. Let $p=P(n)$. We can write $n=mp$, with $p>x^{1/L}$ and $P(m)<p$; $P(m)\le p$ by definition of the largest prime factor of $n$, and the inequality is strict as $P(n)^2\nmid n$. Also, $n\le x$ so $m\le x^{1-1/L}$.
We may also assume that $n$ is such that $r_1^*(n)>0$ (as we get no contribution to the sum we are estimating from $n$ such that $r_1^*(n)=0$). In particular, $r_0^*(n)>0$, so $n$ has no prime factors that are $3$ modulo $4$.

Let $a,b$ be coprime such that $p=a^2+b^2$; the choice of $a,b$ is unique up to order. Let all pairs of two positive coprime squares whose sum is $m$ be denoted $u_i,v_i$. We have $1\le i\le r_0^*(m)$ by definition of $r_0^*$. By multiplicativity of $r_0^*$, all representations of $n$ as a sum of two positive coprime squares arise from representations of its factors, i.e. $n=(au_i+bv_i)^2+|bu_i-av_i|^2=(au_i+bv_i)^2+(bu_i-av_i)^2$ for $1\le i\le r_0^*(m)$ (in some order) -- this was explained in Lemma~\ref{r0*mult}. Out of these, we count representations where one of the squares is the square of a prime to get $r_1^*(n)=\#\{i:au_i+bv_i\text{ is prime}\}+\#\{i:|bu_i-av_i|\text{ is prime}\}$.

We now rewrite our original sum as a double sum over $m$ and over $p$. Summing over $n=mp\le x$ with $p=P(n)>x^{1/L}$ is equivalent to summing over $m\le x^{1-1/L}$, and $x^{1/L}<p\le x/m$, as $n=mp\le x$. Further, since $p$ is coprime to $m$, we get that $\omega^*(n)=k\iff \omega^*(m)=k-1$, as $m$ contains all prime factors of $n$ except $p$. Finally, as mentioned earlier, we may assume that $r_0^*(p),r_0^*(m)>0$, so $p$ and $m$ are both representable as a sum of two coprime squares. Thus we get \vspace{-1ex}
\[\smash[t]{\sum_{{\substack{n\le x\\ \omega^*(n)=k\\P(n)>x^{1/L}}}} \binom{r_1^*(n)}{\ell}=\sum_{{\substack{m\le x^{1-1/L}\\ \omega^*(m)=k-1\\m\in\mathcal{Q}}}}\sum_{{\substack{x^{1/L}<p\le x/m\\p>P(m)\\p\text{ prime, }p\in\mathcal{Q}}}} \binom{r_1^*(mp)}{\ell},}\]
where $\mathcal{Q}=\{q\in\mathbb{N}\colon\exists u,v \mid u^2+v^2=q,\ u,v\text{ coprime}\}$. By our earlier notation we have $m=u_i^2+v_i^2$ for some $i$.

Recall the set $\mathcal{L}_m$ of linear forms from Chapter 3. Here we have $r_0^*(m)$ different pairs $(u_i,v_i)$ so we define $\mathcal{L}_m=\{u_ix+v_iy,v_ix-u_iy\colon 1\le i\le r_0^*(m)\}$.

Recall that $r_1^*(mp)=\#\{i:au_i+bv_i\text{ is prime}\}+\#\{i:|bu_i-av_i|\text{ is prime}\}=\#\{\phi\in\mathcal{L}_m\colon\phi(a,b)\text{ is prime}\}$. Let $\mathcal{L}_{m,a,b}=\{\phi\in\mathcal{L}_m\colon\phi(a,b)\text{ is prime}\}$. Hence $\displaystyle\binom{r_1^*(mp)}{\ell}=\sum_{{\substack{I\subseteq\mathcal{L}_m\\|I|=\ell}}}\mathbbm{1}(\forall\phi\in I\colon\phi(a,b)\text{ is prime})=\sum_{{\substack{I\subseteq\mathcal{L}_m\\|I|=\ell}}}\mathbbm{1}(I\subseteq\mathcal{L}_{m,a,b})=\sum_{{\substack{I\subseteq\mathcal{L}_{m,a,b}\\|I|=\ell}}}1$, where $\mathbbm{1}$ is the indicator function.
Thus we can write
\begin{align}
\sum_{{\substack{m\le x^{1-1/L}\\ \omega^*(m)=k-1\\m\in\mathcal{Q}}}}
&\sum_{{\substack{x^{1/L}<p\le x/m\\p>P(m)\\p\text{ prime, }p\in\mathcal{Q}}}} \binom{r_1^*(mp)}{\ell}
= \sum_{{\substack{m\le x^{1-1/L}\notag\\ \omega^*(m)=k-1\\m\in\mathcal{Q}}}}\sum_{{\substack{x^{1/L}<p\le x/m\\p>P(m)\text{ prime}\\p\in\mathcal{Q},p=a^2+b^2,a<b}}} \sum_{{\substack{I\subseteq\mathcal{L}_{m,a,b}\\|I|=\ell}}}1\\
&= \sum_{{\substack{m\le x^{1-1/L}\\ \omega^*(m)=k-1\\m\in\mathcal{Q}}}}\sum_{{\substack{I\subseteq\mathcal{L}_{m}\\|I|=\ell}}}\sum_{{\substack{x^{1/L}<p\le x/m\\p>P(m)\text{ prime}\\p\in\mathcal{Q},p=a^2+b^2,a<b}}} \mathbbm{1}(I\subseteq\mathcal{L}_{m,a,b}).\label{3sum}
\end{align}
We continue by bounding 
	\begin{equation}\label{eq:sum to bound}
\sum_{{\substack{x^{1/L}<p\le x/m\\p>P(m)\text{ prime}\\p\in\mathcal{Q},p=a^2+b^2,a<b}}} \mathbbm{1}(I\subseteq\mathcal{L}_{m,a,b}) =
	\sum_{{\substack{x^{1/L}<p\le x/m\\p>P(m)\text{ prime}\\p\in\mathcal{Q},p=a^2+b^2,a<b}}} \mathbbm{1}(\phi(a,b)\text{ prime for all }\phi\in I),
\end{equation}
where $I\subseteq\mathcal{L}_m$ is fixed. Let $y=x^{1/10L}$. Then $p=a^2+b^2$ is coprime to $\prod_{q\le y\text{ prime}} q$, as $p>x^{1/L}>y$ and $p$ is prime. Also, given $\phi\in I$, we have $\phi(a,b)\le y$ or $\phi(a,b)>y$. We consider the related cases separately.

\begin{claim}\label{claim: phi small prime} \[\sum_{{\substack{x^{1/L}<p\le x/m\\p>P(m)\text{ prime}\\p\in\mathcal{Q},p=a^2+b^2,a<b}}} \mathbbm{1}(\forall\phi\in I\colon\phi(a,b)\text{ prime},\ \exists\phi\in I\colon\phi(a,b)\le y)\ll\dfrac{x}{m}\cdot\dfrac{L^{O_\ell(1)}}{(\log x)^{\ell+1}}.\]\end{claim}
\begin{proof}
Suppose for a particular $\phi\in I$, we can pick some $a,b$ such that $x^{1/L}<a^2+b^2\le\frac{x}{m}$ and $\phi(a,b)\le y$. There are two possibilities for $\phi$: either $\exists u_i,v_i$ such that $\phi = u_ix+v_iy$, or $\exists u_i,v_i$ such that $\phi = v_ix-u_iy$. In the first scenario, since $a^2+b^2>x^{1/L}=y^{10}$, we must have either $a^2$ or $b^2$ be at least $\frac12y^{10}$; suppose without loss of generality that this is $a^2$. Then $a>\frac1{\sqrt{2}}y^5>y$, so $\phi(a,b)=u_ia+v_ib\ge u_ia\ge a>y$. Hence, $\phi = v_ix-u_iy$ for some $i$.

We now check the maximum possible number of $a,b$ satisfying the required conditions such that $\phi(a,b)\le y$. Suppose $\phi(a,b)=\phi(a',b')=q$ for some $a,a',b,b'$ with $x^{1/L}<a^2+b^2,(a')^2+(b')^2\le\frac{x}{m}$, where $q$ is a prime smaller than $y$. We can rearrange the equation to get $v_i(a-a')=u_i(b-b')$. Now, since $u_i$ and $v_i$ are coprime by definition, we get $u_i\mid (a-a')$ and $v_i\mid (b-b')$. Let $c,d\in\mathbb{Z}$ be such that $a=a'+cu_i,b=b'+dv_i$. Additionally, $u_i$ and $v_i$ are both positive, hence from $v_i(a-a')=u_i(b-b')$ we get that $a-a'$ and $b-b'$ are the same sign; thus without loss of generality $c,d\in\mathbb{N}$.

Suppose $a^2+b^2=p$ and $(a')^2+(b')^2=p'$. Then $b^2-(b')^2+a^2-(a')^2=p-p'$; factorising the differences of two squares gives $p-p'=dv_i(b+b')+cu_i(a+a')$. As before, without loss of generality $a>\frac1{\sqrt{2}}x^{1/2L}$, and $a',c,u_i,d,v_i,b+b'\ge 1$, so $p-p'>\frac1{\sqrt{2}}x^{1/2L}$.

Since $p,p'\le\frac{x}{m}$, we get that the number of pairs $(a,b)$ such that $\phi(a,b)=q$ is bounded above by $\frac{x}{m}\cdot\frac{\sqrt{2}}{x^{1/2L}}$. Since $q\le y$ is any prime, we get
\[\sum_{{\substack{x^{1/L}<p\le x/m\\p>P(m)\text{ prime}\\p\in\mathcal{Q},p=a^2+b^2,a<b}}} \mathbbm{1}(\forall\phi\in I\colon\phi(a,b)\text{ prime},\ \exists\phi\in I\colon\phi(a,b)\le y)\le\frac{x}{m}\cdot\frac{\sqrt{2}}{x^{1/2L}}\cdot\pi(y).\]
Substituting $y=x^{1/10L}$ gives $\pi(y)\sim\frac{10Lx^{1/10L}}{\log x}$. Thus, it now remains to prove that
\begin{align*}
	\dfrac{x}{m}\cdot\dfrac{\sqrt{2}}{x^{1/2L}}\cdot\dfrac{10Lx^{1/10L}}{\log x}
	&\ll\dfrac{x}{m}\cdot\dfrac{L^{O_\ell(1)}}{(\log x)^{\ell+1}} \\
	\iff \dfrac{10\sqrt{2}L}{x^{4/10L}}
	&\ll\dfrac{L^{O_\ell(1)}}{(\log x)^{\ell}}
	\iff \frac1{x^{2/5L}} \ll \dfrac{L^{O_\ell(1)}}{(\log x)^{\ell}} \\
	\iff \dfrac{(\log x)^{\ell}}{L^{O_\ell(1)}} &\ll x^{2/5L}.
\end{align*}
This holds as $2/5L$ is positive, so $x^{2/5L}$ grows faster than any power of $\log x$ or $L$.
\end{proof}
We now seek to obtain a bound for the remaining part of~\eqref{eq:sum to bound}. Note that primes larger than $y$ are in particular coprime to all primes smaller than $y$, thus
\[
\sum_{{\substack{x^{1/L}<p\le x/m\\p>P(m)\text{ prime}\\p\in\mathcal{Q},p=a^2+b^2,a<b}}} \mathbbm{1}(\forall\phi\in I\colon\phi(a,b)\text{ prime},\ \phi(a,b)> y)
\le\sum_{{\substack{a^2+b^2\le x/m\\\gcd(a^2+b^2,\prod_{q\le y}q)=1\\\prod_{\phi\in I}\gcd(\phi(a,b),\prod_{q\le y}q)=1}}}\!\!\!\!\!\!\!\!\!\!\!\!\!\!\!\! 1\ =:\mathcal{C}(x).
\]
\begin{claim}\label{claim: sieve theory bound on summand}
\	$\displaystyle\mathcal{C}(x)
	\ll_\ell
	\frac x m\prod_{\ell+2<p\le y}\left(1-\frac{\ell+1+(\frac{-1}{p})}{p}\right)\prod_{p\mid T}\Bigl(1-\frac1p\Bigr)^{O(\ell)}+E(x)$, \newline
	where $T=m\prod_{i<j}(u_iv_j-v_iu_j)(u_iu_j+v_iv_j)$ and $E(x):=O\Bigl(\frac xm\frac{L^{O_\ell(1)}}{(\log x)^{\ell+1}}\Bigr)$.
\end{claim}
\begin{proof}
We use sieve theory methods from Chapter~\ref{chapter: sieves}. Indeed, define $(\mathcal{A},\mathcal{F},\mathbb{P})$ to be the discrete uniform probability space with \[\mathcal{A}:=\Bigl\{(a,b): 1\le a,b\le\sqrt{\tfrac xm}\Bigr\};\] we consider the bivariate polynomial sieve problem. Note that a good approximation to $|\mathcal{A}|$ is $X:=\dfrac xm$, as $|\mathcal{A}|=\left(\Bigl\lfloor\sqrt{\frac xm}\Bigr\rfloor \right)^2$. In fact, $|\mathcal{A}|-X=O\Bigl(\sqrt{\frac xm}\Bigr)$.

We set up the sieve problem. Let $F(a,b):=(a^2+b^2)\prod_{\phi\in I}\phi(a,b)$, and set $S=\{(a^2+b^2)\prod_{\phi\in I}\phi(a,b): 1\le a,b\le\sqrt{\tfrac xm}\}$, following previous notation. We wish to sift out elements of $S$ which are divisible by primes smaller than $y$, so $\mathfrak{p}$ is the set of all primes and $z:=y=x^{1/10L}$.

We want to find an appropriate `weight' function which was referred to as $g$ in \S~\ref{chapter: sieves}: for all squarefree $d\mid \mathcal{P}(y)$, we define $g(d)$ so that $|\mathcal{A}_d|$ is approximated by $g(d)X$.

\begin{definition}[The function $g(d)$ for the sieve]
\[g(d)=\dfrac1{d^2}\#\{(a^*,b^*)\in\mathbb{Z}_d\times\mathbb{Z}_d: ((a^*)^2+(b^*)^2)\smash{\prod_{\phi\in I}}\phi(a^*,b^*)\equiv 0\pmod d\}.\]
\end{definition}
By the Chinese Remainder Theorem, $g$ is multiplicative.
\begin{remark}\label{R1}
If $\lfloor\sqrt{\frac xm}\rfloor$ were a multiple of $d$, then equal numbers of elements of $\mathcal{A}$ would fall into each residue class modulo $d$, and hence $\mathbb{P}(\mathcal{A}_d)=g(d)$ in this case. We will later come back to the error arising from $d\nmid\lfloor\sqrt{\frac xm}\rfloor$.
\end{remark}

Now, $\mathbb{P}(\mathcal{A}_p)=\mathbb{P}\Bigl(p\mid (a^2+b^2)\prod_{\phi\in I}\phi(a,b)\Bigr)=\mathbb{P}\Bigl( p\mid (a^2+b^2)\cup\bigcup\limits_{\phi\in I}(p\mid \phi(a,b))\Bigr)$. 


\paragraph{The formula for $g(p)$} We define $g(p)$ by approximating these probabilities. We first calculate the cardinality of $\{(a^*,b^*)\in\mathbb{Z}_p\times\mathbb{Z}_p: \phi_0(a^*,b^*)\equiv 0\pmod p\}$ for a fixed $\phi_0=rx+sy\in I$. If $s$ is invertible modulo $p$, then this quantity is $p$, as for every $a^*\pmod p$ there is a unique $b^*\equiv -rs^{-1}a^*\pmod p$. If $p\mid s$, then the quantity is still $p$ because the set consists of pairs $(0,b^*)$ where $b^*$ is arbitrary. Hence  $\#\{(a^*,b^*)\in\mathbb{Z}_p\times\mathbb{Z}_p: \phi_0(a^*,b^*)\equiv 0\pmod p\}=p$.

Next we calculate the cardinality of $\{(a^*,b^*)\in\mathbb{Z}_p\times\mathbb{Z}_p: (a^*)^2+(b^*)^2\equiv 0\pmod p\}$. If $p\equiv 3\pmod 4$, we require $p\mid a^*$ and $p\mid b^*$, as $-1$ is not a quadratic residue modulo $p$, so this quantity is 1. If $p\equiv 1\pmod 4$, then $-1$ is a quadratic residue modulo $p$ -- so for any given $a^*\not\equiv 0\pmod p$, we have two options modulo $p$ for $b^*$ with $(a^*)^2+(b^*)^2\equiv 0\pmod p$. If $a^*\equiv 0\pmod p$ then $b^*\equiv 0\pmod p$ is the only option. Hence in this case the quantity is $2p-1$. We can combine the two cases, writing $\#\{(a^*,b^*)\in\mathbb{Z}_p\times\mathbb{Z}_p: (a^*)^2+(b^*)^2\equiv 0\pmod p\}=p+(p-1)\bigl(\frac{-1}p\bigr)$.


To calculate $g(p)$ we now need to find the total number of elements of $\mathbb{Z}_p\times\mathbb{Z}_p$ which were counted above in the different cases. First we calculate the size of \[\{(a^*,b^*)\in(\mathbb{Z}_p)^2: \smash{\prod_{\phi\in I}}\phi(a^*,b^*)\equiv 0\!\!\!\!\!\pmod p\}=\smash{\bigcup_{\phi\in I}}\{(a^*,b^*)\in(\mathbb{Z}_p)^2: p\mid\phi(a^*,b^*)\}.\]
Indeed, consider $\phi_1=r_1x+s_1y$, $\phi_2=r_2x+s_2y\in I$. Let $M=\begin{bmatrix}
	r_1 & r_2 \\
	s_1 & s_2
\end{bmatrix}$. If $p\mid\det(M)$, then this means that modulo $p$, the linear forms $\phi_1$ and $\phi_2$ are proportional, so the congruence classes $(a^*,b^*)$ such that $p\mid\phi_i(a^*,b^*)$ coincide. If $p\nmid\det(M)$, then these forms are linearly independent modulo $p$, so the only congruence class in the intersection of the two cases is $(0,0)$. Therefore, $\#\{(a^*,b^*)\in(\mathbb{Z}_p)^2: \prod_{\phi\in I}\phi(a^*,b^*)\equiv 0\pmod p\}=1+(p-1)\ell_p$, with $\ell_p$ being the number of linear forms $\phi\in I$ pairwise independent modulo $p$.


\begin{remark} $\det(M)$ can be of 3 possible forms:
\begin{itemize}[leftmargin=0.4cm]
	\item If $r_1=u_i,s_1=v_i$ and $r_2=v_i,s_2=-u_i$, or vice versa, then $\det(M)=u_i^2+v_i^2=m$.
	\item If $r_1=u_i,s_1=v_i$ and $r_2=u_j,s_2=v_j$, or vice versa, then $\det(M)=u_iv_j-v_iu_j$.
	\item If $r_1=u_i,s_1=v_i$ and $r_2=v_j,s_2=-u_j$, or vice versa, then $\det(M)=u_iu_j+v_iv_j$.
\end{itemize}

Hence $p\mid\det(M)\implies p\mid T$.\end{remark}

We finally calculate the cardinality of \[\{(a^*,b^*)\in(\mathbb{Z}_p)^2: p\mid(a^*)^2+(b^*)^2\}\cup\smash{\bigcup_{\phi\in I}}\{(a^*,b^*)\in(\mathbb{Z}_p)^2: p\mid\phi(a^*,b^*)\}.\]Take some fixed $\phi_0=rx+sy\in I$.
Suppose $p\mid ar+bs$, and $p\nmid a$ (so $p\nmid b$). Note that if $p\mid s$, then $ar\equiv 0\pmod p\implies p\mid r$, as $p\nmid a$. But then $\gcd(r,s)>1$, which contradicts the definition of $I$. So we may assume $r,s$ are invertible modulo $p$. Thus $ar+bs\equiv 0\pmod p\implies b\equiv -rs^{-1}a\pmod p$, so
\begin{align*}
	p\mid a^2+b^2 &\iff a^2\equiv -b^2\equiv-(rs^{-1})^2a^2\pmod p \\
	&\iff r^2(s^{-1})^2\equiv -1\pmod p \iff p\mid r^2+s^2=m
	\text{ (by definition of $I$).}
\end{align*}
So given $p\mid\phi_0(a,b)$ and $p\nmid a,b$, we have $p\mid a^2+b^2\iff p\mid m$. Hence, if $p\mid m$ then $\bigcup_{\phi\in I}\{(a^*,b^*)\in(\mathbb{Z}_p)^2: p\mid\phi(a^*,b^*)\}\subseteq\{(a^*,b^*)\in(\mathbb{Z}_p)^2: p\mid(a^*)^2+(b^*)^2\}$, which is true whether $p\nmid a,b$ or not. If $p\nmid m$, then the intersection of the events $\bigcup_{\phi\in I}\{(a^*,b^*)\in(\mathbb{Z}_p)^2: p\mid\phi(a^*,b^*)\}$ and $\{(a^*,b^*)\in(\mathbb{Z}_p)^2: p\mid(a^*)^2+(b^*)^2\}$ is $(0,0)$, and so the cardinality of their union is the sum of cardinalities minus 1.

\begin{remark*}
These arguments break down if $p$ is very small, in fact when $p\le\ell$. We will deal with this case separately.
\end{remark*}

Combining the above arguments, we arrive at
\begin{lemma*}[Formula for $p^2g(p)$] Let $p>\ell$ be prime.
	
\begin{itemize}
	\item If $p\nmid T$ (in particular $p\nmid m$) then $p^2g(p)=1+(p-1)\Bigl(\ell+1+\bigl(\frac{-1}p\bigr)\Bigr)$.
	\item If $p\mid T$ and $p\nmid m$ then $p^2g(p)=1+(p-1)\Bigl(\ell_p+1+\bigl(\frac{-1}p\bigr)\Bigr)$, where $\ell_p$ is the number of linear forms $\phi\in I$ pairwise independent modulo $p$ ($\ell_p\le\ell$).
	\item If $p\mid m$ then $p^2g(p)=1+(p-1)\Bigl(1+\bigl(\frac{-1}p\bigr)\Bigr)$. Note that in this case, $\bigl(\frac{-1}p\bigr)=1$, as $m$ is a sum of two coprime squares and has no prime factors congruent to $3$ modulo $4$. Define $\ell_p:=0$ in this case.
\end{itemize}
\end{lemma*}

%

At this point we will take a slight detour. The next step after setting up the sieve problem is to apply a result to bound $\mathcal{S}(\mathcal{A},\mathfrak{p},z)$. We use Theorem~\ref{lemma: fundamental lemma of sieves} derived from the fundamental lemma of sieve theory given in Chapter~\ref{chapter: sieves}, which in particular bounds the total error term $E(x)$. In order to compute this estimate, we first bound the error terms $R_d$ arising in the sieve problem, where $R_d=|\mathcal{A}_d|-g(d)X$. For example, for $d=1$, since $|\mathcal{A}|$ is actually $\Bigl(\lfloor\sqrt{\frac xm}\rfloor\Bigr)^2$ ($\mathcal{A}$ is indexed by our choice of positive integers $a$ and $b$ each at most $\sqrt{\frac xm}$), we get $R_1=R=|\mathcal{A}|-X=\Bigl(\lfloor\sqrt{\frac xm}\rfloor\Bigr)^2-\frac xm = O(\sqrt{\frac xm})$.

To apply Theorem~\ref{lemma: fundamental lemma of sieves}, we estimate $\sum_{d\le y^2}3^{\omega(d)}|R_d|$, where $\omega(d)$ is the number of distinct prime factors of $d$.

\begin{claim}\label{claim: bound sieve errors}
$\displaystyle\smash{\sum_{d\le y^2}3^{\omega(d)}|R_d|=O\Bigl(\frac xm\frac{L^{O_\ell(1)}}{(\log x)^{\ell+1}}\Bigr)}=E(x)$.
\end{claim}
\begin{proof}
As noted in Remark~\ref{R1}, we start by accounting for the error arising from $d\nmid\lfloor\sqrt{\frac xm}\rfloor$. Let $\displaystyle D:=d\left\lfloor\Bigl\lfloor\sqrt{\tfrac xm}\Bigr\rfloor/d\right\rfloor$ be the largest integer multiple of $d$ that is at most $\sqrt{\frac xm}$. If we considered the set $\mathcal{A}':=\bigl\{(a,b): 1\le a,b\le D\bigr\}$ instead of our original $\mathcal{A}$, we would get $|\mathcal{A}'_d|=g(d)D^2$, so we can rewrite $R_d=|\mathcal{A}_d|-|\mathcal{A}'_d|+g(d)D^2-g(d)X=|\mathcal{A}_d|-|\mathcal{A}'_d|+g(d)(D^2-X)$.

Indeed, $0\le |\mathcal{A}_d|-|\mathcal{A}'_d|=|\mathcal{A}_d\setminus\mathcal{A}'_d|\le|\mathcal{A}\setminus\mathcal{A}'|=X-D^2$, and $0\le g(d)\le 1$, so $R_d=\epsilon (X-D^2)$ for some $\epsilon\in[-1,1]$. 
Now, $0\le \sqrt{\frac xm}-D\le d$ so $0\le X-D^2\le 2d\sqrt{\frac xm}$. Hence $|R_d|\le 2d\sqrt{\frac xm}$.

Now we estimate $\sum_{d\le y^2}3^{\omega(d)}|R_d|$. In fact, $\omega(d)\le \log_3(d)+1$, as each prime factor of $d$ apart from $2$ is at least $3$.  Using our estimate of $|R_d|$ above gives a total error of $\sum_{d\le y^2}3d\cdot 2d\sqrt{\frac xm}=6\sqrt{\frac xm}\sum_{d\le y^2}d^2\ll \sqrt{\frac xm}(y^2)^3=\sqrt{\frac xm}x^{6/10L}\ll\frac xm\frac{L^{O_\ell(1)}}{(\log x)^{\ell+1}}$. This holds as $L$ is eventually large, and so $x^{6/10L}$ is of smaller order than $\sqrt{x}\frac{L^{O_\ell(1)}}{(\log x)^{\ell+1}}$.
\end{proof}
We now continue the proof of Claim~\ref{claim: sieve theory bound on summand}.
Theorem~\ref{lemma: fundamental lemma of sieves} gives \[\sum_{{\substack{a^2+b^2\le x/m\\\gcd(a^2+b^2,\prod_{q\le y}q)=1\\\prod_{\phi\in I}\gcd(\phi(a,b),\prod_{q\le y}q)=1}}}\!\!\!\!\!\!\!\! 1\ \le\ XG(y,y)+O\Bigl(\frac xm\frac{L^{O_\ell(1)}}{(\log x)^{\ell+1}}\Bigr),\]
where $G(z,z)=\frac1{e^{\gamma\kappa}\Gamma(\kappa+1)}\prod_{p<z}(1-g(p))^{-1}(1+O(\tfrac1{\log z}))$. Since $\kappa$ is the sifting dimension and is at most $\ell+2$, we have $\frac1{e^{\gamma\kappa}\Gamma(\kappa+1)}=O_\ell(1)$. Now, $\displaystyle\prod_{p<y}(1-g(p))^{-1}=O_\ell(1)\!\!\!\!\prod_{\substack{\ell+2<p<y\\ p\, \nmid\, T}}
\left(1-\frac{1+(p-1)(\ell+1+(\frac{-1}p))}{p^2}\right)^{-1}\!\!\!\!
\prod_{\substack{\ell+2<p<y\\ p\, \mid\, T}}
\left(1-\frac{1+(p-1)(\ell_p+1+(\frac{-1}p))}{p^2}\right)^{-1}$, where $\prod_{p\le\ell+2}(1-g(p))^{-1}=O_\ell(1)$ as well.

Note that $\forall p\mid T\colon\ell-\ell_p\ll\ell$. We have $\displaystyle1-\frac{c}{p}\asymp\left(1-\frac{1}{p}\right)^{c}$ for constants $c<p$, so $\displaystyle\left(1-\frac{1+(p-1)(\ell+1+(\frac{-1}p))}{p^2}\right)\left(1-\frac{1+(p-1)(\ell_p+1+(\frac{-1}p))}{p^2}\right)^{-1}\ll\left(1-\frac{1}{p}\right)^{O(\ell)}$.

Hence $\displaystyle\prod_{p<y}(1-g(p))^{-1}\ll_\ell\prod_{\ell+2<p<y}
\left(1-\frac{1+(p-1)(\ell+1+(\frac{-1}p))}{p^2}\right)^{-1}\prod_{p\mid T}\Bigl(1-\frac1p\Bigr)^{O(\ell)}$.

Finally, we have $\Bigl(1-\frac{1+(p-1)(\ell+1+(\frac{-1}p))}{p^2}\Bigr)-\Bigl(1-\frac{\ell+1+(\frac{-1}p)}{p}\Bigr)=\frac{\ell+(\frac{-1}p)}{p^2}$; also, $\prod_{p}(1-\frac1{p^s})=\sum_{n\in\mathbb{N}}\frac1{n^s}=O(1)$ for all $s>1$ so $\prod_{p}(1-\frac{\ell+(\frac{-1}p)}{p^2})\ll O_\ell(1)\prod_{p\ge(\ell+2)^2}(1-\frac{1}{p^{3/2}})\sum_{n\in\mathbb{N}}\frac1{n^{3/2}}=O_\ell(1)\ll_\ell1$. We now use $1-\frac{c}{p^2}\asymp\left(1-\frac{1}{p^2}\right)^{c}$ to finally get
\[\smash[b]{\prod_{p<y}(1-g(p))^{-1}\ll_\ell\prod_{\ell+2<p<y}
\left(1-\frac{\ell+1+(\frac{-1}p)}{p}\right)^{-1}\prod_{p\mid T}\Bigl(1-\frac1p\Bigr)^{O(\ell)}},\] and substituting $X=\dfrac xm$ gives the necessary bound on $\mathcal{C}(x)$.
\end{proof}
We next estimate $\displaystyle \frac x m\prod_{\ell+2<p\le y}\left(1-\frac{\ell+1+(\frac{-1}{p})}{p}\right)\prod_{p\mid T}\Bigl(1-\frac1p\Bigr)^{O(\ell)}$; we wish to bound it asymptotically by $\dfrac xm\dfrac{L^{O_\ell(1)}}{(\log x)^{\ell+1}}$, so we bound the products over primes by $\dfrac{L^{O_\ell(1)}}{(\log x)^{\ell+1}}$.
\begin{claim}\label{bound primes dividing T}
	$\displaystyle\smash{\prod_{p\mid T}\Bigl(1-\frac1p\Bigr)^{O(\ell)}}\ll_\ell L^{O(1)}$.
\end{claim}
\begin{proof}
By Theorem~\ref{estimate on omega} we have $\#\{p\mid T\}=\omega(T)\ll\frac{\log T}{\log\log T}$. Let $U:=\frac{\log T}{\log\log T}$. By the Taylor series for the $\exp$ function, 	\[\smash{\prod_{p\mid T}\left(1-\frac{1}{p}\right)^{O(\ell)} \!\!\ll
\prod_{p\le U}\left(1-\frac{1}{p}\right)^{O(\ell)} \!\!\ll \exp\Bigl(-\sum_{p\le U}\frac1p\Bigr).}\] Using Mertens' Second Theorem to estimate $\sum_{p\le U}\frac1p$ we get
\begin{align*}
	\prod_{p\mid T}\left(1-\frac{1}{p}\right)^{O(\ell)} \!\!\ll
	\exp\left(-\log\log U-M+O\Bigl(\frac1{\log U}\Bigr)\right)
	&\ll\frac1{\log(\frac{\log T}{\log\log T})} \ll\frac1{\log\log T}.
\end{align*}
So, to prove the necessary bound, we show that $T\ll\exp(\exp (L^{O(1)}))$, recalling that $\displaystyle T=m\prod_{i<j}(u_iv_j-v_iu_j)(u_iu_j+v_iv_j)$. For any given pair $(i,j)$, we have $\displaystyle u_iu_j+v_iv_j\le\sqrt{(u_i^2+v_i^2)(u_j^2+v_j^2)}=\sqrt{m^2}=m$ by the Cauchy-Schwarz inequality. Similarly, $u_iv_j-v_iu_j\le u_iv_j+v_iu_j\le\sqrt{(u_i^2+v_i^2)(v_j^2+u_j^2)}=m$. Since there are $\binom{r_0^*(m)}2$ possible pairs $(i,j)$ with $i<j$, we obtain $T\ll m^{\binom{r_0^*(m)}2+1}$.

Lemma~\ref{r0*} gives $r_0^*(2)=1$, $\forall j\ge 2\colon r_0^*(2^j)=0$ and $r_0^*(p^j)=2$. With Lemma~\ref{r0*mult}, we get $r_0^*(m)=2^{\omega^*(m)}$, noting that $4\nmid m$ since $m\in\mathcal{Q}$, as in the leftmost summation of~\eqref{3sum}. Additionally, in~\eqref{3sum} we have $\omega^*(m)=k-1$, so $r_0^*(m)=2^{k-1}$.

By definition, $k\ll\log\log x$. So $r_0^*(m)\ll\log x$, and thus $T\ll m^{(\log x)^2}$. Hence we get $\log\log T\ll\log((\log x)^2\log m)=2L+\log\log m$. Moreover $m\le x^{1-1/L}\le x$, so $\log\log m\le L$. Hence $\log\log T\ll 3L\ll L^{O(1)}$, as required.
\end{proof}

Next, we consider $\displaystyle\prod_{\ell+2<p\le y}\Bigl(1-\tfrac{\ell+1+(\frac{-1}{p})}{p}\Bigr)$. Now, $(1-\frac{\ell+1+(\frac{-1}{p})}{p})\asymp(1-\frac1p)^{\ell+1}\Bigl(1-\frac{(\frac{-1}{p})}{p}\Bigr)$, by considering the binomial expansion of $(1-\frac1p)^{\ell+1}$.

Let $\chi_4$ be the non-principal Dirichlet character modulo $4$. Then $\chi_4(p)=(\frac{-1}{p})$ for $p>\ell+2$. We have $L(1,\chi_4)<\infty$ where $L(s,\chi)$ is the Dirichlet $L$-function, as discussed in Theorem~\ref{theorem: mertens APs}; by considering the Euler product form of $L(1,\chi_4)$ we get $L(1,\chi_4)\asymp\prod_{2<p\le y}\left(1-\frac{(\frac{-1}{p})}{p}\right)$. Hence \[\smash{\prod_{\ell+2<p\le y}\left(1-\frac{\ell+1+(\frac{-1}{p})}{p}\right)\asymp\prod_{\ell+2<p\le y}\left(1-\frac{1}{p}\right)^{\ell+1}L(1,\chi_4)^{-1}.}\]

Now, $\prod_{\ell+2<p\le y}(1-\frac{1}{p})\asymp\prod_{p\le y}(1-\frac{1}{p})$. Note that $\prod_{p\le y}\left(1-\frac{1}{p}\right)\le\exp\left(-\sum_{p\le y}\frac1p\right)=\exp\left(-\log\log y-M+O\Bigl(\frac1{\log y}\Bigr)\right)\ll\frac1{\log y}$, by Mertens' second theorem. Since $\log y=\log(x^{1/10L})=\frac1{10L}\log x$, we get

$\displaystyle\prod_{\ell+2<p\le y}\left(1-\frac{1}{p}\right)^{\ell+1}L(1,\chi_4)^{-1}\ll\frac1{(\log y)^{\ell+1}}=\frac{(10L)^{\ell+1}}{(\log x)^{\ell+1}}\asymp\frac{L^{O_\ell(1)}}{(\log x)^{\ell+1}}$.

This now gives us the required bound: 
\begin{equation}\label{bound products}
\prod_{\ell+2<p\le y}\left(1-\frac{\ell+1+(\frac{-1}{p})}{p}\right)\prod_{p\mid R}\Bigl(1-\frac1p\Bigr)^{O(\ell)}\ll L^{O(1)}\cdot\frac{L^{O_\ell(1)}}{(\log x)^{\ell+1}}=\frac{L^{O_\ell(1)}}{(\log x)^{\ell+1}}.
\end{equation}
We can now finish the proof of the main theorem of this section.
\begin{proof}[Proof of Theorem~\ref{thm:r1cl upper bound}]
From~\eqref{3sum} we have 
\[\sum_{{\substack{n\le x\\ \omega^*(n)=k\\P(n)>x^{1/L}}}} \binom{r_1^*(n)}{\ell}=\sum_{{\substack{m\le x^{1-1/L}\\ \omega^*(m)=k-1\\m\in\mathcal{Q}}}}\sum_{{\substack{I\subseteq\mathcal{L}_{m}\\|I|=\ell}}}\sum_{{\substack{x^{1/L}<p\le x/m\\p>P(m)\text{ prime}\\p\in\mathcal{Q},p=a^2+b^2,a<b}}}\!\!\!\!\!\!\!\!\!\!\!\! \mathbbm{1}(\forall\phi\in I:\phi(a,b)\text{ prime}).\]
Using Claim~\ref{claim: phi small prime} and~\ref{claim: sieve theory bound on summand} to estimate the right-most summand and bounding it by Claim~\ref{bound primes dividing T} and~\eqref{bound products} we get
\begin{align*}
	\smash{\sum_{{\substack{n\le x\\ \omega^*(n)=k\\P(n)>x^{1/L}}}} \binom{r_1^*(n)}{\ell}}
&	\ll_\ell 	\sum_{{\substack{m\le x^{1-1/L}\\ \omega^*(m)=k-1,\ m\in\mathcal{Q}}}}\sum_{{\substack{I\subseteq\mathcal{L}_{m}\\|I|=\ell}}}
	\frac{x}m\cdot \frac{L^{O_\ell(1)}}{(\log x)^{\ell+1}}\\
	&=\frac{xL^{O_\ell(1)}}{(\log x)^{\ell+1}}\sum_{{\substack{m\le x^{1-1/L}\\ \omega^*(m)=k-1,\ m\in\mathcal{Q}}}}\!\!\!\!\frac1m\cdot\#\{I\subseteq\mathcal{L}_{m}\colon |I|=\ell\}.
\end{align*}
For given $m$, $\#\mathcal{L}_{m}=2r_0^*(m)$. Moreover, we know $\omega^*(m)=k-1$. By Lemma~\ref{r0*} and Lemma~\ref{r0*mult} we thus have $\#\mathcal{L}_{m}\le 2\cdot2^{k-1}$, and so $\#\{I\subseteq\mathcal{L}_{m}\colon |I|=\ell\}\le 2^{\ell k}$. Thus 
\begin{equation}
	\sum_{{\substack{n\le x\\ \omega^*(n)=k\\P(n)>x^{1/L}}}} \binom{r_1^*(n)}{\ell}
	\ll_\ell \frac{2^{\ell k}xL^{O_\ell(1)}}{(\log x)^{\ell+1}}
	\sum_{{\substack{m\le x^{1-1/L}\\ \omega^*(m)=k-1\\m\in\mathcal{Q}}}}\frac1m. \label{eq: harmonic bound}
\end{equation}
Our final step is to bound the sum on the RHS above, which we do by subdividing the range of the sum into intervals of form $(M,2M]$, $M=2^j$ for $1\le j\le J$ with $J:=\lceil\log_2 (x^{1-1/L})\rceil=\lceil(1-\frac1L)\frac{\log x}{\log 2}\rceil$.

By Lemma~\ref{lemma: numbers with k factors}, we can estimate $\rho_{k-1}(2M)-\rho_{k-1}(M)$, which is the number of $m$ in the interval $(M,2M]$ such that $\omega^*(m)=k-1$. Now $\displaystyle \sum_{{\substack{M<m\le 2M\\ \omega^*(m)=k-1\\m\in\mathcal{Q}}}}\frac1m
\le\frac1M\cdot(\rho_{k-1}(2M)-\rho_{k-1}(M))
\le\frac1M\cdot\rho_{k-1}(2M)
\ll\frac1M\cdot\frac{2M}{\log 2M}\frac{(\frac12\log\log M+O(1))^{k-2}}{(k-2)!}
\ll\frac1{\log M}\frac{(\frac12L+O(1))^{k-2}}{(k-2)!}=\frac1j\frac{(\frac12L+O(1))^{k-2}}{(k-2)!}$, for $M=2^j$.

Hence substituting into~\eqref{eq: harmonic bound} gives
$\displaystyle\sum_{{\substack{n\le x\\ \omega^*(n)=k\\P(n)>x^{1/L}}}} \binom{r_1^*(n)}{\ell}
\ll_\ell \frac{2^{\ell k}xL^{O_\ell(1)}}{(\log x)^{\ell+1}}
\cdot\frac{(\frac12L+O(1))^{k-2}}{(k-2)!}\sum_{j\le J}\frac1j$. Now, since $k\ll\log\log x$, we have $(k-1)k\ll L^{O(1)}$; so
\[\sum_{{\substack{n\le x\\ \omega^*(n)=k\\P(n)>x^{1/L}}}} \binom{r_1^*(n)}{\ell}
\ll_\ell \frac{2^{\ell k}xL^{O_\ell(1)}}{(\log x)^{\ell+1}}
\cdot\frac{(\frac12L+O(1))^{k}}{k!}\sum_{j\le J}\frac1j=
\frac{xL^{O_\ell(1)}}{(\log x)^{\ell+1}}
\cdot\frac{(2^{\ell-1}L)^k}{k!}\sum_{j\le J}\frac1j.\]
But by Lemma~\ref{lemma: sum of 1n}, $\sum_{j\le J}\frac1j\ll\log J\ll\log(1-\frac1L)+\log\log x\ll \frac1L+L\ll L$. So $\displaystyle\sum_{{\substack{n\le x\\ \omega^*(n)=k\\P(n)>x^{1/L}}}} \binom{r_1^*(n)}{\ell}
\ll_\ell \frac{xL^{O_\ell(1)}}{(\log x)^{\ell+1}}
\cdot\frac{(2^{\ell-1}L)^k}{k!}$, as the $L$ is absorbed into $L^{O_\ell(1)}$.

In conclusion, 
\[\sum_{{\substack{n\le x\\ \omega^*(n)=k}}} \binom{r_1(n)}{\ell}
\ll_\ell \frac{xL^{O_\ell(1)}}{(\log x)^{\ell+1}}
\cdot\frac{(2^{\ell-1}L)^k}{k!}+x^{\frac12+o(1)}\ll_\ell \frac{xL^{O_\ell(1)}}{(\log x)^{\ell+1}}
\cdot\frac{(2^{\ell-1}L)^k}{k!}.\qedhere\]
\end{proof}
\section{The second moment of $r_1$}
Stephan Daniel's paper~\cite{daniel_2001} gives a good estimate for the second moment of $r_1$, showing the main term precisely and calculating an error term. He starts by proving the following Lemma:
\begin{lemma}\label{r1c2}
%
	$\displaystyle\sum_{n\le x}\binom{r_1(n)}{2}=\frac98\frac x{\log x}+O\left(\frac{x(L(x))^2}{\log^2x}\right)$.
\end{lemma}
%
%
%
%
%
Since we know the asymptotics of the first moment, we can easily obtain an expression for the second moment of $r_1$ by recalling the definition of binomial coefficients:
\begin{lemma}\label{lemma: daniel second moment r1} 
	$\displaystyle\smash{\sum_{n\le x}r_1(n)^2=\Bigl(\frac\pi2+\frac94\Bigr)\frac x{\log x}+O\left(\frac{x(L(x))^2}{\log^2x}\right)}$.
\end{lemma}
Granville, Sabuncu and Sedunova obtain a weaker result in their paper~\cite{granville2023} based on Theorem~\ref{thm:r1cl upper bound} and an additional calculation of lower bounds on moments of $r_1$, which we will not discuss in this dissertation.
\begin{theorem}[Granville, Sabuncu, Sedunova, 2023]\label{thm:r1c2 gss}
	For any integer $k\asymp L$,
	\[\smash{\sum_{{\substack{n\le x\\\omega(n)=k}}}\binom{r_1(n)}{2}\asymp\frac{xL^{O(1)}}{(\log x)^3}\frac{(2L)^k}{k!}.}\]
\end{theorem}
We give a brief explanation for why this result is consistent with Lemma~\ref{r1c2}.
\begin{lemma}
	$\displaystyle\frac{xL^{O(1)}}{(\log x)^{\ell+1}}\frac{(2L)^k}{k!}$ is maximised when $k$ is approximately $2L$.
\end{lemma}
\begin{proof}
	Stirling's approximation gives us that $k!\sim\sqrt{2\pi k}(\frac ke)^k$. Thus, we seek to maximise $\dfrac{(2L)^k}{\sqrt{k}(\frac ke)^k}=\dfrac{(2Le)^k}{k^{k+1/2}}\asymp (2Le)^k k^{-k-\frac12}$.
	
Let $f(y)=(2Le)^y y^{-y-\frac12}$. We calculate $f'(y)$ via the product rule. Note that $\frac{\mathrm{d}}{\mathrm{d}y}((2Le)^y)
=(2Le)^y(\log (2Le))=(2Le)^y(\log 2L +1)$, and $\frac{\mathrm{d}}{\mathrm{d}y}(y^{-y-\frac12})=\frac{\mathrm{d}}{\mathrm{d}y}(e^{\log y(-y-\frac12)})=e^{\log y(-y-\frac12)}\bigl(-\log y-\frac{y+\frac12}{y}\bigr)$. Hence
\begin{closealign}f'(y)=f(y)\Bigl(\log 2L+1-\log y-1-\frac1{2y}\Bigr)=f(y)\Bigl(\log 2L-\log y-\frac1{2y}\Bigr).\end{closealign}

$f(y)>0$ for $y>0$ and $\log 2L-\log y-\frac1{2y}$ is a decreasing function on $y>0$ so in this range, $f(y)$ has a global maximum at its unique turning point satisfying the equation $\log y+\frac1{2y}=\log 2L$. The left hand side is approximately $\log y$, whence we get that $\displaystyle\frac{xL^{O(1)}}{(\log x)^3}\frac{(2L)^k}{k!}$ is maximised for $k$ close to $2L$.
\end{proof}
In Theorem~\ref{thm:r1c2 gss}, substituting $k=2L$ gives \[\frac{xL^{O(1)}}{(\log x)^3}\frac{(2L)^k}{k!}\asymp\frac{xL^{O(1)}}{(\log x)^3}\frac{(2L)^{2L}}{\sqrt{4\pi L}(\frac {2L}e)^{2L}}=\frac{xL^{O(1)}}{(\log x)^3}\frac{\exp(2L)}{\sqrt{4\pi L}}\asymp\frac{xL^{O(1)}}{(\log x)^3}\frac{(\log x)^2}{L^{O(1)}}=\frac{xL^{O(1)}}{\log x}.\] Note that to obtain the second moment of $r_1$, we sum over all possible $k$, to count all possible $n\le x$ based on their number of prime factors. However, we expect the contribution of other $k$ to be small compared to $k=2L$, hence the result is consistent with a main term of order $\frac{x}{\log x}$. Similarly, this is consistent with Theorem~\ref{sabuncur1}. 
\chapter{New problem: replace primes with sums of two squares}\vspace{-2ex}
We have previously considered the representation numbers $r_0$ and $r_1$. Both count the number of ways to write $n$ as the sum of an integer square and the square of \textit{something}. In the case of $r_0$, we allow the square of any (positive) integer, and in the case of $r_1$, only squares of primes are allowed. As one would expect, since $r_1$ counts a more restricted set, the moments of $r_1$ are smaller than those of $r_0$. In fact, the ratio of main terms in the first moments is $\frac1{\log x}$ and the ratio of main terms in the second moments is $\frac1{\log^2 x}$, up to a constant. One might guess this relationship from the Prime Number Theorem, which tells us that $\mathbb{P}(p\le x\text{ is prime})=\frac1{\log x}$.

In this chapter we consider new representation numbers of the form \[r_S(n):=\#\{(a,b)\in\mathbb{Z}_{\ge 0}\times S: a^2+b^2=n\}.\] We introduce the following notation, where $r_S^*$ is a generalisation of $r_0^*$ and $r_1^*$:

\begin{notation} ($\mathcal{R}$, $\mathcal{R}'$, $r_S^*$)
	
\begin{itemize}
	\item $\mathcal{R}:=\{k\in\mathbb{N}:r_0(k)\ge 1\}$ is the set of numbers which can be written as a sum of two squares.
	\item $\mathcal{R}':=\{k\in\mathcal{R}:r_0^*(k)\ge 1\}$ is the set of numbers which can be written as a sum of two coprime squares.
	\item $r_S^*(n):=\#\{(a,b)\in\mathbb{Z}_{\ge 0}\times S: a^2+b^2=n, \ \gcd(a,b)=1\}$.
\end{itemize}
\end{notation} We can classify the elements of $\mathcal{R}$ and $\mathcal{R}'$ entirely: in the prime factorisation of elements of $\mathcal{R}$, any prime that is $3$ modulo $4$ can only appear in an even power. Elements of $\mathcal{R}'$ are not divisible by primes that are $3$ modulo $4$, nor by $4$.

\section{First moments of $r_S$}
We start by finding the first moments of $r_{\mathcal{R}}$ and $r_{\mathcal{R}'}$. Similar to our discussion at the start of the chapter, this will allow us to make a hypothesis about the behaviour of the second moments of these functions. The proof method is similar to that of Lemma~\ref{lemma: first moment of r1}, but we use a slightly simpler approach which gives a weaker error term as our aim currently is to obtain the main term of these moments.

\begin{lemma}[First moment of $r_{\mathcal{R}}$]\label{lemma: first moment of rR}
	$\displaystyle\sum_{n\le x}r_{\mathcal{R}}(n)= c\frac{x}{\sqrt{\log x}}\left(1+O\Bigl(\frac{1}{\log x}\Bigr)\right)$, where $c$ is the constant $\ \displaystyle\frac{\pi}4\prod_{p\equiv 3\mod 4}\bigl(1-\tfrac1{p^2}\bigr)^{-\frac12}$.
\end{lemma}
\begin{proof}
	We start by writing
\begin{align}
\sum_{n\le x}r_{\mathcal{R}}(n)&=\sum_{a\le\sqrt{x}}\sum_{n\le\sqrt{x-a^2}}\mathbbm{1}_{\{n\in\mathcal{R}\}}=
	\sum_{a\le\sqrt{\frac x2}}\sum_{\substack{n\le\sqrt{x-a^2}\\n\in\mathcal{R}}}1+\sum_{\substack{n\le\sqrt{\frac x2}\\n\in\mathcal{R}}}\sum_{a\le\sqrt{x-n^2}}1-\sum_{a\le\sqrt{\frac x2}}\sum_{\substack{n\le\sqrt{\frac x2}\\n\in\mathcal{R}}}1 \notag\\
	&=\sum_{a\le\sqrt{\frac x2}}M_0(\sqrt{x-a^2})+\sum_{\substack{n\le\sqrt{\frac x2}\\n\in\mathcal{R}}}\lfloor\sqrt{x-n^2}\rfloor-\sqrt{x/2}\cdot M_0(\sqrt{x/2}),\label{eq: rewrite rR}
\end{align} where $M_0$ is the zeroth moment of $r_0$. We recall Lemma~\ref{lemma: zeroth moment of r0}, which states that $M_0(x)=c'\frac x{\sqrt{\log x}}\bigl(1+O((\log x)^{-1})\bigr)$, where $c':=\bigl(2\prod_{p\equiv 3\mod 4}(1-\tfrac1{p^2})\bigr)^{-\frac12}.$

We deal with the components in~\eqref{eq: rewrite rR}, proceeding as in Lemma~\ref{lemma: first moment of r1}. For the first component, as $a\le\sqrt{x/2}$ we have
$\log x\ge\log(x-a^2) \ge\log(x/2)=\log x-\log 2$, so  \[\sqrt{\frac{2(x-a^2)}{\log x}}\le\sqrt{\frac{2(x-a^2)}{\log (x-a^2)}}\le\sqrt{\frac{2(x-a^2)}{\log x-\log 2}}.\] Since
$(\log x-\log 2)^{-1}=\log^{-1} (x)(1+O(\frac{1}{\log x}))$ we get \[M_0(\sqrt{x-a^2})=c'\sqrt{\frac{2(x-a^2)}{\log x}}\left(1+O\Bigl(\frac{1}{\log x}\Bigr)\right),\] and thus $\sum_{a\le\sqrt{\frac x2}}M_0(\sqrt{x-a^2})=c'\sqrt{\frac2{\log x}}\sum_{a\le\sqrt{\frac x2}}\sqrt{x-a^2}\left(1+O(\frac{1}{\log x})\right)$.

As in Lemma~\ref{lemma: first moment of r1} we have $\sum_{a\le\sqrt{x/2}}\sqrt{x-a^2}=\int_0^{\sqrt{x/2}}\sqrt{x-a^2}\mathrm{d}a + O(\sqrt{x})=\frac{\pi+2}8x\left(1+O(\frac{1}{\log x})\right)$, so \[\sum_{a\le\sqrt{\frac x2}}M_0(\sqrt{x-a^2})=c'\frac{\sqrt{2}(\pi+2)}8\frac{x}{\sqrt{\log x}}\left(1+O\Bigl(\frac{1}{\log x}\Bigr)\right).\]
For the second component of~\eqref{eq: rewrite rR}, apply partial summation with $a_n=\mathbbm{1}_{\{n\in\mathcal{R}\}}$ and $f(y)=\sqrt{x-y^2}$ to get
\[\sum_{\substack{n\le\sqrt{\frac x2}\\n\in\mathcal{R}}}(\sqrt{x-n^2}+O(1))=O(\sqrt{x})+M_0(\sqrt{x/2})\sqrt{x/2}+\smash[t]{\int_1^{\sqrt{x/2}}M_0(t)t(x-t^2)^{-1/2}\mathrm{d}t}.\] As previously the $M_0(\sqrt{x/2})\sqrt{x/2}$ term is the third component of~\eqref{eq: rewrite rR}, and $O(\sqrt{x})$ can be absorbed into the error term of the first component as $O(\sqrt{x})=O(\frac{x}{{\log x}^{3/2}})$. It remains to estimate the integral, noting that we can change the lower bound to $2$ as $M_0(t)=0$ for $t< 2$. We substitute the main term of $M_0$ into the integral.

We substitute $t=u\sqrt{x}$ to get
\[\int_2^{\sqrt{x/2}}c'\frac{t}{\sqrt{\log t}}\frac{t}{\sqrt{x-t^2}}\mathrm{d}t=c'\sqrt{x}\int_{2/\sqrt{x}}^{\sqrt{1/2}}\frac{xu^2}{\sqrt{x}\sqrt{1-u^2}}\frac{\mathrm{d}u}{\sqrt{\log(u\sqrt{x})}}.\]
Integration by parts, using the formula for $\int\frac{u^2}{\sqrt{1-u^2}}\mathrm{d}u$, gives
	\begin{align}
		&c'\frac{x}2\left[\frac{\arcsin(u)-u\sqrt{1-u^2}}{\sqrt{\log(u\sqrt{x})}}\right]_{2/\sqrt{x}}^{\sqrt{1/2}}+c'\frac{x}4\int_{2/\sqrt{x}}^{\sqrt{1/2}}\frac{(\arcsin(u)-u\sqrt{1-u^2})\mathrm{d}u}{u(\log(u\sqrt{x}))^{3/2}} \notag\\
		=&\frac{c'x(\frac{\pi}4-\frac12)}{2\sqrt{\log(\sqrt{\frac x2})}}+O(\sqrt{x})+\frac{c'x}{4}\int_{2/\sqrt{x}}^{\sqrt{1/2}}\Bigl(\frac{\arcsin(u)}u-\sqrt{1-u^2}\Bigr)\frac{\mathrm{d}u}{(\log(u\sqrt{x}))^{3/2}}\label{eq: first estimate rR int}
		.\end{align} As previously, we split this last integral into ranges $[\frac2{\sqrt{x}},\frac1{\sqrt[4]{x}}]$ and $[\frac1{\sqrt[4]{x}},\sqrt{1/2}]$. In the first component, we have $\frac{\arcsin(u)}u-\sqrt{1-u^2}=\frac23u^2+O(u^4)=O(\frac1{\sqrt{x}})$, and we bound $(\log(u\sqrt{x}))^{3/2}$ from below by $(\log2)^{3/2}$. In the second, we bound $\frac{\arcsin(u)}u-\sqrt{1-u^2}$ by a constant, and bound $(\log(u\sqrt{x}))^{3/2}$ from below by $(\log(\sqrt[4]{x}))^{3/2}$. Hence
	\begin{align}
		\int_{2/\sqrt{x}}^{\sqrt{1/2}}\Bigl(\frac{\arcsin(u)}u-\sqrt{1-u^2}\Bigr)\frac{\mathrm{d}u}{\log^2(u\sqrt{x})}
		&=O(x^{-1/4})+O((\log x)^{-3/2}).\label{eq: rR estimate small int}
	\end{align}
	Substituting~\eqref{eq: rR estimate small int} into~\eqref{eq: first estimate rR int} gives
	\[\smash{\int_2^{\sqrt{x/2}}M_0(t)t(x-t^2)^{-1/2}\mathrm{d}t=\frac{c'x}{\sqrt{2\log x}}\Bigl(\frac{\pi}4-\frac12\Bigr)+O\Bigl(\frac{x}{(\log x)^{3/2}}\Bigr).}\]
	Substituting this into~\eqref{eq: rewrite rR} gives
\begin{align}
	\sum_{n\le x}r_{\mathcal{R}}(n)&=c'\frac{\pi+2}{4\sqrt{2}}\frac{x}{\sqrt{\log x}}\left(1+O\Bigl(\frac{1}{\log x}\Bigr)\right)+\frac{c'x}{\sqrt{2\log x}}\cdot\frac{\pi-2}4+O\Bigl(\frac{x}{(\log x)^{3/2}}\Bigr) \notag \\
	&=c'\sqrt{2}\cdot\frac{\pi}4\frac{x}{\sqrt{\log x}}\left(1+O\Bigl(\frac{1}{\log x}\Bigr)\right).
\end{align}	
By definition, $c=	c'\sqrt{2}\cdot\frac{\pi}4$, so we are done.
\end{proof}
 
The first moment of $r_{\mathcal{R}'}$ behaves very similarly:
\begin{lemma}[First moment of $r_{\mathcal{R}'}$]
	$\displaystyle\sum_{n\le x}r_{\mathcal{R}'}(n)= c\frac{x}{\sqrt{\log x}}\left(1+O\Bigl(\frac{1}{\log x}\Bigr)\right)$, where \[c=\frac{3\pi}8\sqrt{\prod_{p\equiv 3\pmod 4}(1-p^{-2})}.\]
\end{lemma}
\begin{proof}
The proof will be as above, except we must replace $M_0$ with the zeroth moment of $r_0^*$, which we will denote $M_0^*$. We have \[M_0^*(x)=\#\{n\le x:n\in\mathcal{R}'\}=\#\{n\le x:4\nmid n,p\nmid n\text{ for all }p\equiv 3\pmod 4\}.\] Our task now is to evaluate this quantity.
	
We first calculate the number of $n\le x$ whose prime factorisation only contains primes that are $1$ modulo $4$. This can be done through considering the Dirichlet series $F(s)=\sum_{n\ge 1}\frac{a_n}{n^s}$ with $a_n=\mathbbm{1}(p\mid n\Rightarrow p\equiv 1\pmod 4)$. Since this indicator function is multiplicative, we have an Euler product: $F(s)=\prod_{p\equiv 1\pmod 4}(1-p^{-s})^{-1}$.

Recall the Riemann zeta function $\zeta(s)=\prod_{p}(1-p^{-s})^{-1}$ and the Dirichlet L-function $L(s,\chi)=\sum_{n\ge 1}\frac{\chi(n)}{n^s}$ introduced in Theorem~\ref{theorem: mertens APs}, where $\chi$ is the non-principal Dirichlet character modulo $4$. Since $\chi$ is totally multiplicative we can write the Euler product $L(s,\chi)=\prod_p(1-\frac{\chi(p)}{p^s})^{-1}$. Thus
\begin{equation}\label{F(s)^2 behaviour}
	\smash{\frac{\zeta(s)L(s,\chi)}{F(s)^2}=(1-2^{-s})^{-1}\prod_{p\equiv 3\pmod 4}((1-p^{-s})(1+p^{-s}))^{-1},}\end{equation} by comparing the Euler products of the series.

The Riemann zeta function has a simple pole of residue $1$ at $s=1$, and is meromorphic for $\Re(s)\ge 1$, with no other poles. Since $\chi$ is a non-principal Dirichlet character, we have $L(s,\chi)$ is holomorphic for $\Re(s)\ge 1$, and $L(1,\chi)\ne 0$. In fact, $L(1,\chi)=1-\frac13+\frac15-\dots=\frac{\pi}4$ by using the arctan Taylor series. Thus $\zeta(s)L(s,\chi)$ has a simple pole of residue $\frac{\pi}4$ at $s=1$ and is otherwise holomorphic on $\Re(s)\ge 1$. However, the right hand side in~\eqref{F(s)^2 behaviour} is holomorphic for $\Re(s)>\frac12$, as the Dirichlet series corresponding to $\prod_{p\equiv 3\pmod 4}(1-p^{-2s})^{-1}$ converges uniformly in this half-plane. Hence $F(s)^2$ must have a simple pole at $s=1$ as well, of residue $\frac{\pi}4(1-\frac12)^{-1}\prod_{p\equiv 3\pmod 4}(1-p^{-2})=\frac{\pi}2\prod_{p\equiv 3\pmod 4}(1-p^{-2})$.

We now apply a variant of the Wiener-Ikehara Theorem. The result used here was proved by Kato~\cite[Theorem 3.1]{wienerikeharatauberiantheorem}. We state it below:
\begin{theorem}[Wiener, Ikehara, Delange, Kato]
Let $\{a_n\}$ be a sequence of non-negative real numbers, $d\in\mathbb{R}_{>0}$ and $m$ a positive integer. Suppose that the Dirichlet series $L(s)=\sum_{n=1}^{\infty}\frac{a_n}{n^s}$ converges absolutely for $\Re(s)>d$. Also suppose that $L^m$ has a meromorphic continuation to an open set containing the closed half-plane $\Re(s)\ge d$, holomorphic except for a pole of order $l$ at $s=d$ with $\lim_{s\to d}L(s)^m(s-d)^l=A^m$, where $A>0$. Then \[\sum_{n\le x}a_n\sim\frac{AX^d}{d\Gamma(\frac lm)(\log x)^{1-l/m}}.\]
\end{theorem}

We substitute $L=F$, $m=2$, $d=1$, $l=1$, $A^2=\frac{\pi}2\prod_{p\equiv 3\pmod 4}(1-p^{-2})$. $\Gamma$ is the Gamma function as previously seen, and it is well-known that $\Gamma(\frac 12)=\sqrt{\pi}$. Hence we get \[f(x):=\sum_{n\le x}a_n\sim\sqrt{\frac12\prod_{p\equiv 3\pmod 4}(1-p^{-2})}\frac{x}{\sqrt{\log x}}.\]
Since $M_0^*(x)=f(x)+f(\frac x2)$ we obtain \[M_0^*(x)\sim\frac32\sqrt{\frac12\prod_{p\equiv 3\pmod 4}(1-p^{-2})}\frac{x}{\sqrt{\log x}}.\] Substituting this constant instead of the coefficient of $\frac x{\sqrt{\log x}}$ in $M_0(x)$ into the proof of~\ref{lemma: first moment of rR} gives the required result.
\end{proof}
\begin{remark*}
	The proof method here is similar to Landau's, but we use the more modern technique of the Wiener Ikehara theorem instead of Landau's computations and integrals. Lemma~\ref{lemma: zeroth moment of r0} could also be proved in a similar way (without error term) if we were able to analyse the behaviour of the corresponding Dirichlet series.
\end{remark*}
\section{Moments of $r_{\mathcal{R}'}$}
We will apply the methods of Theorem~\ref{thm:r1cl upper bound} to obtain an upper bound on higher moments of $r_{\mathcal{R}'}$.
\begin{theorem}\label{moments of rR'}
		For any fixed integer $\ell\ge 1$ and integer $k\ll L$ we have
	\[\smash{\sum_{{\substack{n\le x\\\omega^*(n)=k}}}\binom{r_{\mathcal{R}'}^*(n)}{\ell}\ll_\ell\frac{xL^{O_\ell(1)}}{(\sqrt{\log x})^{\ell+1}}\frac{(2^{\ell-1}L)^k}{k!}.}\]
\end{theorem}
We will split the proof into very similar steps to before.
\begin{claim}\label{claim: for rS*}
	\[\sum_{{\substack{n\le x\\\omega^*(n)=k \\ P(n)\le z\text{ or }P(n)^2\mid n}}}\binom{r_{S}^*(n)}{\ell}\ll_\ell \frac{xL^{O_\ell(1)}}{(\log x)^{\ell+1}}\frac{(2^{\ell-1}L)^k}{k!}.\]
\end{claim}
\begin{proof}
	The proof from Claim~\ref{claim: bound for some P(n)} works here as all $r_S^*$ are bounded by $r_0^*$.
\end{proof}In particular, Claim~\ref{claim: for rS*} applies for $S=\mathcal{R}'$.

For all $n$ that we consider (WLOG $r_{\mathcal{R}'}^*(n)>0$), we may now assume $P(n)>x^{1/L}$ and $P(n)^2\nmid n$. Let $p=P(n)$ and write $n=mp$, with $p>x^{1/L}$ and $P(m)<p$, $m\le x^{1-1/L}$. Moreover, $r_0^*(n)>0$ so $p\equiv 1\pmod 4$. Write $p=a^2+b^2$ for unique (up to order) coprime $a,b$, and let all pairs of two positive coprime squares whose sum is $m$ be denoted $u_i,v_i$, where $1\le i\le r_0^*(m)$ by definition of $r_0^*$. 

As before, we have the set of linear forms $\mathcal{L}_m=\{u_ix+v_iy,v_ix-u_iy\colon 1\le i\le r_0^*(m)\}$.

Let  $\mathcal{L}_{m,a,b,S}=\{\phi\in\mathcal{L}_m\colon\phi(a,b)\in S\}$. Then we have $\displaystyle\binom{r_{\mathcal{R}'}^*(mp)}{\ell}=\sum_{{\substack{I\subseteq\mathcal{L}_m\\|I|=\ell}}}\mathbbm{1}(\forall\phi\in I\colon\phi(a,b)\in\mathcal{R}')=\sum_{{\substack{I\subseteq\mathcal{L}_{m,a,b,\mathcal{R}'}\\|I|=\ell}}}1$. Following the reasoning in the previous Theorem, we can write the analogue of~\eqref{3sum} (using the notation in this chapter):
\begin{align}
\sum_{{\substack{n\le x\\ \omega^*(n)=k\\P(n)>x^{1/L}}}} &\binom{r_{\mathcal{R}'}^*(n)}{\ell}=	\sum_{{\substack{m\le x^{1-1/L}\\ \omega^*(m)=k-1\\m\in\mathcal{R}'}}}
	\sum_{{\substack{x^{1/L}<p\le x/m\\p>P(m)\\p\text{ prime, }p\in\mathcal{R}'}}} \binom{r_{\mathcal{R}'}^*(mp)}{\ell}
	= \sum_{{\substack{m\le x^{1-1/L}\notag\\ \omega^*(m)=k-1\\m\in\mathcal{R}'}}}\sum_{{\substack{x^{1/L}<p\le x/m\\p>P(m)\text{ prime}\\p\in\mathcal{R}',p=a^2+b^2,a<b}}} \sum_{{\substack{I\subseteq\mathcal{L}_{m,a,b,\mathcal{R}'}\\|I|=\ell}}}1\\
	&= \sum_{{\substack{m\le x^{1-1/L}\\ \omega^*(m)=k-1\\m\in\mathcal{R}'}}}\sum_{{\substack{I\subseteq\mathcal{L}_{m}\\|I|=\ell}}}\!\!\!\sum_{{\substack{x^{1/L}<p\le x/m\\p>P(m)\text{ prime}\\p\in\mathcal{R}',\ p=a^2+b^2,\ a<b}}}\!\!\!\!\!\! \mathbbm{1}(\phi(a,b)\in{\mathcal{R}'}\text{ for all }\phi\in I).
\end{align}
We now proceed by bounding \[\sum_{{\substack{x^{1/L}<p\le x/m\\p>P(m)\text{ prime}\\p\in\mathcal{R}',p=a^2+b^2,a<b}}}\!\!\!\!\! \mathbbm{1}(\phi(a,b)\in{\mathcal{R}'}\text{ for all }\phi\in I).\] For this, we first bound this quantity by the number of pairs $(a,b)$ such that $x^{1/L}<a^2+b^2\le\dfrac {x}{m}$ and $\displaystyle(a^2+b^2)\prod_{\phi\in I}\phi(a,b)\in\mathcal{R}'$. We may now use the Selberg sieve to estimate the number of such pairs by sieving out with respect to all primes that are $3\pmod 4$ up to $y=x^{1/10L}$.
\begin{claim}
	\[\sum_{{\substack{a^2+b^2\le x/m\\\gcd(a^2+b^2,\prod_{q\in \mathfrak{p}_y}q)=1\\\prod_{\phi\in I}\gcd(\phi(a,b),\prod_{q\in \mathfrak{p}_y}q)=1}}}\!\! 1
	\ll_\ell
	\frac x m\prod_{\substack{\ell+2<p\le y\\p\equiv 3\pmod 4}}\left(1-\frac{\ell}{p}\right)\prod_{p\mid T}\Bigl(1-\frac1p\Bigr)^{O(\ell)}+O\Bigl(\frac xm\frac{L^{O_\ell(1)}}{(\log x)^{\ell+1}}\Bigr),\]
	where $T=m\prod_{i<j}(u_iv_j-v_iu_j)(u_iu_j+v_iv_j)$ and $\mathfrak{p}_y$ is the set of primes congruent to $3$ modulo $4$ up to $y$.
\end{claim}
\begin{proof}
The setup is very similar to before. Define $(\mathcal{A},\mathcal{F},\mathbb{P})$ to be the discrete uniform probability space with $\displaystyle\mathcal{A}:=\Bigl\{(a,b): 1\le a,b\le\sqrt{\tfrac xm}\Bigr\}$, and $|\mathcal{A}|$ approximated by $X:=\dfrac xm$. We are again using the multivariate polynomial sieve with $F(a,b):=(a^2+b^2)\prod_{\phi\in I}\phi(a,b)$, and set $S=\{(a^2+b^2)\prod_{\phi\in I}\phi(a,b): 1\le a,b\le\sqrt{\tfrac xm}\}$, as previously. We now only sift out elements of $S$ which are divisible by primes congruent to $3$ modulo $4$ smaller than $y$, so $\mathfrak{p}$ is the set of primes that are $3$ mod $4$ and $z=y=x^{1/10L}$.

\begin{remark*}
As before, $p\le\ell+2$ are absorbed into the constant depending on $\ell$ that $\ll_\ell$ allows for. This includes $p=2$.
\end{remark*}

We must again estimate the multiplicative function $g(d)$ for the sieve.
In this case, the events $\mathcal{A}_p$ are now $\emptyset$ for $p\equiv 1\pmod 4$, and for $p\equiv 3\pmod 4$ we have $\mathcal{A}_p=\{(a,b)\in\mathcal{A}:p\mid (a^2+b^2)\prod_{\phi\in I}\phi(a,b)\}$. We get the same estimate for $g(p)$ as before in the case $p\equiv 3\pmod 4$. Thus:
\begin{lemma*}[Formula for $p^2g(p)$]\
	\begin{itemize}
		\item If $p\equiv 1\pmod 4$ then $g(p)=0$. Else:
		\item If $p\nmid T$ (in particular $p\nmid m$) then $p^2g(p)=1+\ell(p-1)$.
		\item If $p\mid T$ and $p\nmid m$ then $p^2g(p)=1+\ell_p(p-1)$, where $\ell_p$ is the number of linear forms $\phi\in I$ pairwise independent modulo $p$ ($\ell_p\le\ell$).
		\item If $p\mid m$ then $p\equiv 1\pmod 4$ so this case is no longer relevant.
	\end{itemize}
\end{lemma*}
The error terms $R_d=|\mathcal{A}_d|-g(d)X$ remain the same as before for squarefree products of primes that are $3$ modulo $4$, and as before, we get $R_d=\epsilon (X-D^2)$ for some $\epsilon\in[-1,1]$, so $|R_d|\le 2d\sqrt{\frac xm}$. We use Theorem~\ref{lemma: fundamental lemma of sieves}, noting that $\sum_{d\le y^2}3^{\omega(d)}|R_d|\ll\frac xm\frac{L^{O_\ell(1)}}{(\log x)^{\ell+1}}$:

\[\sum_{{\substack{a^2+b^2\le x/m\\\gcd(a^2+b^2,\prod_{q\le y}q)=1\\\prod_{\phi\in I}\gcd(\phi(a,b),\prod_{q\le y}q)=1}}} 1\le XG(y,y)+O\Bigl(\frac xm\frac{L^{O_\ell(1)}}{(\log x)^{\ell+1}}\Bigr).\]
We can use the same arguments as in Claim~\ref{claim: sieve theory bound on summand}, noting that \[\prod_{p\equiv 3\pmod 4}\left(1-\frac1{p^s}\right)\le\sum_{n\in\mathbb{N}}\frac1{n^s}=O(1)\] for $s>1$, and 
\[\prod_{p\equiv 3\pmod 4}\left(1-\frac{\ell+(\frac{-1}p)}{p^2}\right)\ll \ O_\ell(1)\!\prod_{\substack{p\ge(\ell+2)^2\\p\equiv 3\pmod 4}}\Bigl(1-\frac{1}{p^{3/2}}\Bigr)\sum_{n\in\mathbb{N}}\frac1{n^{3/2}}=O_\ell(1)\ll_\ell1.\] Hence we refine this to
\[\sum_{{\substack{a^2+b^2\le x/m\\\gcd(a^2+b^2,\prod_{q\in \mathfrak{p}_y}q)=1\\\prod_{\phi\in I}\gcd(\phi(a,b),\prod_{q\in \mathfrak{p}_y}q)=1}}}\!\!\!\!\!\!\!\!\!\!\!\! 1
\ll_\ell
\frac x m\prod_{\substack{\ell+2<p\le y\\p\equiv 3\!\!\!\!\pmod 4}}\left(1-\frac{\ell}{p}\right)\prod_{p\mid T}\Bigl(1-\frac1p\Bigr)^{O(\ell)}+O\Bigl(\frac xm\frac{L^{O_\ell(1)}}{(\log x)^{\ell+1}}\Bigr).\qedhere\]
\end{proof}
Claim~\ref{bound primes dividing T} already shows that 	$\displaystyle\prod_{p\mid T}\Bigl(1-\frac1p\Bigr)^{O(\ell)}\ll_\ell L^{O(1)}$, so we are left to bound \[\prod_{\substack{\ell+2<p\le y\\p\equiv 3\pmod 4}}\left(1-\frac{\ell}{p}\right)\asymp\prod_{\substack{\ell+2<p\le y\\p\equiv 3\pmod 4}}(1-\frac1p)^{\ell}\asymp\prod_{\substack{p\le y\\p\equiv 3\pmod 4}}\left(1-\frac{1}{p}\right).\]
Where previously we used Mertens' second theorem, now we use Mertens' theorem for Primes in APs~\ref{theorem: mertens APs}, and we again only require the special case discussed previously of modulo $4$, which gives us that \[\prod_{\substack{p\le y\\p\equiv 3\!\!\pmod 4}}\!\!\!\!\Bigl(1-\frac{1}{p}\Bigr)\le\exp\Bigl(-\!\!\!\!\!\!\!\!\sum_{\substack{p\le y\\p\equiv 3\!\!\pmod 4}}\!\!\!\!\!\!\frac1p\Bigr)=\exp\left(-\frac12\log\log y-M+O\Bigl(\frac1{\log y}\Bigr)\right)\ll\frac1{\sqrt{\log y}}.\]
Since $\log y=\log(x^{1/10L})=\frac1{10L}\log x$, we get

$\displaystyle\prod_{\substack{\ell+2<p\le y\\p\equiv 3\pmod 4}}\left(1-\frac{\ell}{p}\right)\ll\frac1{(\sqrt{\log y})^{\ell+1}}\asymp\frac{L^{O_\ell(1)}}{(\sqrt{\log x})^{\ell+1}}$.
This gives us

\[\sum_{{\substack{a^2+b^2\le x/m\\\gcd(a^2+b^2,\prod_{q\in \mathfrak{p}_y}q)=1\\\prod_{\phi\in I}\gcd(\phi(a,b),\prod_{q\in \mathfrak{p}_y}q)=1}}}\!\!\!\!\!\!\!\!\!\!\!\! 1
\ll_\ell
\frac x m\frac{L^{O_\ell(1)}}{(\sqrt{\log x})^{\ell+1}}.\] 
\begin{proof}[Proof of Theorem~\ref{moments of rR'}]
The proof is finished in the same way as Theorem~\ref{thm:r1cl upper bound}. We have 
\[\sum_{{\substack{n\le x\\ \omega^*(n)=k\\P(n)>x^{1/L}}}} \binom{r_{\mathcal{R}'}^*(n)}{\ell}\ll_\ell\frac{xL^{O_\ell(1)}}{(\sqrt{\log x})^{\ell+1}}\sum_{{\substack{m\le x^{1-1/L}\\ \omega^*(m)=k-1\\m\in\mathcal{R}'}}}\frac1m\cdot\#\{I\subseteq\mathcal{L}_{m}\colon |I|=\ell\},\]
 where $\mathcal{L}_m$ is the same set as before; hence we obtain the same result but with the $(\log x)^{\ell+1}$ in the denominator replaced by $(\sqrt{\log x})^{\ell+1}$.
\end{proof}
Note that when considering the second moment ($\ell=2$), this is again maximised when $k$ is approximately $2L$, in which case $\frac{xL^{O_\ell(1)}}{(\sqrt{\log x})^{\ell+1}}\frac{(2^{\ell-1}L)^k}{k!}$ is approximately $x\sqrt{\log x}L^{O(1)}$.
\section{Moments of $r_{\mathcal{R}}$}
As we have seen, many steps of the proof of Theorem~\ref{thm:r1cl upper bound} can be generalised to $r_S$, and the main change is the estimate we obtain on \[\sum_{{\substack{x^{1/L}<p\le x/m\\p>P(m)\text{ prime}\\p\in\mathcal{R}',\ p=a^2+b^2,\ a<b}}}\!\!\!\!\! \mathbbm{1}(\forall\phi\in I:\phi(a,b)\in{S})\] by sieve methods. In the cases $S=\mathbb{N}$ and $S=\mathcal{R}'$, we were able to use a regular Selberg sieve to sieve out all small primes in Theorem~\ref{thm:r1cl upper bound}, or only those congruent to $3$ modulo $4$ in Theorem~\ref{moments of rR'}. However, in the case $S=\mathcal{R}$, we have to use a slightly modified Selberg sieve, which does not sieve based on residue classes modulo primes but rather modulo prime powers.

We start with the discrete uniform probability space $(\mathcal{A},\mathcal{F},\mathbb{P})$ with $\mathcal{A}:=\bigl\{(a,b): 1\le a,b\le\sqrt{\tfrac xm}\bigr\}$, of size approximately $X:=\frac xm$. We still wish our sieve problem to sift elements of $S=\{(a^2+b^2)\prod_{\phi\in I}\phi(a,b): 1\le a,b\le\sqrt{\tfrac xm}\}$, but the events $\mathcal{A}_p$ are now $\emptyset$ for $p\equiv 1\pmod 4$, and for $p\equiv 3\pmod 4$ we have $\mathcal{A}_p=\{(a,b)\in\mathcal{A}:p^e\mathrel\Vert (a^2+b^2)\prod_{\phi\in I}\phi(a,b),\text{ some odd }e\}$.

\begin{remark*}
In fact, the events $\mathcal{A}_p$ above do not represent all possible ways of failing the conditions of $a^2+b^2\in{\mathcal{R}},\forall\phi\in I:\phi(a,b)\in{\mathcal{R}}$. For example, given $p\equiv 3\pmod 4$ we could have $p$ divides $\phi_1(a,b)$ and $\phi_2(a,b)$ to an odd power each, some $\phi_1,\phi_2\in I$, but $p$ divides $(a^2+b^2)\prod_{\phi\in I}\phi(a,b)$ to an even power. The sieve gives us an upper bound on our sum, which could possibly be further refined by considering these additional cases.
\end{remark*}
We will consider even smaller events $\mathcal{A}_p=\{p\mathrel\Vert (a^2+b^2)\prod_{\phi\in I}\phi(a,b)\}$, which give an even cruder upper bound.
We want to find a formula for $g(p),\ p\equiv 3\pmod 4$ so that $|\mathcal{A}_p|-Xg(p)$ is small, as previously. However, note that in the case $p\mid T$, there exist two linear forms in $I$ that are linearly dependent modulo $p$, so $p\mid(a^2+b^2)\prod_{\phi\in I}\phi(a,b)\iff p^2\mid(a^2+b^2)\prod_{\phi\in I}\phi(a,b)$. Hence we set $g(p)=0$ in this case.

If $p\nmid T$, we will count the number of residue classes $(a,b)\in\mathbb{Z}_{p^2}\times\mathbb{Z}_{p^2}$ satisfying $p\mathrel\Vert (a^2+b^2)\prod_{\phi\in I}\phi(a,b)$. Firstly, any pair $(pk_1,pk_2)$ can be excluded as this implies $p^2\mid(a^2+b^2)$. Note that all the linear forms are linearly independent modulo $p$ for $p\nmid T$, as discussed in Claim~\ref{claim: sieve theory bound on summand}. So if $p\nmid a,p\nmid b$, then $p\mathrel\Vert (a^2+b^2)\prod_{\phi\in I}\phi(a,b)\implies\exists\phi\in I:p\mathrel\Vert\phi(a,b)$. 

Thus, we now calculate the cardinality of $\{(a^*,b^*)\in\mathbb{Z}_{p^2}\times\mathbb{Z}_{p^2}: p\mathrel\Vert\phi_0(a^*,b^*)\}$ for a fixed $\phi_0=rx+sy\in I$. Suppose $s$ is invertible modulo $p$ -- we can always assume one of $r$, $s$ is not divisible by $p$ as they are coprime. For every $a'\pmod p$ there is a unique $b'\equiv -rs^{-1}a'\pmod p$ such that $p\mid\phi(a',b')$. Hence, for each $a^*\pmod{p^2}$ there are $p$ possible $b^*$ with $p\mid\phi(a^*,b^*)$, and out of those, a unique $b^*\equiv -rs^{-1}a^*\pmod{p^2}$ such that $p^2\mid\phi(a^*,b^*)$. Hence $\#\{(a^*,b^*)\in\mathbb{Z}_{p^2}\times\mathbb{Z}_{p^2}: p\mathrel\Vert\phi_0(a^*,b^*)\}=p^2(p-1)$.

However, these $p^2(p-1)$ residues include some of the form $(pk_1,pk_2)$. We calculate separately how many such residues have been included. Indeed, $p\mathrel\Vert\phi_0(pk_1,pk_2)\iff\phi_0(k_1,k_2)\not\equiv 0\pmod p$. We previously calculated $\#\{(k_1,k_2)\in\mathbb{Z}_p\times\mathbb{Z}_p: \phi_0(k_1,k_2)\equiv 0\pmod p\}=p$. Hence, for $p^2-p$ residues $(a^*,b^*)\in\mathbb{Z}_{p^2}\times\mathbb{Z}_{p^2}$ of the form $(pk_1,pk_2)$, we have $p\mathrel\Vert\phi_0(pk_1,pk_2)$. Thus $\#\{(a^*,b^*)\in\mathbb{Z}_{p^2}\times\mathbb{Z}_{p^2}\setminus p\mathbb{Z}_p\times p\mathbb{Z}_p: p\mathrel\Vert\phi_0(a^*,b^*)\}=p^2(p-1)-p(p-1)=p(p-1)^2$.
\begin{lemma*}[Formula for $g(p)$]\
	\begin{itemize}
		\item If $p\equiv 1\pmod 4$ or $p\mid T$ then $g(p)=0$.
		\item If $p\nmid T$ then $p^4g(p)=\ell\cdot p(p-1)^2$.	
	\end{itemize}
\end{lemma*}
We begin by estimating the error terms $R_d=|\mathcal{A}_d|-g(d)X$ for squarefree products of primes that are $3$ modulo $4$. We follow the same method as in Claim~\ref{claim: bound sieve errors}, but the calculation is slightly different as this time we worked modulo $p^2$.

Indeed, we define $\displaystyle D:=d^2\left\lfloor\Bigl\lfloor\sqrt{\tfrac xm}\Bigr\rfloor/d^2\right\rfloor$ to be the largest integer multiple of $d^2$ that is at most $\sqrt{\frac xm}$, and $\mathcal{A}':=\Bigl\{(a,b): 1\le a,b\le D\Bigr\}$ satisfies $|\mathcal{A}'_d|=g(d)D^2$. Thus $R_d=|\mathcal{A}_d|-|\mathcal{A}'_d|+g(d)D^2-g(d)X=|\mathcal{A}_d|-|\mathcal{A}'_d|+g(d)(D^2-X)$.

As before, $R_d=\epsilon (X-D^2)$ for some $\epsilon\in[-1,1]$. Since $0\le \sqrt{\frac xm}-D\le d^2$ we have $0\le X-D^2\le 2d^2\sqrt{\frac xm}$. Hence $|R_d|\le 2d^2\sqrt{\frac xm}$.

We now estimate $\sum_{d\le y^2}3^{\omega(d)}|R_d|$, where $\omega(d)\le \log_3(d)+1$.  Using our estimate of $|R_d|$ above gives a total error of $\sum_{d\le y^2}3d\cdot 2d^2\sqrt{\frac xm}=6\sqrt{\frac xm}\sum_{d\le y^2}d^3\ll \sqrt{\frac xm}(y^2)^4=\sqrt{\frac xm}x^{8/10L}\ll\frac xm\frac{L^{O_\ell(1)}}{(\log x)^{\ell+1}}$. This holds as $L$ is eventually large, and so $x^{8/10L}$ is of smaller order than $\sqrt{x}\frac{L^{O_\ell(1)}}{(\log x)^{\ell+1}}$. Hence we again get
\[\smash[t]{\sum_{{\substack{a^2+b^2\le x/m\\\gcd(a^2+b^2,\prod_{q\le y}q)=1\\\prod_{\phi\in I}\gcd(\phi(a,b),\prod_{q\le y}q)=1}}} 1\le XG(y,y)+O\Bigl(\frac xm\frac{L^{O_\ell(1)}}{(\log x)^{\ell+1}}\Bigr)},\]where 
$G(z,z)=O_\ell(1)\prod_{p<z}(1-g(p))^{-1}(1+O(\tfrac1{\log z}))$, and \[\prod_{p<y}(1-g(p))^{-1}=O_\ell(1)\prod_{\substack{\ell+2<p<y\\p\equiv 3\!\!\!\!\pmod 4\\ p\, \nmid\, T}}
\left(1-\frac{\ell\cdot(p-1)^2}{p^3}\right)^{-1}
.\]
Since $1-\dfrac{\ell\cdot(\frac{p-1}p)^2}{p}\ll\Bigl(1-\dfrac1p\Bigr)^{O(\ell)}$, we can rewrite the RHS above as  \[O_\ell(1)\prod_{\substack{\ell+2<p<y\\p\equiv 3\!\!\!\!\pmod 4}}
\left(1-\frac{\ell\cdot(p-1)^2}{p^3}\right)^{-1}\prod_{p\mid T}\Bigl(1-\frac1p\Bigr)^{O(\ell)}.\]
We have from Claim~\ref{bound primes dividing T} that	$\displaystyle\prod_{p\mid T}\Bigl(1-\frac1p\Bigr)^{O(\ell)}\ll_\ell L^{O(1)}$, and it remains to bound $\displaystyle\prod_{\substack{\ell+2<p<y\\p\equiv 3\!\!\pmod 4}}
\left(1-\frac{\ell\cdot(p-1)^2}{p^3}\right)^{-1}$. Now, $\displaystyle\Bigl(1-\frac{\ell\cdot(p-1)^2}{p^3}\Bigr)-\Bigl(1-\frac{\ell}{p}\Bigr)=\ell\frac{2p-1}{p^3}$ and $\displaystyle\prod_{p\equiv 3\pmod 4}\left(1-\ell\frac{2p-1}{p^3}\right)\ll O_\ell(1)\prod_{p\ge(\ell+2)^2}\Bigl(1-\frac{1}{p^{3/2}}\Bigr)\sum_{n\in\mathbb{N}}\frac1{n^{3/2}}=O_\ell(1)\ll_\ell1$, so
\[\prod_{p<y}(1-g(p))^{-1}\ll_\ell\prod_{\substack{\ell+2<p<y\\p\equiv 3\!\!\!\!\pmod 4}}
\left(1-\frac{\ell}{p}\right)^{-1}\prod_{p\mid T}\Bigl(1-\frac1p\Bigr)^{O(\ell)}.\] We have now obtained the same bound as in Theorem~\ref{moments of rR'}, hence we get
\begin{theorem}\label{moments of rR}
	For any fixed integer $\ell\ge 1$ and integer $k\ll L$ we have
	\[\smash{\sum_{{\substack{n\le x\\\omega^*(n)=k}}}\binom{r_{\mathcal{R}}^*(n)}{\ell}\ll_\ell\frac{xL^{O_\ell(1)}}{(\sqrt{\log x})^{\ell+1}}\frac{(2^{\ell-1}L)^k}{k!}.}\]
\end{theorem}
Although we used quite a crude sieve, the bound is of same order as for $r_{\mathcal{R}'}^*$, which is smaller than $r_{\mathcal{R}}^*$ everywhere. This suggests that no immediate improvement can be made on the result simply by defining the sieving criteria more carefully. In fact, this bound is possibly sharp in the case $\ell=2$.

Note that the theorems we have proved in this chapter are actually about moments of $r_S^*$ for $S=\mathcal{R}$, $S=\mathcal{R}'$, and not $r_S$ itself. This is because the proof technique of Claim~\ref{claim: approx r1 by r1*}, which relies on Lemma~\ref{lemma1}, cannot be reproduced for the functions $r_\mathcal{R}$ and $r_\mathcal{R}'.$
\chapter*{Conclusion}
In this thesis, we have seen results about the asymptotic behaviour of moments of representation numbers. We focused on $r_1(n)$, which counts the number of ways to represent $n$ as the sum of a square and the square of a prime, and derived our own results about the behaviour of $r_{\mathcal{R}}(n)$, which counts the number of ways to represent $n$ as the sum of a square and the square of a sum of two squares. In particular, the results about $r_1$ proved by Granville, Sabuncu and Sedunova, which we investigated in detail and used as a foundation for other estimates, are consistent with Stephan Daniel's theorem about second moments of $r_1$. This remarkable result has a very refined error bound, but it is difficult to generalise Daniel's highly technical proof to higher moments. It would be interesting to develop a proof similar to Daniel's about the second moment of $r_{\mathcal{R}}$ and compare it to the results proved in this thesis to see whether the upper bounds obtained are reasonable.

Given more time, it may be possible to further refine some of the estimates in the results we showed, and continue to investigate moments of $r_{S}$ for other interesting sets $S$. Granville, Sabuncu and Sedunova also worked on lower bounds, which I have not touched on in this paper. Although the lower bounds for moments of $r_1$ are conjecturally of the same order of magnitude, the current lower bounds which can be proved without further assumptions are quite far from the upper bounds. Hopefully more progress will be made in the future to bring the two closer.

\newpage
\printbibliography[heading=bibintoc] 
%
%
\end{document}